\documentclass[11pt]{amsart}
\usepackage{amsmath, amsthm, amssymb}
\usepackage{amscd}
\usepackage{mathabx}
\usepackage{hyperref}
\usepackage{mathrsfs}
\usepackage{xcolor, changepage} 
\usepackage{graphicx}
\usepackage{titletoc}
\usepackage{enumerate}
\usepackage{accents}
\usepackage{esint}
\usepackage{slashed}
\usepackage{tikz}
\usetikzlibrary{matrix,arrows}
\usetikzlibrary{shapes}
\usetikzlibrary{calc}
\usetikzlibrary{decorations.pathreplacing,decorations.markings,decorations.pathmorphing}
\usepackage[all]{xy}
\usepackage{tikz-cd}

\usepackage{caption}
\usepackage{subcaption}
\usepackage{geometry}
\theoremstyle{plain}
\newtheorem{theorem}{Theorem}
\newtheorem{proposition}[theorem]{Proposition} 
\newtheorem{lemma}[theorem]{Lemma}
\newtheorem{remark}[theorem]{Remark}
\newtheorem{corollary}[theorem]{Corollary}
\newtheorem{definition}[theorem]{Definition}

\numberwithin{theorem}{section}

\DeclareMathOperator{\dist}{dist}

\DeclareMathOperator{\supp}{supp}

\DeclareMathOperator{\Int}{Int}
\DeclareMathOperator{\divv}{div}

\DeclareMathOperator{\Ric}{Ric}
\DeclareMathOperator{\Capacity}{Cap}

\newcommand{\EADM}{\mathcal{M}}
\newcommand{\PADM}{\mathcal{P}}

\numberwithin{equation}{section}

\title[PMT for singular AH manifolds]{positive mass theorems for singular asymptotically hyperbolic manifolds }

\author{Yuguang Shi}
\address{Key Laboratory of Pure and Applied Mathematics,
	School of Mathematical Sciences, Peking University, Beijing, 100871, P. R. China
}
\email{ygshi@math.pku.edu.cn}

\author{Chengzhang Sun}

\address{Beijing International Center for Mathematical Research, Peking University, Beijing, 100871, P. R. China
}
\email{chengzhang@bicmr.pku.edu.cn}

\author{Zijun Wang} 

\address{Key Laboratory of Pure and Applied Mathematics,
	School of Mathematical Sciences, Peking University, Beijing, 100871, P. R. China
}

\email{2501110021@stu.pku.edu.cn}

\thanks{ Y. Shi  is funded by NSFC12431003. }


\begin{document}
\begin{abstract}
    Through a careful analysis of the Yamabe equation on singular spaces, we establish a positive mass theorem for singular asymptotically hyperbolic manifolds with arbitrary ends. We also derive a rigidity result that is novel even in the smooth setting.
\end{abstract}

\maketitle
\tableofcontents	
	
\section{Introduction}

\subsection{Singular asymptotically hyperbolic manifolds}	
	   
The positive mass theorem (PMT) is a foundational result in Riemannian  geometry and general relativity, characterizing the nonnegativity of the total mass of manifolds with nonnegative scalar curvature and standard asymptotic fall-off behavior. 

In recent years, extensive progress has been made on the nonexistence of metrics with positive scalar curvature (PSC) on singular spaces and the corresponding positive mass theorems. In \cite{LM2019, Kaz24, WX2024}, the authors established Geroch-type conjectures for certain smooth, uniformly Euclidean metrics that are smooth outside a codimension-$3$ submanifold. In \cite{DaSW2024}, the authors considered related problems for manifolds with isolated conical singularities. \cite{DWWW2024} established global rigidity in general settings using methods from the RCD theory. \cite{JSZ22} and \cite{CLZ22} established positive mass theorems for $C^0\cap W^{1,p}$ metrics. More recently, \cite{HSY26}(see also \cite{HSY25}) and \cite{BHHSZ2026} have dealt with PMT for higher-dimensional asymptotically flat (AF) manifolds with a singular set. In \cite{BQ08}, the authors obtained PMT for spin asymptotically hyperbolic (AH) manifolds with corners along compact hypersurfaces. Most recently, \cite{khuri2026} proves the positive mass theorem for AF manifolds with an $L^\infty$-metric when the codimension is slightly larger than 3, and \cite{BZ2026} proves the Geroch-type conjectures in similar settings. 

The primary motivation of this paper is to investigate which classes of singularities yield positive contributions to the PMT in the asymptotically hyperbolic setting. Obviously, the singularity arising in negative-mass AdS-Schwarzschild manifolds exhibits negative mass contributions. We note that singularities satisfying RCD conditions are those near which both the Ricci curvature and, consequently, the scalar curvature are bounded below in a suitable sense. Intuitively, such singularities should yield positive mass contributions. Gromov, in his four lectures \cite{Gro23}, has already proposed such questions for Alexandrov spaces. We prove that PMT holds for low-dimensional singularities under the RCD condition, or when the Ricci curvature has a quadratic lower bound near the singular set (cf.\ Theorem \ref{thm:pmt AH}). Clarifying the sign contribution of different singularities will deepen our understanding of how local singular geometry affects global mass invariants on AH manifolds. 

\begin{definition}[Definition 2.3 in \cite{Wang2001}]\label{def:AH mfd}
    A complete noncompact Riemannian manifold \((M^n, g)\) is said to be asymptotically hyperbolic (AH) if there exists a compact manifold \(X\) with boundary and a smooth function \(\rho\) on \(X\) such that:
    \begin{enumerate}
        \item[(i)] \(M^n\) is diffeomorphic to \(X \setminus \partial X\) (we identify \(M^n\) with \(X \setminus \partial X\) in the sequel),
        \item[(ii)] \(\rho = 0\) on \(\partial X\) and \(\rho > 0\) on \(X \setminus \partial X\),
        \item[(iii)] \(\bar{g} = \rho^2 g\) extends to a smooth metric on \(X\),
        \item[(iv)] \(|d\rho|_{\bar{g}} = 1\) on \(\partial X\),
        \item[(v)] each component \(\Sigma\) of \(\partial X\) is the standard round sphere \((\mathbb{S}^{n-1}, \gamma_{\mathrm{std}})\), and in a collar neighborhood of \(\Sigma\),
        \[
        g = \sinh^{-2} \rho \, (d\rho^2 + \gamma_\rho)
        \]
        with \(\gamma_\rho\) being a \(\rho\)-dependent family of metrics on \(\mathbb{S}^{n-1}\) that has the expansion
        \begin{equation}\label{eq:AH expansion}
            \gamma_\rho = \gamma_{\mathrm{std}} + \frac{\rho^n}{n} h + O(\rho^{n+1}),
        \end{equation}
        where \(h\) is a symmetric 2-tensor on \(\mathbb{S}^{n-1}\), and the term $O(\rho^{n+1})$ is a function in $C^2_{n+1}$. 
    \end{enumerate}
\end{definition}
We use the conventions in \cite{ACF92} for the weighted H\"older spaces; see also Section \ref{sec:smoothness}. Although Wang's asymptotics are more restrictive, we use them to simplify our discussion. The condition (iii) is necessary for applying tangential smoothness of solutions to Yamabe equations from \cite{ACF92}, and we do not intend to decide the weakest regularity condition on $\bar{g}$. In particular, the term $O(\rho^{n+1})$ in fact belongs to $C^\infty_{n+1}$ assuming (iii).

We are interested in the following class of singular spaces: 

\begin{definition}[Section 2 in \cite{BHHSZ2026}]
    A metric measure space $(M^n,d,\mu)$ is called an almost-manifold if
    \begin{itemize}
    \item[(AM)] There is a positive integer $n\geq 3$ such that 
    \begin{itemize}
        \item $(M^n,d,\mu)$ is {\it locally Ahlfors $n$-regular}: for any bounded subset $U$ of $M^n$ there are constants $\varepsilon=\varepsilon(U)$ and $C=C(U)$ such that
        $$C^{-1}r^n\leq\mu(B_r(x))\leq Cr^{n}$$
        for any $x\in U$ and $r\in (0,\epsilon)$;
        \item $(M^n,d,\mu)$ admits a {\it regular-singular decomposition} 
        $$ M^n=\mathcal R\sqcup\mathcal S,$$ 
       where $\mathcal R$ is a {\it connected smooth} manifold of dimension $n$ equipped with a Riemannian metric $g$ such that $d_L=d_g$ and $\mu=\mu_g$ and $\mathcal S$ is a compact  subset of $M^n$ with $\mathcal H^{n-2}(\mathcal S)=0$, where $d_L$ denotes the length metric associated to $d$, and $d_g$ and $\mu_g$ are the metric and volume measure induced by $g$ respectively. Note that we have $\mu(\mathcal S)=0$ in particular. We assume  that all AH ends described above are contained in $\mathcal R$. 
       \item $\mathcal{S}=\bigcup^m_{i=1} \mathcal{S}_i$, where each $\mathcal{S}_i$ is a compact set of $M^n$. There exist $\Lambda>0$, $\delta_0>0$ such that  for each $i$ there is a $0\leq k_i\leq n-2$ with $\Lambda^{-1}\leq \mathcal{H}^{k_i}(\mathcal{S}_i)\leq \Lambda$, and   for any $x\in \mathcal{S}_i$, $r\leq \delta_0$ there holds $\Lambda^{-1}r^{k_i}\leq \mathcal{H}^{k_i}(\mathcal{S}_i\cap B(x,r))	\leq \Lambda r^{k_i}$. 
     \item[(LPI)]  {\it Local $(1,2)$-Poincar\'e Inequality}: for any bounded subset $U$ of $M$ there are constants $\varepsilon=\varepsilon(U)>0$,  $\lambda=\lambda(U)\geq 1$ and $C_P=C_P(U)>0$ such that we have
    $$
    \fint_{B}|\phi-\phi_B|\,\mathrm d\mu\leq C_Pr\left(\fint_{\mathcal R\cap\lambda B}|\nabla_g \phi|^2\,\mathrm d\mu\right)^{\frac{1}{2}}\mbox{ for any }\phi\in W^{1,2}_{\mathrm{loc}}(M),
    $$
    where $B$ and $\lambda B$ denote $B_r(x)$ and $B_{\lambda r}(x)$ for any $x\in U$ and $r\in (0,\varepsilon)$ respectively, and $\phi_B$ denotes the average value of $\phi$ in $B$.
    \end{itemize}
\end{itemize}
\end{definition}

Owing to Lemma 2.2 and Lemma 2.4 in \cite{BHHSZ2026}, we know that the Poincar\'e inequality and the Sobolev inequality hold in any neighborhood of $\mathcal{S}$. Using these inequalities, one can construct a positive singular harmonic function with singular set either $\{p\}$, for any $p\in\mathcal S$, or the whole set $\mathcal S$; see \cite{HSY25,BHHSZ2026}.

We will assume throughout this article that $M$ is connected, and that there exists a compact set $K$ such that $\mathcal{S}\subset K\subset M$, and that 
$$
    M\setminus K=E_0\cup\bigcup_{i\in I}E_i
$$
is a decomposition into connected components, with $E_0$ isometric to an open neighborhood of $\partial X$ of some AH manifold as in Definition \ref{def:AH mfd}, and $E_i$ isometric to a complete non-compact Riemannian manifold with a compact subset removed. By Lemma 2.17 in \cite{BHHSZ2026}, we can always assume that $K$ is connected. We will refer to $E_0$ as an AH end of $M$, with conformal infinity $\partial X$, and $E_i$ as an arbitrary end of $M$. The set $I$ is allowed to be empty, in which case there is only a single AH end $E_0$. 

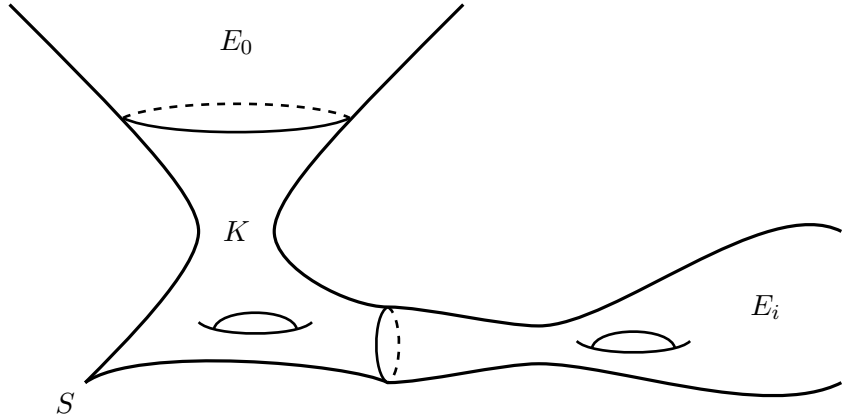
\begin{figure}[ht]
\centering
\begin{tikzpicture}
    \coordinate (E01) at (-3,3);
    \coordinate (E02) at (3,3);
    \coordinate (Ei1) at (8,0);
    \coordinate (Ei2) at (8,-2);
    \coordinate (K1) at (-1.5,1.5);
    \coordinate (K2) at (1.5,1.5);
    \coordinate (K3) at (2,-1);
    \coordinate (K4) at (2,-2);
    \coordinate (S) at (-2,-2);

    \draw[line width=1.2pt]
    (E01)
    .. controls (K1) and (-0.5,0.5) .. (-0.5,0)
    .. controls (-0.5,-0.5) and (-1.5,-1.5) .. (S);

    \draw[line width=1.2pt]
    (E02)
    .. controls (K2) and (0.5,0.5) .. (0.5,0)
    .. controls (0.5,-0.5) and (1.5,-1) .. (K3)
    .. controls (2.5,-1) and (3.5,-1.25) .. (4,-1.25)
    .. controls (5,-1.25) and (7,0.5) .. (Ei1);

    \draw[line width=1.2pt]
    (S)
    .. controls (-1.3,-1.5) and (1.5,-1.75) .. (K4)
    .. controls (2.5,-2) and (3.5,-1.75) .. (4,-1.75)
    .. controls (5,-1.75) and (7,-2.5) .. (Ei2);

    \draw[dashed,line width=1pt]
    (K1)
    .. controls (-1,1.75) and (1,1.75) .. (K2);
    \draw[line width=1pt]
    (K1)
    .. controls (-1,1.25) and (1,1.25) .. (K2);

    \draw[dashed,line width=1pt]
    (K3)
    .. controls (2.2,-1.1) and (2.2,-1.9) .. (K4);
    \draw[line width=1pt]
    (K3)
    .. controls (1.8,-1.1) and (1.8,-1.9) .. (K4);

    \draw[line width=1pt]
    (-0.5,-1.2) .. controls (-0.3,-1.4) and (0.8,-1.4) .. (1,-1.2);
    \draw[line width=1pt]
    (-0.3,-1.3) .. controls (-0.25,-1) and (0.75,-1) .. (0.8,-1.3);
    \draw[line width=1pt]
    (4.5,-1.45) .. controls (4.7,-1.65) and (5.8,-1.65) .. (6,-1.45);
    \draw[line width=1pt]
    (4.7,-1.55) .. controls (4.75,-1.25) and (5.75,-1.25) .. (5.8,-1.55);

    \node at (0,2.5) {$E_0$};
    \node at (0,0) {$K$};
    \node[below left] at (S) {$S$};
    \node at (7,-1) {$E_i$};

\end{tikzpicture}
\caption{The AH and arbitrary ends.}
\label{fig:ends}
\end{figure}

\begin{theorem}\label{thm:pmt AH}
Let $(M^n, d, \mu)$ be an almost manifold with its regular part $\mathcal R$ containing an AH end $E_0$ and possibly some arbitrary ends. Suppose its scalar curvature $R_g\geq -n(n-1)$ on $\mathcal R$ and $\dim_{H}(\mathcal{S})< \frac{n-2}{2}$. Moreover, we assume that either 
\begin{itemize}
    \item there exists a neighborhood $U$ of $\mathcal{S}$ such that the Ricci curvature 
    $$
        \Ric_g (x) \geq -C d^{-2}(x, \mathcal{S})
    $$
    for all $x\in U\cap \mathcal{R}$, or
    \item $(M^n, d, \mu)$ is an $RCD(K,N)$ space.
\end{itemize}
Here, $C, K, N$ are universal constants that depend only on $(M^n,d,\mu)$. Then we have 

\begin{equation}\label{eq:mass nonnegative}
    \int_{S^{n-1}} \operatorname{tr}_{\gamma_{\mathrm{std}}}(h) \,d \mu_{\gamma_{\mathrm{std}}} \geq\left|\int_{S^{n-1}} \operatorname{tr}_{\gamma_{\mathrm{std}}}(h) x \,d \mu_{\gamma_{\mathrm{std}}}\right|.
\end{equation}

The equality holds only if $R_g=-n(n-1)$ on $\mathcal{R}$. If $\mathcal{R}$ is spin, then the equality implies that $\mathcal{R}$ is Einstein. If $\mathcal{S}=\varnothing$ and if $\int_M\|\kappa\|^2_g\,d\mu_g<\infty$ where $\kappa:=\Ric_g+(n-1)g$, the equality implies that $M$ is isometric to the hyperbolic space $\mathbb{H}^n$. 
\end{theorem}

\begin{remark}
Note that for manifolds with only AH ends, the condition $\int_M\|\kappa\|^2_g\,d\mu_g<\infty$ holds automatically.
\end{remark}

In this article, we use $d\mu_g$ to represent the volume form of any Riemannian metric $g$. We will denote 
$$\EADM(g)=\int_{S^{n-1}} \operatorname{tr}_{\gamma_{\mathrm{std}}}(h) \,d \mu_{\gamma_{\mathrm{std}}}, \quad \PADM(g)=\int_{S^{n-1}} \operatorname{tr}_{\gamma_{\mathrm{std}}}(h) x \,d \mu_{\gamma_{\mathrm{std}}}.$$

To prove Theorem \ref{thm:pmt AH}, we choose a sequence of decreasing domains of $\mathcal{S}$ with smooth boundaries such that $U_1 \supset U_2\supset\cdots\supset U_k\supset\cdots \supset\mathcal{S}$ (see \eqref{eq:neighborhood} below). Following the arguments in \cite{AMcpt88}, we conformally deform the metric $g$ on $M\setminus \bar U_j$. Unfortunately, we find it quite difficult to control the solution of the Yamabe equation \eqref{eq:Yamabe} on the arbitrary ends. Hence, we cannot ensure the completeness of the deformed metric. To resolve this issue, we employ gluing arguments (cf.\ Proposition \ref{prop:PMT strict} and Lemma \ref{lem:BQ conformal change} below). More specifically, let $M'\subset M$ be a neighborhood of the AH end $E_0$ and the compact set $K$; we assume that the closure of $U_j$ is contained in $K$ for all $j>0$, and $\Sigma_i=\partial M'\cap E_i$ is smooth. We consider the following Yamabe equation:
 \begin{equation}\label{eq:yamabe eq0}
        \begin{cases}
            L_g u_j:=-c_n\Delta_g u_j+R_g u_j-fu_j^{\alpha} =0 & \text{in }M'\setminus \bar U_j, \\
            u_j=1 & \text{on }\partial M',\\
            u_j=\infty & \text{on }\partial U_j,  \\
            u_j=1 & \text{on }\partial X.
        \end{cases}
    \end{equation} 
Here, $c_n:=\frac{4(n-1)}{n-2}$, $\alpha:=\frac{n+2}{n-2}$, and $f\in C^\infty(M\setminus \mathcal {S})$ satisfies 
$-n(n-1)\le f\le \min\{R_g, -1\}$. 

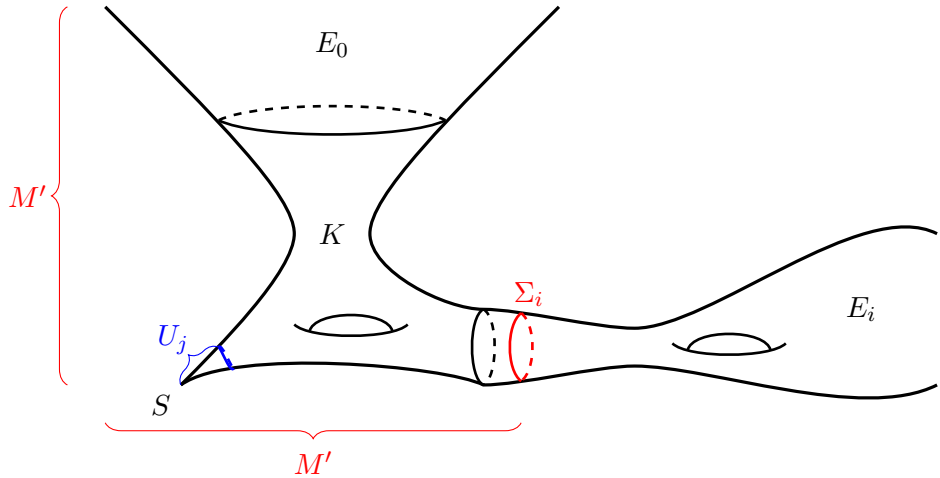
\begin{figure}[ht]
\centering
\begin{tikzpicture}
    \coordinate (E01) at (-3,3);
    \coordinate (E02) at (3,3);
    \coordinate (Ei1) at (8,0);
    \coordinate (Ei2) at (8,-2);
    \coordinate (K1) at (-1.5,1.5);
    \coordinate (K2) at (1.5,1.5);
    \coordinate (K3) at (2,-1);
    \coordinate (K4) at (2,-2);
    \coordinate (S) at (-2,-2);

    \draw[line width=1.2pt]
    (E01)
    .. controls (K1) and (-0.5,0.5) .. (-0.5,0)
    .. controls (-0.5,-0.5) and (-1.5,-1.5) .. (S);

    \draw[line width=1.2pt]
    (E02)
    .. controls (K2) and (0.5,0.5) .. (0.5,0)
    .. controls (0.5,-0.5) and (1.5,-1) .. (K3)
    .. controls (2.5,-1) and (3.5,-1.25) .. (4,-1.25)
    .. controls (5,-1.25) and (7,0.5) .. (Ei1);

    \draw[line width=1.2pt]
    (S)
    .. controls (-1.3,-1.5) and (1.5,-1.75) .. (K4)
    .. controls (2.5,-2) and (3.5,-1.75) .. (4,-1.75)
    .. controls (5,-1.75) and (7,-2.5) .. (Ei2);

    \draw[dashed,line width=1pt]
    (K1)
    .. controls (-1,1.75) and (1,1.75) .. (K2);
    \draw[line width=1pt]
    (K1)
    .. controls (-1,1.25) and (1,1.25) .. (K2);

    \draw[dashed,line width=1pt]
    (K3)
    .. controls (2.2,-1.1) and (2.2,-1.9) .. (K4);
    \draw[line width=1pt]
    (K3)
    .. controls (1.8,-1.1) and (1.8,-1.9) .. (K4);

    \draw[dashed,color=red,line width=1pt]
    (2.5,-1.05)
    .. controls (2.7,-1.2) and (2.7,-1.8) .. (2.5,-1.95);
    \draw[color=red,line width=1pt]
    (2.5,-1.05)
    .. controls (2.3,-1.2) and (2.3,-1.8) .. (2.5,-1.95);

    \draw[dashed,color=blue,line width=1pt]
    (-1.5,-1.5)
    .. controls (-1.45,-1.45) and (-1.35,-1.75) .. (-1.3,-1.78);
    \draw[color=blue,line width=1pt]
    (-1.5,-1.5)
    .. controls (-1.45,-1.55) and (-1.35,-1.85) .. (-1.3,-1.78);

    \draw[line width=1pt]
    (-0.5,-1.2) .. controls (-0.3,-1.4) and (0.8,-1.4) .. (1,-1.2);
    \draw[line width=1pt]
    (-0.3,-1.3) .. controls (-0.25,-1) and (0.75,-1) .. (0.8,-1.3);
    \draw[line width=1pt]
    (4.5,-1.45) .. controls (4.7,-1.65) and (5.8,-1.65) .. (6,-1.45);
    \draw[line width=1pt]
    (4.7,-1.55) .. controls (4.75,-1.25) and (5.75,-1.25) .. (5.8,-1.55);

    \node at (0,2.5) {$E_0$};
    \node at (0,0) {$K$};
    \node[below left] at (S) {$S$};
    \node at (7,-1) {$E_i$};

    \draw[decorate,decoration={brace,amplitude=6pt},color=blue]
    (-2,-2) -- (-1.5,-1.5) node [midway,xshift=-10pt,yshift=10pt] {$U_j$};
    \draw[decorate,decoration={brace,amplitude=6pt},color=red]
    (-3.5,-2) -- (-3.5,3) node [midway,xshift=-15pt] {$M'$};
    \draw[decorate,decoration={brace,amplitude=6pt},color=red]
    (2.5,-2.5) -- (-3,-2.5) node [midway,yshift=-15pt] {$M'$};
    \node at (2.6,-0.8) [color=red] {$\Sigma_i$};

\end{tikzpicture}
\caption{The domain $M'$.}
\label{fig:cutoff}
\end{figure}

Then the resulting metric $\tilde g_j:=u_j^{\frac{4}{n-2}}g$ is complete on $M'$ and its scalar curvature satisfies $R_{g_j}\geq -n(n-1)$. Owing to the assumption that $\dim_H(\mathcal S)<\frac{n-2}{2}$, we observe that the limit of $u_j$, denoted by $u_\infty$, is bounded above by 1 (see Proposition \ref{prop:bound for u_infty} below). The key to obtaining $u_\infty\leq 1$ is an analogue of Theorem 2.1 in \cite{Aviles82} for the Yamabe equation \eqref{eq:Yamabe} on the almost manifold $(M, d, \mu)$ (see Propositions \ref{prop:bound2} and \ref{prop:upper bound RCD} below). 

Hence, by the strong maximum principle, we have either $u_\infty\equiv 1$, or $u_\infty<1$ and 
$$
    \frac{\partial u_\infty}{\partial v}>0 ~~\text {on} ~~\partial M'.
$$
In the latter case, for sufficiently large $j$, we have
$$
    \frac{\partial u_j}{\partial v}>0 ~~\text {on} ~~\partial M',
$$
which implies that the mean curvature of $\partial M'$ with respect to $\tilde g_j$, taken with the outward normal vector, is less than that with respect to $g$ for sufficiently large $j$. For details, see Lemmas \ref{lem:normal derivative} and \ref{lem:mean curvature}. Then we use the argument from \cite{BQ08} to obtain a smooth metric $\tilde{g}_\epsilon$ on $M\setminus\bar{U}_j$, which is $C^0$-close to $\tilde{g}_j$. Overall, the conformal factor from $g$ to $\tilde{g}_\epsilon$ is less than $1$. Hence, the mass of $\tilde{g}_\epsilon$ is less than that of $g$, and \eqref{eq:mass nonnegative} follows from the PMT for a smooth AH manifold, which has been established recently \cite{BW2026,hirsch2026,tsang2026} for all dimensions greater than or equal to 3. See Propositions \ref{prop:PMT strict} and \ref{prop:PMT R<0} for details. 

\begin{figure}[ht]
\centering
\begin{tikzpicture}
    \coordinate (E01) at (-3,3);
    \coordinate (E02) at (3,3);
    \coordinate (Ei1) at (8,0);
    \coordinate (Ei2) at (8,-2);
    \coordinate (K1) at (-1.5,1.5);
    \coordinate (K2) at (1.5,1.5);
    \coordinate (K3) at (2,-1);
    \coordinate (K4) at (2,-2);
    \coordinate (S1) at (-3,-2);
    \coordinate (S2) at (-2,-3);

    \draw[line width=1.2pt]
    (E01)
    .. controls (K1) and (-0.7,0.5) .. (-0.7,0)
    .. controls (-0.7,-0.5) and (-1.5,-1.5) .. (S1);

    \draw[line width=1.2pt]
    (E02)
    .. controls (K2) and (0.7,0.5) .. (0.7,0)
    .. controls (0.7,-0.5) and (1.5,-1) .. (K3)
    .. controls (2.5,-1) and (3.5,-1.25) .. (4,-1.25)
    .. controls (5,-1.25) and (7,0.5) .. (Ei1);

    \draw[line width=1.2pt]
    (S2)
    .. controls (-1.3,-1.5) and (1.8,-2) .. (K4)
    .. controls (2.5,-2) and (3.5,-1.75) .. (4,-1.75)
    .. controls (5,-1.75) and (7,-2.5) .. (Ei2);

    \draw[dashed,line width=1pt]
    (K1)
    .. controls (-1,1.75) and (1,1.75) .. (K2);
    \draw[line width=1pt]
    (K1)
    .. controls (-1,1.25) and (1,1.25) .. (K2);

    \draw[dashed,line width=1pt]
    (K3)
    .. controls (2.2,-1.1) and (2.2,-1.9) .. (K4);
    \draw[line width=1pt]
    (K3)
    .. controls (1.8,-1.1) and (1.8,-1.9) .. (K4);

    \draw[dashed,color=red,line width=1pt]
    (2.5,-1.05)
    .. controls (2.7,-1.2) and (2.7,-1.8) .. (2.5,-1.95);
    \draw[color=red,line width=1pt]
    (2.5,-1.05)
    .. controls (2.3,-1.2) and (2.3,-1.8) .. (2.5,-1.95);

    \draw[line width=1pt]
    (-0.5,-1.2) .. controls (-0.3,-1.4) and (0.8,-1.4) .. (1,-1.2);
    \draw[line width=1pt]
    (-0.3,-1.3) .. controls (-0.25,-1) and (0.75,-1) .. (0.8,-1.3);
    \draw[line width=1pt]
    (4.5,-1.45) .. controls (4.7,-1.65) and (5.8,-1.65) .. (6,-1.45);
    \draw[line width=1pt]
    (4.7,-1.55) .. controls (4.75,-1.25) and (5.75,-1.25) .. (5.8,-1.55);

    \node at (0,2.5) {$E_0$};
    \node at (0,0) {$K\setminus\bar U_j$};
    \node at (7,-1) {$E_i$};
    \node at (2.6,-0.7) [color=red] {$\partial M'$};

\end{tikzpicture}
\caption{The blown-up metric $\tilde{g}_j$ can be smoothed out near $\partial M'$.}
\end{figure}
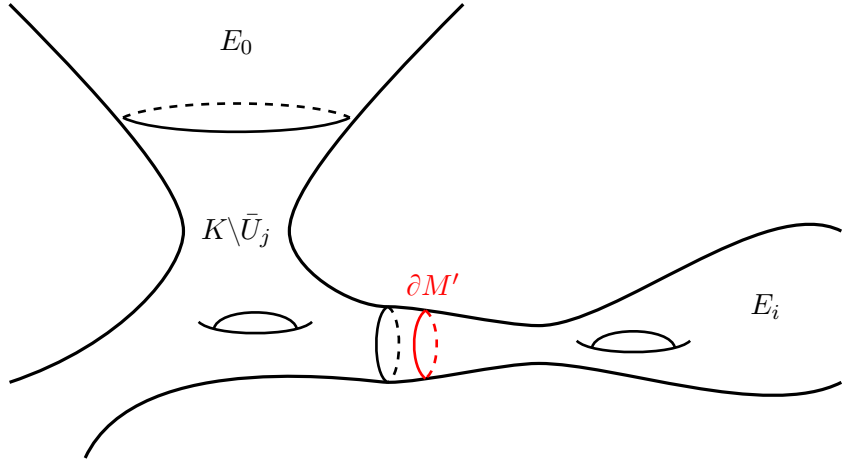

In the former case, i.e., $u_\infty\equiv 1$, we use a slightly different approach, namely, we can solve the Yamabe equation on $M\setminus \bar U_j$ (instead of $M'\setminus \bar U_j$). There is no need to smooth out the metric $\tilde g_j$ and the argument is easier, cf.\ Proposition \ref{prop:PMT R<0}.

The rigidity part of PMT for AH manifolds is more subtle than that of the AF case. The case of smooth AH manifolds without arbitrary ends was addressed in \cite{Huang2020, Hir2025}. 

We first show that the scalar curvature $R=-n(n-1)$ if the equality in \eqref{eq:mass nonnegative} holds, which is a direct consequence of our blow-up argument. Suppose that there is a $p\in \mathcal R$ with $R(p)>-n(n-1)$. We know that $u_\infty <1$ in $\mathcal R$. Thanks to Proposition \ref{prop:gap}, we see that 
$$
    u_\infty<1-\delta\rho^n
$$
for some $\delta>0$. This is crucial for achieving strict mass decay, and it violates the PMT, cf.\ Proposition \ref{prop:PMT strict}. 

To prove the Ricci rigidity when $\mathcal{S}=\varnothing$, we impose
the $L^2$-condition $\int_M\|\kappa\|_g^2\,d\mu_g<\infty$
on the traceless Ricci tensor $\kappa:=\Ric_g+(n-1)g$.
We perturb the metric by $g_t:=g+t\eta\kappa$, where $\eta$ is a smooth
cut-off supported in a large compact set and equal to $1$ in a neighborhood of the AH end and a point with $\kappa\neq0$.  Solving the Yamabe equation on an exhaustion yields
global solutions $u_t$; letting $t\nearrow0$, the linearization
$v=\lim v_t$ with $u_t=1+tv_t$ decomposes as $v=v_1+v_2$.  The first part
$v_1$ is strictly positive near the AH end by a comparison argument, while
an $L^2$-energy estimate driven by the integrability of $\kappa$ shows
that $v_2$ is negligible when the cut-off region is made sufficiently
large.  The positivity of $v$ then forces the mass to drop for $t<0$,
contradicting the positive mass theorem unless $\kappa\equiv0$, i.e.\ 
$\Ric_g=-(n-1)g$, cf.\ Proposition \ref{prop:ricci-rigidity}.

In the spin case, we show that $\mathcal R$ is Einstein by constructing a Killing spinor. Roughly speaking, on each blown-up manifold $(M\setminus\bar U_j,\tilde{g}_j)$ there is an ``almost-Killing'' spinor, and it can be transplanted to an almost-Killing spinor on an open subset of $\mathcal{R}$. By letting $j\to\infty$, a subsequence converges to a Killing spinor on $\mathcal{R}$. 

Lastly, when the singular set is absent, we can use an argument of Anderson \cite{anderson03} to show that the manifold, which is proved to be Einstein, is indeed hyperbolic. 

\subsection{Organization of the Paper}

The remainder of this paper is organized as follows. In Section \ref{sec:Yamabe eqn}, we review and establish some basic properties of the Yamabe equation of interest to us. From these properties, we obtain a smooth solution $u_\infty$ of the Yamabe equation on $\mathcal{R}$ in Section \ref{sec:Yamabe eqn reg part}. Then we prove that $u_\infty$ is bounded in Section \ref{sec:bound under Ricci cond} and Section \ref{sec:bound under RCD cond}. In Section \ref{sec:PMT proof}, we prove the positive mass theorem and establish the scalar curvature rigidity. In Section \ref{sec:Ricci rigid smooth}, we prove the Ricci rigidity when $\mathcal S=\varnothing$ and $\kappa\in L^2$, and in Section \ref{sec:Ricci rigid spin}, we prove the Ricci rigidity when $\mathcal R$ is spin. In Section \ref{sec:hyperbolic}, we prove that, when $\mathcal S=\varnothing$, the equality holds if and only if $M$ is isometric to $\mathbb{H}^n$. 

\bigskip
\section*{Acknowledgements}
We thank Lan-Hsuan Huang and Gang Li for their interest in this work, and especially thank Gang Li for pointing us to the paper \cite{anderson03}.

AI usage: After the second draft, the authors consulted GPT-6 (in combination with Danus to review the previous version of the manuscript) to correct typos and improve parts of the exposition and proofs. The main results remain unchanged from the second draft. All statements are proposed by the authors without AI, and the authors are responsible for any mistakes that may occur.

\section{The Yamabe equation}\label{sec:Yamabe eqn}

We establish some basic properties of the following Yamabe equation, most of which are well-known, cf.\ \cite{AMcpt88, AMnoncpt88, ACF92}. 

\subsection{Local estimates}\label{sec:loc est}

Let $(M^n,g)$ be a Riemannian manifold (without boundary) and $n\ge 3$. Set $c_n=\frac{4(n-1)}{n-2}$, $\alpha=\frac{n+2}{n-2}$. Let $f\in C^\infty(M)$ satisfy $f\le -1$. Consider the equation: 
\begin{equation}\label{eq:Yamabe}
    -c_n\Delta_gu+R_gu=fu^{\alpha},
\end{equation}
For convenience, we write 
\begin{equation}\label{eq:operator}
    L_gu=-c_n\Delta_gu+R_gu-fu^{\alpha}.
\end{equation}

We have the a priori estimate for sub-solutions of the equation. 

\begin{lemma}
    Suppose that $K$ is a compact subset of a domain $\Omega$. Then there exists a constant $C$ depending on $K$ and $\Omega$ such that if a function $u\in H^1_{\mathrm{loc}}(\Omega)$ with $u\ge 0$ in $\Omega$ satisfies
    $$
        L_gu\le 0
    $$
    in the weak sense, then 
    $$
        \max_{K}u\le C.
    $$
\end{lemma}
\begin{proof}
    This follows from the same argument as in \cite[Theorem 1.1]{AMnoncpt88}. 
\end{proof}

\begin{corollary}\label{cor:convergence}
    Suppose $\{K_i\}$ is a compact exhaustion of $M$. If $u_i\in H^1_{\mathrm{loc}}(K_i)$, $u_i\ge 0$ and $u_i$ solves the equation \eqref{eq:Yamabe} on a neighborhood of $K_i$ for each $i$, then  $u_i$ is smooth, and there is a subsequence of $\{u_i\}$ that converges smoothly on compact sets to a solution to \eqref{eq:Yamabe} on $M$.
\end{corollary}
\begin{proof}
    Since $u_i$ is bounded in a neighborhood of $K_i$, the standard elliptic estimate shows that $u_i$ is smooth. The local boundedness of $u_i$ also implies that there is a subsequence of $\{u_i\}$ that converges smoothly on compact sets to a solution to \eqref{eq:Yamabe} on $M$.
\end{proof}

The variational argument solves the Dirichlet boundary value problem. 
\begin{lemma}\label{lem:variation}
    Let $\Omega$ be a connected, precompact open set in $M$ with smooth boundary. For any smooth positive function $u_0$ on $\partial\Omega$, the Dirichlet problem
    \begin{equation}
        \begin{cases}
            L_gu=0 & \text{on }\Omega, \\
            u=u_0 & \text{on }\partial\Omega,
        \end{cases}
    \end{equation}
    has a unique positive solution. Moreover, the solution $u$ is smooth on $\bar\Omega$. 
\end{lemma}
\begin{proof}
    Consider the energy functional 
    $$
        E(u)=\int_{\Omega} c_n|\nabla u|^2+R_gu^2-\frac{(n-2)f}{n}(u^2)^{\frac{n}{n-2}}.
    $$
    Since $R_g$ is bounded on $\bar\Omega$, $-f\ge1$, and $\frac{2n}{(n-2)}>2$, Young\textquotesingle s inequality gives coercivity. The positive gradient and highest-power terms are weakly lower semicontinuous, while the quadratic term is continuous under strong $L^2$ convergence. Thus, the direct method gives a minimizer on the affine space
    $$
        A=\{u\in H^1(\Omega):\ u|_{\partial\Omega}=u_0\}.
    $$
    A minimizer may be replaced by its absolute value, so we take $u\ge0$. It is a weak solution; testing against a sufficiently large constant gives boundedness, and elliptic regularity then gives smoothness up to $\partial\Omega$. The strong maximum principle, as in Lemma \ref{lem:sMaxPr-1}, and the positive boundary data imply $u>0$.

    To prove the uniqueness, suppose $u_1,u_2$ are two positive solutions to the equation. Then the function $v=u_1^{-1}u_2$ satisfies
    $$\begin{cases}
        L_{\bar g}v=0 & \text{on }\Omega, \\
        v=1 & \text{on }\partial\Omega,
    \end{cases}$$
    where $\bar{g}=u_1^{4/(n-2)}g$. Note that $R_{\bar g}=f$ and we can apply the maximum principle to $w=v-1$. Since $w$ satisfies
    $$
        c_n\Delta_{\bar g}w=((1+w)-(1+w)^\alpha)f,
    $$
    $w$ cannot achieve a positive maximum or a negative minimum in the interior of $\Omega$. Hence, $w\equiv 0$ on $\Omega$, and it follows that $v\equiv 1$ and $u_1\equiv u_2$ on $\Omega$. 
\end{proof}

We emphasize that this is not sufficient for our purpose since the metric is often degenerate at the boundary of the domain that we consider. We also note that the existence of a sub-solution and a super-solution gives rise to a solution. 

\begin{lemma}\label{lem:sub-super}
    Let $\Omega$ be a connected, precompact open set in $M$ with smooth boundary, and let $u_0\in C^\infty(\partial\Omega)$ be positive. If there exist $\bar u,\underline{u}\in H^1(\Omega)$ that satisfy 
    \begin{enumerate}
        \item $0\le \underline{u}\le \bar u\le U$ on $\Omega$ for some constant $U$, 
        \item $\underline{u}\le u_0\le \bar u$ on $\partial\Omega$, in the trace sense, and 
        \item $L_g\bar u\ge 0$ and $L_g\underline{u}\le 0$ in the weak sense, 
    \end{enumerate}
    then there exists a unique positive $u\in C^\infty(\bar\Omega)$ such that $L_gu=0$ and $u|_{\partial\Omega}=u_0$. 
\end{lemma}
\begin{proof}
    We can rewrite the equation as 
    $$
        -c_n\Delta_gu=fu^\alpha-R_gu=:F(x,u).
    $$
    Then we can find $\lambda>0$ such that 
    $$
        |\partial_u F(x,u)|=|\alpha fu^{\alpha-1}-R_g|\le \lambda,
    $$
    for all $x\in\bar\Omega$ and $u\in[0,U]$. Let $u_1$ be the unique solution to 
    $$\begin{cases}
        -c_n\Delta_gu_1+\lambda u_1=F(x,\underline{u})+\lambda \underline{u} & \text{on }\Omega, \\
        u_1=u_0 & \text{on }\partial\Omega.
    \end{cases}$$
    Inductively, we can define $u_{k+1}$ as the unique solution to
    $$\begin{cases}
        -c_n\Delta_gu_{k+1}+\lambda u_{k+1}=F(x,u_k)+\lambda u_k & \text{on }\Omega, \\
        u_{k+1}=u_0 & \text{on }\partial\Omega.
    \end{cases}$$
    Then the standard argument (cf.\ \cite[\S9.3]{Evans2010}) shows that 
    $$
        0\le \underline{u}\le u_1\le u_2\le \cdots \le \bar u\le U.
    $$
    Hence, $u_k$ converges to a function $u$ that satisfies $L_gu=0$ and $u|_{\partial\Omega}=u_0$. The other assertions follow from Lemma \ref{lem:variation}.
\end{proof}

\begin{lemma}\label{lem:comparison}
    Let $\Omega$ be a connected precompact open set in $M$. If $u_1,u_2\in C^2(\Omega)\cap C^0(\bar\Omega)$ satisfy
    \begin{enumerate}
        \item $u_1,u_2\ge 0$,
        \item $L_gu_1\le L_gu_2$ on $\Omega$, and
        \item $u_1\le u_2$ on $\partial \Omega$,
    \end{enumerate}
    then $u_1\le u_2+C$ on $\Omega$. Here, $C=(\sup_{\Omega}R_g^-)^{\frac{1}{\alpha-1}}$. 
\end{lemma}
\begin{proof}
    If $u_1\le u_2$ on $\Omega$, then the conclusion holds. Otherwise, there is a point $x_0\in \Omega$ such that $u_1(x_0)-u_2(x_0)=\max_{\Omega}(u_1-u_2)>0$. Hence, at $x_0$, we have 
    $$
        0\ge c_n\Delta_g(u_1-u_2)\ge (R_gu_1-fu_1^{\alpha})-(R_gu_2-fu_2^{\alpha}),
    $$
    and thus 
    $$
        R_gu_1-fu_1^\alpha\le R_gu_2-fu_2^\alpha.
    $$
    Consider the function 
    $$
        \varphi(u)=R_g(x_0)u-f(x_0)u^\alpha.
    $$
    Note that $\varphi'(u)=R_g(x_0)-\alpha f(x_0)u^{\alpha-1}$, and thus we consider two cases. 

    Case 1: $R_g(x_0)\ge 0$. Then $\varphi(u)$ is increasing when $u\ge 0$. Hence, we have $u_1(x_0)\le u_2(x_0)$, which contradicts the choice of $x_0$.

    Case 2: $R_g(x_0)<0$. Then $\varphi(u)$ is decreasing when $u\in[0,C_1]$ and is increasing when $u\in[C_1,\infty)$, where
    $$
        C_1=\left(\frac{R_g(x_0)}{\alpha f(x_0)}\right)^{\frac{1}{\alpha-1}}<C.
    $$
    If $u_1(x_0)>C$, then 
    $$
        R_g(x_0)u_2(x_0)-f(x_0)u_2(x_0)^\alpha\ge R_g(x_0)u_1(x_0)-f(x_0)u_1(x_0)^\alpha>0,
    $$
    and thus 
    $$
        u_2(x_0)^{\alpha-1}>\frac{R_g(x_0)}{f(x_0)}>\frac{R_g(x_0)}{\alpha f(x_0)}=C_1^{\alpha-1}.
    $$
    Hence $u_2(x_0)>C_1$. Since $u_1(x_0)>C>C_1$ and $\varphi$ is strictly increasing on $[C_1,\infty)$, the inequality $\varphi(u_2)\ge\varphi(u_1)$ implies $u_2(x_0)\ge u_1(x_0)$, which contradicts $u_1(x_0)-u_2(x_0)>0$. Hence, $u_1(x_0)\le C$. Since $u_2\ge 0$, we have $u_1-u_2\le u_1(x_0)-u_2(x_0)\le C$ on $\Omega$.
\end{proof}
\begin{remark}
    If $R_g\ge f$, then we can take $C=1$ by carefully analyzing the proof. 
\end{remark}

Next, we note the strong maximum principle holds. 

\begin{lemma}\label{lem:sMaxPr-1}
    Let $\Omega$ be a connected, precompact open set in $M$. If $u\in C^2(\Omega)\cap C^0(\bar\Omega)$ satisfies
    \begin{enumerate}
        \item $u\ge 0$,
        \item $L_gu\ge 0$ on $\Omega$, 
        \item $u(x_0)=0$ for some $x_0\in\Omega$,
    \end{enumerate}
    then $u\equiv 0$ on $\Omega$.
\end{lemma}
\begin{proof}
    We write the equation $L_gu\ge 0$ as
    $$
        \Delta_gu\le c_n^{-1}(R_gu-fu^{\alpha})\le Cu,
    $$
    where 
    $$
        C=\max_{\bar\Omega} c_n^{-1}(R_g-fu^{\alpha-1}).
    $$
    Then the strong maximum principle implies that $u\equiv 0$ on $\Omega$.
\end{proof}
\begin{lemma}\label{lem:sMaxPr-2}
    Let $\Omega$ be a connected, precompact open set in $M$.  Assume that $R_g\ge f$. If $u\in C^2(\Omega)\cap C^0(\bar\Omega)$ satisfies
    \begin{enumerate}
        \item $0\le u\le 1$,
        \item $L_gu\le 0$ on $\Omega$, 
        \item $u(x_0)=1$ for some $x_0\in\Omega$,
    \end{enumerate}
    then $u\equiv 1$ on $\Omega$.
\end{lemma}
\begin{proof}
    Let $v=u-1\le 0$. Then $v$ satisfies
    $$
        \Delta_gv\ge c_n^{-1}(R_gu-fu^{\alpha})\ge c_n^{-1}fu(1-u^{\alpha-1})=-c_n^{-1}fu\varphi(u)v,
    $$
    where 
    $$
        \varphi(u)=
        \begin{cases}
            \frac{1-u^{\alpha-1}}{1-u}, u\ne 1; \\
            \alpha-1, u=1.
        \end{cases}
    $$
    Let $u_0=\min_{\bar\Omega}u$. Then the function $\varphi(u)$ is bounded on $[u_0,1]$. Hence, there is a constant $C$ such that $\Delta_gv\ge Cv$. The strong maximum principle implies that $v\equiv 0$ on $\Omega$.
\end{proof}

\subsection{The Yamabe equation on manifolds with boundary}\label{sec:Yamabe eqn bdry}

Let $(M^n,g)$ be a complete Riemannian manifold with compact boundary $\partial M\ne\varnothing$. Fix a compact connected $K\supset\partial M$. Let 
$$
    M\setminus K=E_0\cup\bigcup_{i\in I}E_i
$$
be the components, where $E_0$ is an AH end with conformal infinity $\partial X$. The other ends may or may not be AH, and we allow that $I=\varnothing$. 

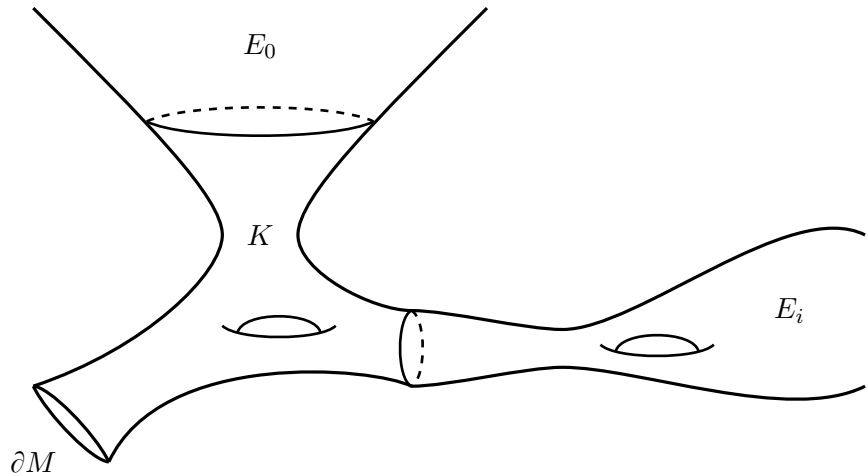
\begin{figure}[ht]
\centering
\begin{tikzpicture}
    \coordinate (E01) at (-3,3);
    \coordinate (E02) at (3,3);
    \coordinate (Ei1) at (8,0);
    \coordinate (Ei2) at (8,-2);
    \coordinate (K1) at (-1.5,1.5);
    \coordinate (K2) at (1.5,1.5);
    \coordinate (K3) at (2,-1);
    \coordinate (K4) at (2,-2);
    \coordinate (S1) at (-3,-2);
    \coordinate (S2) at (-2,-3);

    \draw[line width=1.2pt]
    (E01)
    .. controls (K1) and (-0.5,0.5) .. (-0.5,0)
    .. controls (-0.5,-0.5) and (-1.5,-1.5) .. (S1);

    \draw[line width=1.2pt]
    (E02)
    .. controls (K2) and (0.5,0.5) .. (0.5,0)
    .. controls (0.5,-0.5) and (1.5,-1) .. (K3)
    .. controls (2.5,-1) and (3.5,-1.25) .. (4,-1.25)
    .. controls (5,-1.25) and (7,0.5) .. (Ei1);

    \draw[line width=1.2pt]
    (S2)
    .. controls (-1.3,-1.5) and (1.5,-1.75) .. (K4)
    .. controls (2.5,-2) and (3.5,-1.75) .. (4,-1.75)
    .. controls (5,-1.75) and (7,-2.5) .. (Ei2);

    \draw[dashed,line width=1pt]
    (K1)
    .. controls (-1,1.75) and (1,1.75) .. (K2);
    \draw[line width=1pt]
    (K1)
    .. controls (-1,1.25) and (1,1.25) .. (K2);

    \draw[dashed,line width=1pt]
    (K3)
    .. controls (2.2,-1.1) and (2.2,-1.9) .. (K4);
    \draw[line width=1pt]
    (K3)
    .. controls (1.8,-1.1) and (1.8,-1.9) .. (K4);

    \draw[line width=1.2pt]
    (S1)
    .. controls (-2.8,-2) and (-2.2,-2.6) .. (S2); 
    \draw[line width=1.2pt]
    (S1)
    .. controls (-2.8,-2.4) and (-2.2,-3) .. (S2); 

    \draw[line width=1pt]
    (-0.5,-1.2) .. controls (-0.3,-1.4) and (0.8,-1.4) .. (1,-1.2);
    \draw[line width=1pt]
    (-0.3,-1.3) .. controls (-0.25,-1) and (0.75,-1) .. (0.8,-1.3);
    \draw[line width=1pt]
    (4.5,-1.45) .. controls (4.7,-1.65) and (5.8,-1.65) .. (6,-1.45);
    \draw[line width=1pt]
    (4.7,-1.55) .. controls (4.75,-1.25) and (5.75,-1.25) .. (5.8,-1.55);

    \node at (0,2.5) {$E_0$};
    \node at (0,0) {$K$};
    \node at (7,-1) {$E_i$};
    \node at (-3,-3) {$\partial M$};

\end{tikzpicture}
\caption{The manifold $M$ with boundary.}
\label{fig:ends2}
\end{figure}

Consider the equation 
$$
    L_gu=-c_n\Delta_gu+R_gu-fu^{\alpha}=0,
$$
where $f\in C^\infty(M)$ satisfies
\begin{enumerate}
    \item $-n(n-1)\le f\le -1$, 
    \item $f=-n(n-1)$ on $K$, and
    \item $f=R_g$ on $E_0\cap\{\rho<\rho_0\}$ for some $\rho_0>0$.
\end{enumerate}
Here, $\rho: M\to(0,\infty)$ is a defining function of the AH end $E_0$ as in the introduction. Let $d(x)=\dist_g(x,\partial M)$. Let $\Int M=M\setminus\partial M$ be the interior of $M$. 

\begin{proposition}\label{prop:YamabeAH}
    There exists $u\in C^\infty(\Int M)$ satisfying the equation \eqref{eq:Yamabe} such that 
    \begin{enumerate}
        \item $u>0$ in the interior of $M$, 
        \item $u(x)d(x)^{\frac{n-2}{2}}\to 1$ as $x\to\partial M$, and
        \item $u(x)\to 1$ as $x\to\partial X$.
    \end{enumerate}
    The solution is unique if $I=\varnothing$. Moreover, near $\partial X$, we have 
    $$
        1-A\rho^n+B\rho^{n+1}\le u\le 1+A\rho^n-B\rho^{n+1}
    $$
    for some constants $A,B>0$ depending on $E_0$ and $\rho_0$.
\end{proposition}

The proof follows from the techniques introduced in \cite{ACF92}. We first establish some preliminary lemmas.

\begin{lemma}\label{lem:def fn}
    There exists a smooth function $r:M\to[0,1]$ such that 
    \begin{enumerate}
        \item $r^{-1}(0)=\partial M$, 
        \item $r=1$ outside $K$, 
        \item $|dr|_g=1$ on $\partial M$, and
        \item $R_h=-n(n-1)+R_nr^n$ for some $R_n\in C^\infty(M)$.
    \end{enumerate}
    Here, $h=r^{-2}g$ is a Riemannian metric on $\Int M$. 
\end{lemma}
\begin{proof}
    This is \cite[Lemma 2.1]{ACF92}. The only extra conditions here are that $r\le 1$ on $M$ and $r=1$ outside the compact set $K$. We can achieve these by adjusting $r$ outside a collar neighborhood of $\partial M$. 
\end{proof}

\begin{lemma}
    Let $(r,h)$ be as in Lemma \ref{lem:def fn}. Then $L_gu=0$ is equivalent to $L_hv=0$ where $v=r^{\frac{n-2}{2}}u$. 
\end{lemma}
\begin{proof}
    This follows from the conformal covariance of the conformal Laplacian. 
\end{proof}

\begin{lemma}\label{lem:YamabeAH}
    Let $(r,h)$ be as in Lemma \ref{lem:def fn}. Then there exists $v\in C^\infty(\Int M)\cap C^0(M)$ such that 
    \begin{enumerate}
        \item $L_hv=0$ on $\Int M$,
        \item $v$ is bounded and $v>0$ on $M$,
        \item $v=1$ on $\partial M$, and 
        \item $v(x)\to 1$ as $x\to\partial X$.
    \end{enumerate}
   Moreover, $r^{-\eta}(v-1)$ is bounded for any real number $\eta<n$ near $\partial M$, and $\rho^{-n}(v-1)$ is bounded near $\partial X$.  Such $v$ is unique if $I=\varnothing$. 
\end{lemma}
\begin{proof}
    (i) Existence. Fix a positive number $\eta_0<n-1$. Since $|dr|_g=1$ on $\partial M$, in a collar $\{r\le r_0\}$, we have
    \begin{align*}
    \Delta_h(r^{\eta_0})
    &=\eta_0(\eta_0+1-n)|dr|_g^2r^{\eta_0}
      +\eta_0r^{\eta_0+1}\Delta_gr\\
    &=\eta_0(\eta_0+1-n)r^{\eta_0}+O(r^{\eta_0+1}).
    \end{align*}
    Here $f=-n(n-1)$ and $R_h=-n(n-1)+O(r^n)$. Put $s=ar^{\eta_0}$, where $a\ge1$. Direct substitution gives
    \begin{align*}
    L_h(1-s)
    &=c_n\eta_0(\eta_0+1-n)s\\
    &\quad+n(n-1)\big((1-s)^\alpha-(1-s)\big)
      +O(rs)+O(r^n)(1-s),\\
    &=c_n\eta_0(\eta_0+1-n)s\\
    &\quad+n(n-1)\big((1-s)^\alpha-(1-s)\big)
      +O(rs)+O(r^n),\quad 0\le s\le1,\\
    L_h(1+s)
    &=-c_n\eta_0(\eta_0+1-n)s\\
    &\quad+n(n-1)\big((1+s)^\alpha-(1+s)\big)
      +O(rs)+O(r^n)(1+s),\\
    &=-c_n\eta_0(\eta_0+1-n)s\\
    &\quad+n(n-1)\big((1+s)^\alpha-(1+s)\big)
      +O(rs)+O(r^n),\quad s\ge0.
    \end{align*}
    The constants in these remainder bounds are independent of $a$. Since $\eta_0+1-n<0$, $(1-s)^\alpha\le1-s$ for $0\le s\le1$, and $(1+s)^\alpha\ge1+s$ for $s\ge0$, decreasing $r_0$ suitably absorbs the remainder terms: indeed $r^n/s=a^{-1}r^{n-\eta_0}$ and $a\ge 1$. Thus
    \[
    L_h(1-ar^{\eta_0})\le0\quad\text{where }0\le ar^{\eta_0}\le1,
    \qquad L_h(1+ar^{\eta_0})\ge0
    \]
    in this collar. Note that once $r_0$ is fixed, $a\ge 1$ can be arbitrarily large. We use these functions only in the region
    \[
        \{x\in M:r(x)<r_0, \ ar(x)^{\eta_0}\le 1\},
    \]
    and join them to the constant barriers below.
    
    On a collar of $\partial X$, say $\{\rho<\rho_0\}$, we have $h=g$, $f=R_g<0$, and
    \[
    \Delta_h\rho^{\eta_0}
    =\eta_0(\eta_0+1-n)\rho^{\eta_0}+O(\rho^{\eta_0+2})<0.
    \]
    Consequently, for any $a\ge1$,
    \begin{align*}
    L_h(1-a\rho^{\eta_0})
    &=c_na\Delta_h\rho^{\eta_0}
     +R_g\bigl(1-a\rho^{\eta_0}-(1-a\rho^{\eta_0})^\alpha\bigr)\le0
     \quad(0\le a\rho^{\eta_0}\le1),\\
    L_h(1+a\rho^{\eta_0})
    &=-c_na\Delta_h\rho^{\eta_0}
     +R_g\bigl(1+a\rho^{\eta_0}-(1+a\rho^{\eta_0})^\alpha\bigr)\ge0.
    \end{align*}
    The collars of $\partial M$ and $\partial X$ are chosen disjoint. Since $R_h$ is bounded below and $-f\ge1$, a constant $b>1$ with
    \[
    b^{\alpha-1}>\max\{0,-\inf_M R_h\}
    \]
    is a global supersolution, while $0$ is a global subsolution. Choose one sufficiently large $a$ so that, in each collar, the lower branch reaches $0$ and the upper branch reaches $b$ before the inner edge of the collar. Define
    \[
    \begin{array}{c|c|c}
     &\underline v&\bar v\\ \hline
    \text{collar of }\partial M
     &\max\{1-ar^{\eta_0},0\}&\min\{1+ar^{\eta_0},b\}\\
    \text{collar of }\partial X
     &\max\{1-a\rho^{\eta_0},0\}&\min\{1+a\rho^{\eta_0},b\}\\
    \text{outside the two collars}&0&b
    \end{array}
    \]
    These definitions agree on neighborhoods of the inner edges. The functions are locally Lipschitz, satisfy $0\le\underline v\le1\le\bar v\le b$, and equal $0,b$, respectively, on every arbitrary end. The maximum of subsolutions and the minimum of supersolutions preserve the respective weak inequalities. Thus $\underline v$ and $\bar v$ are global weak barriers.

    Hence, Lemma \ref{lem:sub-super} implies that, on any compact $K_j\subset\Int M$, there exists $v_j\in C^\infty(K_j)$ such that $L_hv_j=0$ on $K_j$ and $\underline{v}\le v_j\le \bar{v}$. By Corollary \ref{cor:convergence}, there is a subsequence of $\{v_j\}$ converging smoothly on compact sets to $v\in C^\infty(\Int M)$ such that $L_hv=0$ on $\Int M$ and $\underline{v}\le v\le \bar{v}$. The boundary conditions (3) and (4) follow from the construction of $\underline{v}$ and $\bar{v}$. We also have $0\le v\le b$ on $M$. By the strong maximum principle, we have $v>0$.

    (ii) Uniqueness. Suppose $v_1,v_2>0$ satisfy the given conditions (1) -- (4). Let $v=v_1^{-1}v_2$ and $\tilde{h}=v_1^{4/(n-2)}h$. Then, by the covariance of the conformal Laplacian, we have
    $$
        L_{\tilde h}v=0\text{ on }\Int M,
    $$
    and $v=1$ on $\partial M$ and $\partial X$. Note that
    $$
        R_{\tilde h}=v_1^{-\alpha}(R_hv_1-c_n\Delta_h v_1)=f.
    $$
    Hence, $w=v-1$ satisfies the equation 
    $$
        -c_n\Delta_{\tilde h}w+(w+1)f-(w+1)^\alpha f=0,
    $$
    and the maximum principle implies that $w\equiv 0$, and thus $v_1=v_2$. 

    (iii) Regularity and asymptotics. Elliptic regularity gives $v\in C^\infty(\Int M)$. For the inner boundary $\partial M$, part (i) shows $w:=v-1=O(r^{\eta_0})$. Let $L=-\Delta_h+n$ be the linearization of $L_h$ at $1$. Then 
    \[
    c_nLw=-(R_h+n(n-1))(1+w)-n(n-1)\big((1+w)^\alpha-1-\alpha w\big).
    \]
    Here, we take $\eta_0\ge n/2$, which is possible since $n>2$. Then  
    \[
        Lw=O(r^n)+O(w^2)=O(r^n).
    \]
    For any fixed $0<\eta<n$,
    \[
    (-\Delta_h+n)r^\eta
    =(\eta+1)(n-\eta)r^\eta+O(r^{\eta+1})
    \ge c_\eta r^\eta
    \]
    on a sufficiently small collar $\{r<r_1\}$ (where $r_1<r_0$ for the $r_0$ in part (i)). Choose $C$ so that $Cr^\eta$ dominates $|w|$ on $\{r=r_1\}$ and 
    \[
    C(-\Delta_h+n)r^\eta\ge|(-\Delta_h+n)w|
    \] 
    throughout the collar. Since $w\to0$ at $r=0$, the maximum principle applied to $\pm w-Cr^\eta$ gives $|w|\le Cr^\eta$. This proves the asserted estimate for every $0<\eta<n$; the case $\eta\le0$ follows from boundedness. 
    
    On the AH collar, part (i) gives $w:=v-1=O(\rho^{\eta_0})$. Since $g=h$ and $f=R_g=-n(n-1)+O(\rho^n)$ there, we have
    \[
    c_n\Delta_hw=R_g((w+1)^\alpha-(w+1))=wR_g\left(\alpha\int_0^1(1+sw)^{\alpha-1}\,ds-1\right).
    \]
    We introduce the linear operator 
    \[
    \tilde{L}=-\Delta_h+c_n^{-1}R_g\left(1-\alpha\int_0^1(1+sw)^{\alpha-1}\,ds\right)=-\Delta_h+n+O(\rho^{\eta_0}).
    \]
    (The previous argument using $L$ does not work for $n=3$.) It follows that $\tilde{L}w=0$, 
    \[
    \tilde{L}
    \bigl(\rho^n(1-\rho)\bigr)=(n+2)\rho^{n+1}
    +O(\rho^{n+2})+O(\rho^{n+\eta_0})>0
    \]
    on $\{\rho<\rho_1\}$, where $\rho_1<\min\{1,\rho_0\}$ for the $\rho_0$ in part (i) and $\eta_0$ is chosen to be $>1$. Choose $A$ so that $A\rho^n(1-\rho)$ dominates $w$ on $\{\rho=\rho_1\}$. Both functions tend to zero at infinity, so the maximum principle yields
    \[
    w\le A\rho^n(1-\rho)
    =A\rho^n-A\rho^{n+1},
    \]
    This proves the stated AH bounds. The higher-order barriers are used only on this smaller collar, after constructing the global solution.
\end{proof}

\begin{proof}[Proof of Proposition \ref{prop:YamabeAH}]
    Let $v$ be as in the previous lemma. Then $u=r^{-\frac{n-2}{2}}v$ is a solution to the equation \eqref{eq:Yamabe} on $\Int M$, with $ur^{\frac{n-2}{2}}\to 1$ as $r\to 0$. Since $|dr|_g=1$ on $\partial M$, and since $r$ is smooth, we have $r(x)\sim \dist_g(x,\partial M)$ as $x\to\partial M$. This finishes the proof of the proposition. 
\end{proof}

\begin{corollary}\label{cor:AH-coefficient-convergence}
Let $u_j$ be positive solutions of \eqref{eq:Yamabe} on domains containing a neighborhood of the AH end $E_0$. Suppose $R_g=f$ there, $u_j\to1$ at $\partial X$, and $u_j\to u_\infty$ smoothly on compact subsets of this exterior. Suppose their boundary expansions are
\[
u_j=1+A_j(x)\rho^n+O(\rho^{n+1}),\qquad
u_\infty=1+A_\infty(x)\rho^n+O(\rho^{n+1}),
\]
with smooth leading coefficients. Then
\[
\|A_j-A_\infty\|_{C^0(\partial X)}\longrightarrow0.
\] 
\end{corollary}
The existence of these expansions will be established in Proposition \ref{prop:tangential smoothness} below. Note that we only need to assume that $u_j$ solves the Yamabe equation on the AH end. 
\begin{proof}
Local convergence gives a uniform bound on $\{\rho=\rho_1\}$ for any small $\rho_1$. Apply the AH barrier argument in part (iii) of Lemma \ref{lem:YamabeAH}, where we fix $\rho_1$ as in there. The initial barriers $\underline{v}$ and $\bar{v}$ in part (i) of Lemma \ref{lem:YamabeAH} are by definition independent of $j$, while the barriers in part (iii) of Lemma \ref{lem:YamabeAH} can be chosen to be independent of $j$ since $u_j\to u_\infty$ uniformly on compact sets. Hence
\begin{equation}\label{eq:uniform-fixed-AH}
|u_j-1|\le A\rho^n\qquad(0<\rho<\rho_1),
\end{equation}
with $A,\rho_1$ independent of $j$. This estimate also holds for $u_\infty$ by local convergence. Subtract the equations for $u_j$ and $u_\infty$. On $E_0$,
\[
\left(-\Delta_g+c_n^{-1}\left[R_g-f\alpha\int_0^1(u_\infty+s(u_j-u_\infty))^{\alpha-1}\,ds\right]\right)(u_j-u_\infty)=0.
\]
The zeroth-order coefficient is $n+O(\rho^n)$ uniformly in $j$. For sufficiently small $\rho_1$, $\rho^n(1-\rho)$ is a positive supersolution for all these operators on $0<\rho\le\rho_1$. Put
\[
\varepsilon_j=\sup_{E_0\cap\{\rho=\rho_1\}}|u_j-u_\infty|\longrightarrow0.
\]
The maximum principle implies that 
\[
|u_j-u_\infty|\le
\frac{\varepsilon_j}{\rho_1^n(1-\rho_1)}\rho^n(1-\rho).
\]
Dividing by $\rho^n$ and taking $\rho\downarrow0$ yields
\[
\|A_j-A_\infty\|_{C^0(\partial X)}\le
\frac{\varepsilon_j}{\rho_1^n(1-\rho_1)}
\]
which tends to $0$ as $j\to \infty$.
\end{proof}

If $I\ne\varnothing$, to obtain uniqueness, we may cut off these ends and impose a Dirichlet boundary condition on the cut-off boundary. To be precise, let 
$$
    M\setminus K=E_0\cup\bigcup_{i\in I}E_i
$$ 
as before. Choose a compact connected domain $K'$ containing $K$ in its interior, with smooth boundary away from $\partial M$. Set $\Sigma_i=\partial K'\cap E_i$ and 
$$
    M'=K'\cup E_0.
$$

\begin{figure}[ht]
\centering
\begin{tikzpicture}
    \coordinate (E01) at (-3,3);
    \coordinate (E02) at (3,3);
    \coordinate (Ei1) at (8,0);
    \coordinate (Ei2) at (8,-2);
    \coordinate (K1) at (-1.5,1.5);
    \coordinate (K2) at (1.5,1.5);
    \coordinate (K3) at (2,-1);
    \coordinate (K4) at (2,-2);
    \coordinate (S1) at (-3,-2);
    \coordinate (S2) at (-2,-3);

    \draw[line width=1.2pt]
    (E01)
    .. controls (K1) and (-0.5,0.5) .. (-0.5,0)
    .. controls (-0.5,-0.5) and (-1.5,-1.5) .. (S1);

    \draw[line width=1.2pt]
    (E02)
    .. controls (K2) and (0.5,0.5) .. (0.5,0)
    .. controls (0.5,-0.5) and (1.5,-1) .. (K3)
    .. controls (2.5,-1) and (3.5,-1.25) .. (4,-1.25)
    .. controls (5,-1.25) and (7,0.5) .. (Ei1);

    \draw[line width=1.2pt]
    (S2)
    .. controls (-1.3,-1.5) and (1.5,-1.75) .. (K4)
    .. controls (2.5,-2) and (3.5,-1.75) .. (4,-1.75)
    .. controls (5,-1.75) and (7,-2.5) .. (Ei2);

    \draw[dashed,line width=1pt]
    (K1)
    .. controls (-1,1.75) and (1,1.75) .. (K2);
    \draw[line width=1pt]
    (K1)
    .. controls (-1,1.25) and (1,1.25) .. (K2);

    \draw[dashed,line width=1pt]
    (K3)
    .. controls (2.2,-1.1) and (2.2,-1.9) .. (K4);
    \draw[line width=1pt]
    (K3)
    .. controls (1.8,-1.1) and (1.8,-1.9) .. (K4);

    \draw[dashed,color=red,line width=1pt]
    (2.5,-1.05)
    .. controls (2.7,-1.2) and (2.7,-1.8) .. (2.5,-1.95);
    \draw[color=red,line width=1pt]
    (2.5,-1.05)
    .. controls (2.3,-1.2) and (2.3,-1.8) .. (2.5,-1.95);

    \draw[line width=1.2pt]
    (S1)
    .. controls (-2.8,-2) and (-2.2,-2.6) .. (S2); 
    \draw[line width=1.2pt]
    (S1)
    .. controls (-2.8,-2.4) and (-2.2,-3) .. (S2); 

    \draw[line width=1pt]
    (-0.5,-1.2) .. controls (-0.3,-1.4) and (0.8,-1.4) .. (1,-1.2);
    \draw[line width=1pt]
    (-0.3,-1.3) .. controls (-0.25,-1) and (0.75,-1) .. (0.8,-1.3);
    \draw[line width=1pt]
    (4.5,-1.45) .. controls (4.7,-1.65) and (5.8,-1.65) .. (6,-1.45);
    \draw[line width=1pt]
    (4.7,-1.55) .. controls (4.75,-1.25) and (5.75,-1.25) .. (5.8,-1.55);

    \node at (0,2.5) {$E_0$};
    \node at (0,0) {$K$};
    \node at (7,-1) {$E_i$};
    \node at (-3,-3) {$\partial M$};

    \draw[decorate,decoration={brace,amplitude=6pt},color=red]
    (-3.5,-2) -- (-3.5,3) node [midway,xshift=-15pt] {$M'$};
    \draw[decorate,decoration={brace,amplitude=6pt},color=red]
    (2.5,-3.5) -- (-3,-3.5) node [midway,yshift=-15pt] {$M'$};
    \node at (2.6,-0.8) [color=red] {$\Sigma_i$};
\end{tikzpicture}
\caption{The domain $M'$.}
\label{fig:cutoff2}
\end{figure}

\begin{lemma}\label{lem:YamabeArbitrary}
    Let $(r,h)$ be as in Lemma \ref{lem:def fn}. Then there exists a unique $v\in C^\infty(M'\setminus\partial M)\cap C^0(M')$ such that 
    \begin{enumerate}
        \item $L_hv=0$ on $\Int M'$,
        \item $v$ is bounded and $v>0$ on $M'$,
        \item $v=1$ on $\partial M$, 
        \item $v(x)\to 1$ as $x\to\partial X$, and 
        \item $v=1$ on $\Sigma_i$ for each $i\in I$.
    \end{enumerate}
    Moreover, $r^{-\eta}(v-1)$ is bounded for any real number $\eta<n$, and $\rho^{-n}(v-1)$ is bounded. 
\end{lemma}
\begin{proof}
    Let $\{K_j\}$ be a compact exhaustion of $M'\setminus\partial M$, with $K_j\supset \Sigma_i$ for all $i\in I$. Let $\underline{v}$ and $\bar{v}$ be the sub-solution and super-solution constructed in the proof of Lemma \ref{lem:YamabeAH}. Note that $\underline{v}\le 1\le \bar{v}$ on $\Sigma_i$ for each $i\in I$. Hence, we can apply Lemma \ref{lem:sub-super} to see that there exists a unique $v_j\in C^\infty(K_j)$ such that $L_hv_j=0$ on $K_j$, $\underline{v}\le v_j\le \bar{v}$ on $K_j$, and $v_j=1$ on $\partial K_j$. 
    
    By Corollary \ref{cor:convergence}, there is a subsequence of $\{v_j\}$ that converges smoothly on compact sets to a solution $v\in C^\infty(M'\setminus\partial M)\cap C^0(M')$ to the equation. The barriers and the strong maximum principle give positivity and the boundary behavior. The sharper decay estimates follow from the comparisons in part (iii) of the proof of Lemma \ref{lem:YamabeAH}, which apply on the same boundary collars.
\end{proof}

From this, we conclude that 

\begin{proposition}\label{prop:YamabeArbitrary}
    There exists a unique solution $u\in C^\infty(M'\setminus\partial M)$ to the equation \eqref{eq:Yamabe} such that 
    \begin{enumerate}
        \item $u>0$ on $M'$, 
        \item $u(x)d(x)^{\frac{n-2}{2}}\to 1$ as $x\to\partial M$, 
        \item $u(x)\to 1$ as $x\to\partial X$, and 
        \item $u(x)=1$ on $\Sigma_i$ for each $i\in I$. 
    \end{enumerate}
    Moreover, near $\partial X$, we have 
    $$
        1-A\rho^n+A\rho^{n+1}\le u\le 1+A\rho^n-A\rho^{n+1}
    $$
    for some constant $A>0$ depending on $\rho_0$.
\end{proposition}

\subsection{Bounded solutions of the Yamabe equation}\label{sec:Yamabe eqn bound}

For the arguments below, we may assume $\mathcal S\ne\varnothing$: if it is empty, we may designate one smooth point outside the AH exterior as singular. This auxiliary choice preserves the hypotheses used here and does not change the original singular set in the later rigidity statements.

Let $(M,d,\mu)$ be as in the introduction, with singular set $\mathcal{S}$ and regular set $\mathcal{R}$. Let $K\supset\mathcal{S}$ be a compact set such that 
$$
    M\setminus K=E_0\cup\bigcup_{i\in I}E_i,
$$
where $E_0$ is an AH end with conformal infinity $\partial X$. The other ends may or may not be AH, and we allow that $I=\varnothing$.  
Choose a compact domain $K'$ containing $K$ in its interior and with smooth boundary in $\mathcal R$. Set $\Sigma_i=\partial K'\cap E_i$ and $M'=K'\cup E_0$.
 (See Figure \ref{fig:cutoff}.) By Lemma 2.17 in \cite{BHHSZ2026}, we can always assume that $K'\setminus\mathcal{S}$ is connected. 

In this subsection, we show some properties of a bounded solution to the equation \eqref{eq:Yamabe}. More precisely, suppose that $u\in C^\infty(M'\setminus\mathcal{S})$ satisfies 
\begin{enumerate}
    \item $L_gu:=-c_n\Delta_gu+R_gu-fu^\alpha=0$ on $M'\setminus\mathcal{S}$,
    \item $0\le u\le C$ on $M'\setminus\mathcal{S}$ for some constant $C$, and
    \item $u(x)=1$ on $\Sigma_i$ for each $i\in I$, and $u(x)\to 1$ as $x\to\partial X$. 
\end{enumerate}
Recall that $R_g=f$ on $E_0\cap\{\rho<\rho_0\}$ for some $\rho_0>0$. We may also assume that $R_g\le -1$ on $E_0\cap\{\rho<\rho_0\}$. We do not assert the uniqueness of $u$, since we have no information on the behavior of $u$ near the singular set $\mathcal{S}$.

\begin{proposition}\label{prop:upper bound}
    Assume that $R_g\ge f$ on $M'\setminus\mathcal{S}$. Then the function $u$ described above satisfies $0\le u\le 1$ on $M'\setminus\mathcal{S}$.
\end{proposition}
\begin{proof}
    First, we show that there exists a positive harmonic function $G$ on $M'\setminus\mathcal{S}$ such that 
    \begin{enumerate}
        \item $G(x)\to \infty$ as $x\to\mathcal{S}$, 
        \item $G(x)=1$ on $\Sigma_i$ for each $i\in I$, and 
        \item $G(x)\to 0$ as $x\to\partial X$.
    \end{enumerate}
    To construct such a $G$, we start from a positive harmonic function $G_S$ on $M'\setminus\mathcal{S}$ such that $G_S(x)\to \infty$ as $x\to\mathcal{S}$, $G_S(x)=0$ on $\Sigma_i$ for each $i\in I$, and $G_S(x)\to 0$ as $x\to\partial X$. (cf.\ \cite{HSY25}.) On the other hand, there exists a positive harmonic function $G_1$ such that $G_1=1$ on $\Sigma_i$ for each $i\in I$, and $G_1(x)\to 0$ as $x\to\partial X$. Then we can take $G=G_S+G_1$.

    Next, suppose the contrary that there is a point $x_0\in M'\setminus\mathcal{S}$ such that $u(x_0)>1$. Then $x_0$ is in the interior of $M'\setminus\mathcal{S}$. For any $\epsilon>0$, we take a neighborhood $U$ of $\mathcal{S}$ such that $x_0\notin U$ and 
    $$
        v_\epsilon:=1+\epsilon G\ge C.
    $$
    Then we have 
    $$
        L_gv_\epsilon=R_g(1+\epsilon G)-(1+\epsilon G)^\alpha f\ge (1+\epsilon G-(1+\epsilon G)^\alpha)f\ge 0.
    $$
    Since $v_\epsilon\ge u$ on $\partial U$, $\Sigma_i$ for each $i\in I$, and $\partial X$, the maximum principle implies that $v_\epsilon\ge u$ on $M'\setminus U$. (Indeed, if $u-v_\epsilon$ achieves a positive maximum at some point $x_1$ in the interior of $M'\setminus U$, then at $x_1$ we have
    \begin{align*}
        0&\ge c_n\Delta_g(u-v_\epsilon)\ge (u-v_\epsilon)R_g-(u^\alpha-v_\epsilon^\alpha)f\\
        &\ge((u-v_\epsilon)-(u^\alpha-v_\epsilon^\alpha))f=((u-u^\alpha)-(v_\epsilon-v_\epsilon^\alpha))f>0,
    \end{align*}
    where we have used the fact that $u\mapsto u-u^\alpha$ is decreasing when $u>1$. This is a contradiction.) Hence,
    $$
        u(x_0)\le v_\epsilon(x_0)=1+\epsilon G(x_0).
    $$
    This holds for any $\epsilon>0$, and thus $u(x_0)\le 1$.
\end{proof}

\begin{proposition}\label{prop:lower bound}
    For any open neighborhood $U$ of $\mathcal{S}$, we can find $\delta>0$ such that $u\ge \delta$ on $M'\setminus U$. The constant $\delta$ depends only on $(M,d,\mu)$, $g$ and $U$, independent of $u$. 
\end{proposition}
\begin{proof}
    We first prove a uniform positive lower bound on each $K\Subset M'\setminus U$. Otherwise, there are solutions $u_j$ and points $x_j$ in a fixed compact subset with $u_j(x_j)\to0$. Passing to a subsequence, $x_j\to x$ and $u_j\to u_\infty\ge0$ locally in $C^2$, with $u_\infty(x)=0$. Then $x\in K$. Apply the strong maximum principle on a connected relatively compact open neighborhood of $K$ to obtain
    \[
    u_\infty\equiv0\quad\text{on }K.
    \]
    Thus $u_j\to0$ uniformly on $K$. 
    
    On the other hand, \eqref{eq:uniform-fixed-AH} gives
    \[
    u_j\ge\frac12\qquad\text{on }E_0\cap\{\rho\le \rho_1\},
    \]
    for a common $\rho_1\in(0,1)$. We may assume that $R_g=f<0$ there, and that $K$ contains 
    \[
    E_0\cap\{\rho=\rho_1\}.
    \]
    This $K$ will be fixed in this argument. Uniform convergence to zero on $K$ contradicts this lower bound. This proves that 
    \[
    u|_K\ge\delta_K>0
    \]
    for any solution $u$. 
    
    Lastly, we show that $u\ge\frac{\delta_K}{2}$ on $E_0\cap\{\rho<\rho_1\}$. If not, then there exists $x_0\in E_0\cap\{\rho<\rho_1\}$ such that $u(x_0)$ attains the minimum of $u$. Then at $x_0$ we have 
    \[
    0\le c_n\Delta_gu=R_g(u-u^\alpha)<0,
    \]
    which is impossible. 
\end{proof}
\begin{remark}
    If $R_g$ is bounded on $M'\setminus\mathcal{S}$, then $\inf_{M'\setminus\mathcal{S}}u>\delta>0$ for some $\delta$ depending only on $(M,d,\mu)$ and $g$. For we know that in this case $u$ is $C^{0,\beta}(K')$ for some $\beta\in(0,1)$. 
\end{remark}

The following is a key step in the proof of rigidity of the positive mass theorem.

\begin{proposition}\label{prop:gap}
    If $u<1$, then there exists $\delta>0$ and $\rho_1\in(0,\rho_0)$ such that 
    $$
        u\le 1-\delta\rho^n\text{ on }E_0\cap\{\rho<\rho_1\}.
    $$
    In particular, if $R_g\ge f$, then the condition $u<1$ can be replaced by $u\not\equiv 1$. 
\end{proposition}

We first observe that 

\begin{lemma}
    Assume $R_g=f\ge-n(n-1)$ on $E_0$. If $0\le u\le1$ solves \eqref{eq:Yamabe} on $E_0$, then $v=1-u\ge0$ satisfies
    $$
        \Delta_gv-nv\le 0 \text{ on }E_0.
    $$
\end{lemma}
\begin{proof}
    Note that for $v\ge 0$, $(1-v)^\alpha\ge 1-\alpha v$. Hence, 
    \begin{align*}
        0&=-c_n\Delta_g(1-v)+R_g(1-v)-f(1-v)^\alpha\\
        &\ge c_n\Delta_gv+f(1-v)-f(1-\alpha v)\\
        &=c_n\Delta_gv+fv(\alpha-1)\\
        &\ge c_n\Delta_gv-n(n-1)v(\alpha-1)=c_n(\Delta_g-n)v.
    \end{align*}
    The last step is due to $c_n=\frac{4(n-1)}{n-2}$, $\alpha=\frac{n+2}{n-2}$.
\end{proof}

\begin{proof}[Proof of Proposition \ref{prop:gap}]
    We first construct a sub-solution to the equation $\Delta_gv-nv=0$. Let $w=\rho^n(1+\rho)$, and the calculation in Lemma \ref{lem:YamabeAH} gives 
    $$
        \Delta_g w-nw=(n+2)\rho^{n+1}+O(\rho^{n+2}).
    $$
    Then there exists $\rho_1\in(0,\rho_0)$ such that $\Delta_gw-nw>0$ on $E_0\cap\{\rho<\rho_1\}$. We can assume that $E_0\cap\{\rho=\rho_1\}$ is smooth. 

    Since $v=1-u>0$, we can find $\delta_1>0$ such that
    $$
        v-\delta_1w>0 \text{ on }E_0\cap\{\rho=\rho_1\}.
    $$
    Then the maximum principle (for linear equations) implies that $v-\delta_1w>0$ on $E_0\cap\{\rho<\rho_1\}$, and thus 
    $$
        u=1-v<1-\delta_1w<1-\delta_1\rho^n
    $$
    on $E_0\cap\{\rho<\rho_1\}$. 

    If $R_g\ge f$, then the strong maximum principle (Lemma \ref{lem:sMaxPr-2}) shows that we only need to assume that $u\not\equiv 1$, which implies that $u<1$ on $E_0$. 
\end{proof}

\subsection{Tangential smoothness of the solution}\label{sec:smoothness}

The main result in this subsection is the following:
\begin{proposition}\label{prop:tangential smoothness}
    Let $u$ be the solution to \eqref{eq:Yamabe} obtained in Proposition \ref{prop:YamabeAH} or Proposition \ref{prop:YamabeArbitrary}. Then near $\partial X$, we can write 
    $$
        u(\theta,\rho)=1+A(\theta)\rho^n+B(\theta,\rho),
    $$
    where $A\in C^\infty(\partial X)$ and $B\in C_{n+1}^\infty(E_0)$. Here and in the sequel, we use coordinates $(\theta,\rho)$, $\theta\in \partial X$ to denote a point in $E_0$ 
\end{proposition}
Here, we use the notation $C_{\delta}^\infty(E_0)$ to denote the space of smooth functions $u$ on $E_0$ such that $\rho^{-\delta}|\nabla^ku|_g$ is bounded as $\rho\to 0$, for all $k\in\mathbb{Z}_{\ge 0}$.

To prove this, we start with an observation. Let $v=u-1$. Then $v$ satisfies the equation
\begin{equation}\label{eq:Yamabe v}
    -c_n(\Delta_gv-nv)=-R_g(v+1)+R_g(v+1)^\alpha+\frac{4n(n-1)}{n-2}v.
\end{equation}
In this subsection, we will denote 
$$
    A_n=\Delta_g-n, \quad F(p,v)=c_n^{-1}R_g(p)(v+1)-c_n^{-1}R_g(p)(v+1)^\alpha-nv.
$$
Then $(A_nv)(p)=F(p,v)$, or simply $A_nv=F(v)$ if the point $p$ is understood.

\begin{lemma}
    If $v\in C^k_\delta(E_0)$ for some $\delta>0$ and an integer $k\ge 0$, then $F(p,v(p))\in C^k_\delta(E_0)$.
\end{lemma}
\begin{proof}
    We can write $R_g=-n(n-1)+\rho^nR$ for some $R$ smooth up to $\partial X$, and 
    $(v+1)^\alpha-(\alpha v+1)=\varphi v^2$ for some $\varphi$ smooth on $[0,\infty)$.
    Then 
    $$
        F(p,v)=-\frac{\rho^nRv}{n-1}-c_n^{-1}R_g(p)\varphi v^2.
    $$
    If $v\in C^k_\delta(E_0)$, then
    $$
        \partial_\rho^i F(p,v(p))=\sum_{r=0}^i\binom{i}{r}\partial_\rho^r\left(-\frac{\rho^nR}{n-1}-c_n^{-1}R_g(p)\varphi v\right)\partial_\rho^{i-r}v=\sum_{r=0}^i O(\rho^{\delta-r})=O(\rho^{\delta-i}).
    $$
    Hence, $F(p,v(p))\in C^k_\delta(E_0)$.
\end{proof}

\begin{proof}[Proof of Proposition \ref{prop:tangential smoothness}]
    By Schauder estimates, we conclude that $v\in C^{k+2}_\delta(E_0)$, and thus $v\in C^\infty_\delta(E_0)$. The argument in \cite[\S4]{ACF92} then shows that $v$ is polyhomogeneous. In particular, we have the desired expansion of $u$ near $\partial X$.
\end{proof}

\begin{lemma}\label{lem:AH coordinates}
    If a smooth metric $g_1$ on $E_0$ satisfies 
    $$
        g_1=\sinh^{-2}\rho(d\rho^2+\gamma_\rho), \quad \gamma_\rho=\gamma_{\mathrm{std}}+\frac{\rho^n}{n}h_1+A_1\rho^n d\rho^2+O(\rho^{n+1}),
    $$
    where $h_1$ is a symmetric $2$-tensor on $\partial X$, and $A_1$ is a smooth function on $\partial X$, then there exists a defining function $\tilde\rho$ such that 
    $$
        g_1=\sinh^{-2}\tilde\rho(d\tilde\rho^2+\tilde\gamma_{\tilde\rho}), \quad \tilde\gamma_{\tilde\rho}=\gamma_{\mathrm{std}}+\frac{\tilde\rho^n}{n}\tilde h_1+O(\tilde\rho^{n+1}),
    $$
    where $\tilde{h}_1=h_1+A_1\gamma_{\mathrm{std}}$. In particular, $g_1$ is AH as in Definition \ref{def:AH mfd}. 
\end{lemma}
\begin{proof}
    Let $r=2\frac{\cosh\rho-1}{\sinh\rho}$. Then $r$ is a geodesic defining function for the AH end, i.e.,
    $$
        g_1=r^{-2}(dr^2+\gamma'_r)
    $$
    where $\{\gamma'_r\}$ is a family of metrics on $\partial X$:
    $$
        \gamma'_r=\left(1-\frac{r^2}{4}\right)^2\gamma_{\mathrm{std}}+\frac{r^n}{n}h_1+A_1r^{n}dr^2+O(r^{n+1}),
    $$
    where $O(r^{n+1})\in C^\infty_{n+1}(E_0)$. Let $\tilde{r}=\tilde{r}(\theta,r)$ be the solution to the equation
    $$
        \begin{cases}
            \tilde{r}^{-1}\partial_r\tilde{r}=r^{-1}\sqrt{1+A_1r^{n}}, \\
            \tilde{r}(\theta,0)=0. 
        \end{cases}
    $$
    Then we have 
    $$
        \tilde{r}(\theta,r)=r+\frac{A_1r^{n+1}}{2n}+O(r^{n+2}),
    $$
    for some $O(r^{n+2})\in C^\infty_{n+2}(E_0)$. A direct calculation shows that 
    $$
        g_1=\tilde r^{-2}(d\tilde{r}^2+\tilde\gamma'_{\tilde{r}}), \quad \tilde{\gamma}'_{\tilde{r}}=\left(1-\frac{\tilde{r}^2}{4}\right)^2\gamma_{\mathrm{std}}+\frac{\tilde{r}^n}{n}(h_1+A_1\gamma_{\mathrm{std}})+O(\tilde{r}^{n+1}).
    $$
    Hence, $\tilde{h}_1=h_1+A_1\gamma_{\mathrm{std}}$. We can then use $\tilde{r}=2\frac{\cosh\tilde{\rho}-1}{\sinh\tilde{\rho}}$ to get the desired formula. 
\end{proof}

As a consequence, we have 

\begin{proposition}[{cf.\ \cite[Lemma 6.5]{BQ08}}]\label{prop:mass aspect change}
    Let $u$ be a smooth positive function on $E_0$ such that in a collar neighborhood of $\partial X$, we can write
    $$
        u(\theta,\rho)=1+A(\theta)\rho^n+B(x,\rho),
    $$
    where $A\in C^\infty(\partial X)$ and $B\in C_{n+1}^\infty(E_0)$. Then the metric $\tilde{g}=u^{\frac{4}{n-2}}g$ is asymptotically hyperbolic and has mass aspect
    $$
        \tilde{h}=h+\frac{4(n+1)}{n-2}A\gamma_{\mathrm{std}},
    $$
    where $h$ is the mass aspect of $g$. 
\end{proposition}
\begin{proof}
    We have 
    $$
        \tilde{g}=u^{\frac{4}{n-2}}g=\sinh^{-2}\rho\left(d\rho^2+\gamma_{\mathrm{std}}+\frac{\rho^n}{n}\left(h+\frac{4nA}{n-2}\gamma_{\mathrm{std}}\right)+\frac{4A}{n-2}\rho^nd\rho^2+O(\rho^{n+1})\right).
    $$
    Hence, applying Lemma \ref{lem:AH coordinates} with 
    $$
        h_1=h+\frac{4nA}{n-2}\gamma_{\mathrm{std}}, \quad A_1=\frac{4A}{n-2},
    $$
    we see that $\tilde{g}$ is AH with mass aspect 
    $$
        \tilde{h}=h+\frac{4(n+1)}{n-2}A\gamma_{\mathrm{std}},
    $$
    as needed. 
\end{proof}

\begin{corollary}\label{cor:mass change}
    If $A(\theta)\le 0$ for all $\theta\in\partial X$, then 
    $$
        \EADM(\tilde{g})-|\PADM(\tilde{g})|\le \EADM(g)-|\PADM(g)|.
    $$
    If, in addition, $A(\theta_0)<0$ at some $\theta_0\in\partial X$, then the inequality is strict. 
\end{corollary}
\begin{proof}
    Note that $\operatorname{tr}_{\gamma_{\mathrm{std}}}\tilde{h}=\operatorname{tr}_{\gamma_{\mathrm{std}}}h+\frac{4(n^2-1)}{n-2}A$. Hence, we have
    \begin{align*}
        \EADM(\tilde{g})&=\int_{S^{n-1}}\operatorname{tr}_{\gamma_{\mathrm{std}}}\tilde{h}\,d\mu_{\gamma_{\mathrm{std}}}=\int_{S^{n-1}}\operatorname{tr}_{\gamma_{\mathrm{std}}}h\,d\mu_{\gamma_{\mathrm{std}}}+\frac{4(n^2-1)}{n-2}\int_{S^{n-1}}A\,d\mu_{\gamma_{\mathrm{std}}}\\
        &=\EADM(g)+\frac{4(n^2-1)}{n-2}\int_{S^{n-1}}A\,d\mu_{\gamma_{\mathrm{std}}},\\
        |\PADM(\tilde{g})|&=\left|\int_{S^{n-1}}x\operatorname{tr}_{\gamma_{\mathrm{std}}}\tilde{h}\,d\mu_{\gamma_{\mathrm{std}}}\right|=\left|\int_{S^{n-1}}x(\operatorname{tr}_{\gamma_{\mathrm{std}}}h+\frac{4(n^2-1)}{n-2}A)d\mu_{\gamma_{\mathrm{std}}}\right|\\
        &\ge \left|\int_{S^{n-1}}x\operatorname{tr}_{\gamma_{\mathrm{std}}}h\,d\mu_{\gamma_{\mathrm{std}}}\right|-\frac{4(n^2-1)}{n-2}\int_{S^{n-1}}|A|\,d\mu_{\gamma_{\mathrm{std}}}\\
        &=|\PADM(g)|+\frac{4(n^2-1)}{n-2}\int_{S^{n-1}}A\,d\mu_{\gamma_{\mathrm{std}}}.
    \end{align*}
    Note that the triangle inequality is strict unless $A\equiv 0$. This completes the proof. 
\end{proof}

\section{Manifolds with an AH end and singularities}\label{sec:AH mfd}

In this section, $(M,d,\mu)$ is an almost manifold with singularity $\mathcal{S}$ as in the introduction. Take a compact set $K\supset\mathcal{S}$ such that 
$$
    M\setminus K=E_0\cup\bigcup_{i\in I}E_i,
$$
where $E_0$ is an AH end with conformal infinity $\partial X$. The other ends may or may not be AH, and we allow that $I=\varnothing$. (See Figure \ref{fig:ends}.) 
Choose a compact domain $K'$ containing $K$ in its interior and with smooth boundary in $\mathcal R$. Set $\Sigma_i=\partial K'\cap E_i$ and $M'=K'\cup E_0$.
 (See Figure \ref{fig:cutoff}.) By Lemma 2.17 in \cite{BHHSZ2026}, we can always assume that $K'\setminus\mathcal{S}$ is connected. 

\subsection{Solving the Yamabe equation on the regular part}\label{sec:Yamabe eqn reg part}

Let $k=\dim_H(\mathcal{S})<n-2$. Recall that from \cite{HSY25, BHHSZ2026}, since $M$ is asymptotically hyperbolic, it is non-parabolic, and thus there exists a harmonic function 
$$
    G:M\setminus\mathcal{S}\to (0,+\infty)
$$
such that $G(x)\to +\infty$ as $x\to \mathcal{S}$. If we set 
$$
    r(x)=d(x,\mathcal{S}),
$$
then there exists $C_1\in(0,\infty)$ such that 
\begin{equation}\label{eq:G lower bound}
    G(x)\ge C_1r^{2+k-n},
\end{equation}
in a neighborhood of $\mathcal{S}$. We take a sequence of regular values $\{a_j:j\in \mathbb{Z}_{>0}\}$ of $G^{\frac{1}{2+k-n}}$ and define 
\begin{equation}\label{eq:neighborhood}
    U_j=\{x\in M: G^{\frac{1}{2+k-n}}(x)<a_j\}.
\end{equation}
Then $\{U_j:j\in \mathbb{Z}_{>0}\}$ is a sequence of open neighborhoods of $\mathcal{S}$, and $\partial U_j$ is smooth for each $j$. Moreover, we assume $a_j\to 0$ and thus $\mathcal{S}=\bigcap_{j=1}^\infty U_j$. 

We may assume that the defining function $\rho$ is globally defined on $M$ and continuous. We choose a small $\rho_0>0$ such that $\{\rho\le \rho_0\}\cap K=\varnothing$, and choose a function $f\in C^\infty(M)$ such that 
\begin{enumerate}
    \item $-n(n-1)\le f\le \min\{R_g,-1\}$,
    \item $f=-n(n-1)$ on $K$, and
    \item $f=R_g$ on $E_0\cap\{\rho<\rho_0\}$.
\end{enumerate}

Consider the following Yamabe equation on $M'\setminus \bar U_j$:
$$
    L_gu=-c_n\Delta_gu+R_gu-fu^{\alpha}=0 \text{ on }M'\setminus \bar U_j
$$
with the boundary conditions
\begin{itemize}
    \item $u(x)\to 1$ as $x\to\partial X$,
    \item $u(x)=1$ on $\Sigma_i$ for each $i\in I$, 
    \item $u(x)\dist_g(x,\partial U_j)^{\frac{n-2}{2}}\to 1$ as $x\to\partial U_j$.
\end{itemize}
We can assume that $U_j\subset K$ is true for any $j$. If $u$ is a positive smooth solution to \eqref{eq:Yamabe}, then the scalar curvature of $\tilde{g}=u^{\frac{4}{n-2}}g$ is $f\ge -n(n-1)$:
\begin{equation}\label{eq:comformal sc}
        R_{\tilde{g}}= u^{-\alpha}(R_{g}u-c_n \Delta_g u)=f.
\end{equation}
By Lemma \ref{prop:YamabeAH}, Lemma \ref{lem:YamabeAH} and Corollary \ref{cor:AH-coefficient-convergence}, we have

\begin{proposition}\label{prop:uj}
    Let $r_j(x)=\dist_g(x,\partial U_j)$. There exists a unique positive function $u_j\in C^\infty(M'\setminus\bar U_j)$ such that 
    \begin{itemize}
        \item $L_gu_j=0$ on $M'\setminus\bar U_j$, 
        \item $u_j=r_j^{\frac{2-n}{2}}(1+a_jr_j)$ for some bounded function $a_j$ on $M'\setminus U_j$, smooth near $\partial U_j$,
        \item $u_j=1$ on $\Sigma_i$ for each $i\in I$, and
        \item $u_j=1+A_j\rho^n+O(\rho^{n+1})$ as $\rho\to 0$, where $A_j\in C^\infty(\partial X)$.
    \end{itemize}
    Moreover, there is a constant $A$ depending only on the geometry of the AH end $E_0$, such that $|A_j|<A$.
\end{proposition}

\begin{corollary}\label{cor:comformalmassaspects}
    The metric $\tilde{g}_j=u_j^{\frac{4}{n-2}}g$ is complete, asymptotically hyperbolic on $M'\setminus \bar U_j$, and has scalar curvature $R_{\tilde{g}_j}=f\ge -n(n-1)$. The mass aspect of $\tilde{g}_j$ is given by
    \begin{equation}\label{eq:mass aspect}
        \tilde h_j=h+\frac{4(n+1)}{n-2}A_j\gamma_{\mathrm{std}}.
    \end{equation}
\end{corollary}
\begin{proof}
    The scalar curvature of $\tilde{g}_j$ is computed in \eqref{eq:comformal sc}. The asymptotically hyperbolic property and the mass aspect follow from the expansion of $u_j$ near $\partial X$, cf.\ Propositions \ref{prop:tangential smoothness}, \ref{prop:mass aspect change}. The completeness of $\tilde{g}_j$ follows from the asymptotic behavior of $u_j$ near $\partial U_j$, cf.\ \cite[Lemma 5.2]{AMnoncpt88} for example. 
\end{proof}

By Corollary \ref{cor:convergence}, there is a subsequence of $\{u_j\}$ that converges smoothly on compact sets to a solution $u_\infty\in C^\infty(M'\setminus\mathcal{S})$ to the equation. We summarize the properties of $u_\infty$ in the following proposition:

\begin{proposition}\label{prop:u-limit}
    There is a subsequence of $\{u_j\}$, which by abuse of notation is still denoted by $\{u_j\}$, such that as $j\to \infty$, $u_j\to u_\infty$ smoothly on compact subsets of $M'\setminus \mathcal S$. As a result, $u_\infty\in C^\infty(M'\setminus \mathcal S)$ satisfies 
    \begin{enumerate}
        \item $L_gu_\infty=0$ and $u_\infty>0$ on $M'\setminus \mathcal S$,
        \item $u_\infty(x)=1$ on $\Sigma_i$ for each $i\in I$, and 
        \item $u_\infty(x)\to 1$ as $x\to\partial X$.
    \end{enumerate}
    Moreover, in a collar neighborhood of $\partial X$, we have 
    \begin{equation}\label{eq:u-infty expansion}
        u_\infty(x,\rho)=1+A_\infty(x)\rho^n+B_\infty(x,\rho),
    \end{equation}
    where $A_\infty\in C^\infty(\partial X)$ and $B_\infty\in C_{n+1}^\infty(E_0)$.
\end{proposition}
\begin{proof}
    This is due to Corollary \ref{cor:convergence}, Lemma \ref{lem:sMaxPr-1}, and Proposition \ref{prop:tangential smoothness}.
\end{proof}

\subsection{Bounding the limit under the Ricci curvature condition}\label{sec:bound under Ricci cond}
The next goal is to show that $u_\infty$ is bounded. Let $\Omega$ be a connected bounded open neighborhood of $\mathcal{S}$, and $r(x)=d(x,\mathcal{S})$. 

\begin{lemma}\label{lem:gradient estimate}
    Let $G:M\setminus\mathcal{S}\to (0,\infty)$ be a harmonic function. Assume that for some $U\Supset \Omega$ with 
    $$
        \sup_{x\in\Omega}d(x,\mathcal{S})<d(\Omega,\partial U),
    $$ 
    the Ricci curvature condition
    $$
        \Ric_g\ge -C_2'r^{-2} \text{ on }U\cap \mathcal{R},
    $$
    holds. Then
    \begin{equation}\label{eq:gradient estimate}
        \frac{|\nabla G|}{G}\le C_2r^{-1} \text{ on }\Omega\cap \mathcal{R},
    \end{equation}
    for some constant $C_2\in(0,\infty)$ depending only on $n$ and $C_2'$. 
\end{lemma}
\begin{proof}
    Fix $x\in\Omega\setminus\mathcal{S}$ and let $t=\frac{1}{2}r(x)$. Then the gradient estimate for positive harmonic functions (cf.\ \cite{Yau75,CY75}) implies that, on the ball $B_t(x)\subset U$, we have 
    $$
        |\nabla G|\le C_2''t^{-1}G,
    $$
    where $C_2''\in(0,\infty)$ depends only on $n$ and $C_2'$. In particular, we have $|\nabla G|\le C_2r^{-1}G$, and thus the desired estimate holds.
\end{proof}

In this subsection, we assume that for some neighborhood $U$ of $\mathcal{S}$, the Ricci curvature condition 
$$
    \Ric_g\ge -Cr^{-2} \text{ on }U\cap \mathcal{R},
$$
holds. Then for $\Omega$ as above, we may choose $U_1$ so that $U_1\Supset\Omega\cup U$, with 
$$
    \sup_{x\in\Omega}d(x,\mathcal{S})<d(\Omega,\partial U_1).
$$ 
On $U_1\setminus U$, there is an upper bound 
$$
    C'=\sup_{x\in U_1\setminus U}-r^2\Ric_g,
$$
and thus on $U_1$, Lemma \ref{lem:gradient estimate} holds for $C_2'=\max\{C,C'\}$. Hence, \eqref{eq:gradient estimate} holds for some $C_2$, which depends also on the set $\Omega$. 

We first prove a rough estimate for $u_\infty$:
 
\begin{proposition}\label{prop:rough estimate}
    If $u$ solves \eqref{eq:Yamabe} on $M\setminus\mathcal S$, then 
    $$
        u\le\Lambda G^{\frac{n-2}{2(n-k-2)}} \text{ on }\Omega\setminus\mathcal{S},
    $$
    for some constant $\Lambda\in(0,\infty)$ depending on $G$, $\Omega$ and $\sup_{\partial\Omega} u$.
\end{proposition}
\begin{proof}
    Let $U_j=\{x\in M: G^{\frac{1}{2+k-n}}(x)<a_j\}$ be as in the previous section. Consider the function 
    $$
        v_j=\Lambda (G^{\frac{1}{2+k-n}}-a_j)^{\frac{2-n}{2}} \text{ on }\Omega\setminus\bar{U}_j,
    $$
    where $\Lambda$ is a large constant to be determined. A direct calculation shows that
    $$
        \Delta_g v_j=\tfrac{\Lambda(n-2)n}{4(n-k-2)^2}(G^{\frac{1}{2+k-n}}-a_j)^{-\frac{2+n}{2}}\frac{|\nabla G|^2}{G^2}G^{\frac{2}{2+k-n}}\left(1-\tfrac{2(n-k-1)}{n}(G^{\frac{1}{2+k-n}}-a_j)G^{\frac{1}{n-k-2}}\right).
    $$
    By Lemma \ref{lem:gradient estimate}, $G^{\frac{1}{2+k-n}}>a_j$ on $\Omega\setminus\bar{U}_j$, and \eqref{eq:G lower bound}, we have 
    \begin{align*}
        \Delta_g v_j&\le \frac{\Lambda(n-2)n}{4(n-k-2)^2}(G^{\frac{1}{2+k-n}}-a_j)^{-\frac{2+n}{2}}\frac{C_2^2}{r^2}G^{\frac{2}{2+k-n}}\\
        &\le \frac{\Lambda(n-2)n}{4(n-k-2)^2}(G^{\frac{1}{2+k-n}}-a_j)^{-\frac{2+n}{2}}\frac{C_2^2}{C_1^{\frac{2}{n-k-2}}}.
    \end{align*}
    Since $\Omega$ is bounded, we have $G^{\frac{1}{2+k-n}}\le K<\infty$ on $\Omega$. Thus, 
    $$
        L_gv_j\ge \Lambda(G^{\frac{1}{2+k-n}}-a_j)^{-\frac{2+n}{2}}(-\tilde{C}-f\Lambda^{\frac{4}{n-2}}-n(n-1)K^2),
    $$
    where $\tilde{C}=\frac{c_n(n-2)n}{4(n-k-2)^2}C_2^2C_1^{\frac{2}{2+k-n}}$. We choose $\Lambda$ sufficiently large, i.e., 
    $$
        \Lambda\ge \left(\tilde C+n(n-1)K^2\right)^{\frac{n-2}{4}},
    $$
    and then 
    $$
        L_gv_j\ge 0 \text{ on }\Omega\setminus\bar{U}_j.
    $$
    Moreover, we can make 
    $$
        \Lambda\ge K^{\frac{n-2}{2}}\sup_{\partial\Omega}u,
    $$
    and thus $v_j\ge u$ on $\partial\Omega$. By the comparison lemma \ref{lem:comparison}, we have $u\le 1+v_j$ on $\Omega\setminus\bar{U}_j$. Lastly, let $j\to\infty$. We have $u\le 1+\Lambda G^{\frac{n-2}{2(n-k-2)}}$ on $\Omega\setminus\mathcal{S}$. Since $G\to +\infty$ as $x\to \mathcal{S}$, we can absorb the constant $1$ into $\Lambda$ and get the desired estimate.
\end{proof}

Now we can show that $u_\infty$ is bounded:

\begin{proposition}\label{prop:bound2}
    If $u$ solves \eqref{eq:Yamabe} on $M\setminus\mathcal S$, and if $k=\dim_H(\mathcal{S})<\frac{n-2}{2}$, then $u$ is bounded on $\Omega\setminus\mathcal{S}$. Indeed, $u\le 2+\sup_{\partial\Omega}u$ on $\Omega\setminus\mathcal{S}$.
\end{proposition}
\begin{proof}
    By Proposition \ref{prop:rough estimate}, we have $u\le \Lambda G^{\frac{n-2}{2(n-k-2)}}$ on $\Omega\setminus\mathcal{S}$. Note that $\frac{n-2}{2(n-k-2)}<1$ if $k<\frac{n-2}{2}$. We claim that $u\le 2+\sup_{\partial\Omega}u$ on $\Omega\setminus\mathcal{S}$. 
    
    If not, then we can find $x_0\in\Omega\setminus\mathcal{S}$ such that $u(x_0)>2+\sup_{\partial\Omega}u$. Let 
    $$
        v_\epsilon=1+\epsilon (1+G)+\sup_{\partial \Omega}u.
    $$
    Since $v_\epsilon>1$ and $R_g\ge f$, we have
    $$
        L_gv_\epsilon=-c_n\Delta_g v_\epsilon+R_gv_\epsilon-fv_\epsilon^{\alpha}\ge fv_\epsilon(1-v_\epsilon^{\frac{4}{n-2}})\ge 0.
    $$
    We can find $j>0$ such that $x_0\in \Omega\setminus\bar U_j$, and that $u<v_\epsilon$ on $\partial U_j$, which is possible due to 
    $$
        u\le\Lambda G^{\frac{n-2}{2(n-k-2)}}<\Lambda(1+G)^{\frac{n-2}{2(n-k-2)}}<\epsilon(1+G)<v_\epsilon.
    $$
    Since $u\le v_\epsilon$ on $\partial\Omega$, we have $u\le 1+v_\epsilon$ on $\Omega\setminus\bar{U}_j$ by the comparison lemma \ref{lem:comparison}. Hence, 
    $$
        u(x_0)\le 2+\epsilon+\epsilon G(x_0)+\sup_{\partial \Omega}u.
    $$
    But this is true for any $\epsilon>0$, and thus $u(x_0)\le 2+\sup_{\partial\Omega}u$, which is a contradiction.
\end{proof}

Propositions \ref{prop:rough estimate} and \ref{prop:bound2} show that $u_\infty$ is bounded on any bounded neighborhood $\Omega$ of $\mathcal{S}$. Away from $\mathcal{S}$, the function $u_\infty$ is regular by Proposition \ref{prop:u-limit} and satisfies $u_\infty(x)\to 1$ as $x\to\partial X$; hence $u_\infty$ is bounded on all of $M'\setminus\mathcal{S}$. Applying Proposition \ref{prop:upper bound}, we obtain the following uniform estimate.

\begin{proposition}\label{prop:bound for u_infty}
    The limit $u_\infty$ obtained in Proposition \ref{prop:u-limit} is bounded on $M\setminus\mathcal{S}$ from above by $1$. In particular, $A_\infty\le 0$ in the expansion \eqref{eq:u-infty expansion}.
\end{proposition}

\subsection{Estimate of Green functions with a single pole}\label{sec:Green fn}

In this subsection, we assume that $(M,d,\mu)$ is an $RCD(K,N)$ space. Since $M$ is asymptotically hyperbolic, we have $K<0$. We present a gradient estimate of the Green function $G(x):=G_p(x)$, where $p$ is any fixed point in $M$. More precisely, we fix $q\in E_0$ in the AH end. We first show that 

\begin{lemma}\label{lem:singular harmonic}
    Let $p,q\in M$ and $r_0>0$. Let $\Omega$ be a bounded open subset of $M$ containing $\mathcal{S}\cup B_{2r_0}(p)\cup\{q\}$. There exists a unique positive function $u\in W_{\mathrm{loc}}^{1,2}(\Omega\setminus\{p\})$ with
    \begin{itemize}
        \item $\Delta_g u=0$ on $\Omega\setminus\{p\}$,
        \item $u=0$ on $\partial\Omega$ in the trace sense,
        \item $u(q)=1$, and
        \item $u(x)\to \infty$ as $x\to p$.
    \end{itemize} 
    Moreover, the following holds
    $$
        u(x)\ge C_1d(x,p)^{2-n},
    $$
    and
    $$
        |\nabla u|(x) \leq C_2 u(x)^{\frac{n-1}{n-2}},
    $$
    for all $x\in B_{r_0}(p)\setminus\{p\}$. The positive constants depend only on the metric measure space $(M,d,\mu)$, the points $p,q$, and the number $r_0$, and the set $\Omega$, not on $x$.
\end{lemma}
\begin{proof}
    Let $r_i\searrow 0$ and let $v_i$ be a solution to
    $$
        \begin{cases}
            \Delta v_i=0, & \text{in }\Omega\setminus \bar B_{r_i}(p),\\
            v_i=0, & \text{on }\partial \Omega,\\
            v_i=1, & \text{on }\partial B_{r_i}(p).
        \end{cases}
    $$
    Here, we assume that $r_i<d(p,\partial\Omega\cup\{q\})$. We know that $v_i>0$ on $\Omega\setminus \bar B_{r_i}(p)$ by the maximum principle. Let 
    $$u_i(x)=\frac{v_i(x)}{v_i(q)}.$$
    Then $u_i$ is a positive harmonic function on $\Omega\setminus \bar B_{r_i}(p)$ with $u_i(q)=1$. By Harnack's inequality, there is a subsequence of $\{u_i\}$ that converges to a positive harmonic function $u$ on $\Omega\setminus\{p\}$ with $u(q)=1$, and $u(x)\to \infty$ as $x\to p$, cf.\ \cite[Proposition 2.18]{HSY25}, where the uniqueness of such $u$ is also established. Thus, $u$ is uniquely determined by $p,q$ and $\Omega$. 

    Proposition 2.20 in \cite{HSY25} shows that 
    $$
        u(x)\ge C_1d(x,p)^{2-n},
    $$
    for all $x\in B_{r_0}(p)\setminus\{p\}$. The constant $C_1$ depends only on the metric measure space $(M,d,\mu)$, the number $r_0$, and the quantity 
    $$
        C(u):=\int_{\partial\Omega}|\nabla u|d\sigma>0.
    $$
    This is exactly the first inequality we want. For completeness, we sketch the proof here and explain the dependence of $C_1$ on $r_0$. Recall that the variational capacity of a set $E$ with respect to a bounded set $A$ is defined by
    $$
        \operatorname{Cap}(E,A)=\inf\left\{\int_M|\nabla w|^2\,d\mu:w\in W^{1,2}(M), w\ge 1 \text{ on }E, w=0 \text{ on } M\setminus A \right\}, \quad E\subset A\subset M.
    $$
    Let $t=\max_{x\in\partial B_r(p)}u(x)$, where $r\in(0,r_0)$. By \cite[Proposition 2.19]{HSY25}, and the fact that $u$ is harmonic, and that $\Omega_t:=\{x\in\Omega:u(x)\ge t\}$ is contained in $B_r(p)$, we have
    \begin{equation}\label{eq:cap-1}
        C(u)t^{-1}=\Capacity(\Omega_t,\Omega)\le \Capacity(B_r(p),\Omega)\le \Capacity(B_r(p),B_{2r_0}(p)).
    \end{equation}
    The last term can be bounded from above by the energy of the following test function:
    $$
        w(x)=\begin{cases}
            1, & x\in B_r(p),\\
            (r^{2-n}-(2r_0)^{2-n})^{-1}(r_p(x)^{2-n}-(2r_0)^{2-n}), & x\in B_{2r_0}(p)\setminus B_r(p),\\
            0, & x\in M\setminus B_{2r_0}(p),
        \end{cases}
    $$
    where $r_p(x)=d(x,p)$. A direct calculation shows that
    \begin{equation}\label{eq:cap-2}
        \begin{split}
            \Capacity(B_r(p), B_{2r_0}(p))&\leq \int_{B_{2r_0}(p)\setminus B_r(p)} |\nabla w|^2\, d\mu\leq C r^{2n-4}\int_{B_{2r_0}(p)\setminus B_r(p)}r_p^{2(1-n)}(x)\, d\mu\\
            &\le Cr^{2n-4}\sum^{k_0}_{k=1}\int_{B_{2^{k}r}(p)\setminus B_{2^{k-1}r}(p)}r_p^{2(1-n)}(x)\, d\mu\\
            &\leq C r^{n-2}\sum^{k_0}_{k=1} 2^{-k(n-2)} 2^{2(n-1)}\\
            &\leq C r^{n-2}\sum^{\infty}_{k=1} 2^{-k(n-2)} 2^{2(n-1)}\\
            &\leq C r^{n-2}.
        \end{split}
    \end{equation}
    Here, $C$ may vary from line to line, which depends only on the number $r_0$, the point $p$, and the metric measure space; $k_0$ is a positive integer determined by the ratio of $r$ and $r_0$. In the third line, we have used the local volume growth condition,
    $$
        \operatorname{Vol}(B_{2^k r}(p))\leq C (2^k r)^n.
    $$
    Combining \eqref{eq:cap-1} and \eqref{eq:cap-2}, we have
    $$
        t\ge C r^{2-n},
    $$
    where $C$ depends only on the metric measure space $(M,d,\mu)$, the points $p,q$, the set $\Omega$, and the number $r_0$. By Harnack's inequality, we have
    $$
        u(x)\ge C_1 d(x,p)^{2-n},
    $$
    for all $x\in B_{r_0}(p)\setminus\{p\}$, and thus, the first inequality is proved. 

    Lastly, by Cheng-Yau's gradient estimate, cf.\ \cite[Theorem 1.2]{Jiang2014}, we have, for any $x\in \partial B_r(p)$, where $r\in(0,r_0)$,
    $$
        \frac{|\nabla u|(x)}{u(x)}\leq C r^{-1}.
    $$
    This constant $C$ depends on the metric measure space $(M,d,\mu)$, namely, the constants $K,N$ in the RCD condition. Thus, combining with the first inequality, we have
    $$
        |\nabla u|(x) \leq C r^{-1} u(x)\le CC_1^{\frac{1}{2-n}} u(x)^{\frac{n-1}{n-2}}=:C_2u(x)^{\frac{n-1}{n-2}}, 
    $$
    This proves the second inequality.
\end{proof}

\begin{remark}
    The normalization $u(q)=1$ in Lemma \ref{lem:singular harmonic} is not essential for the estimates, but the constants $C_1$ and $C_2$ will depend on the choice of $u$. Let $G$ be any singular harmonic function on $\Omega$ with pole $p$. The same proof shows that 
    $$
        G(x)\ge C_1 d(x,p)^{2-n}, \quad |\nabla G|(x)\le C_2 G(x)^{\frac{n-1}{n-2}},
    $$
    where $C_1=\lambda_1C(G)$ and $C_2=\lambda_2C_1^{1/(2-n)}$. Here, $\lambda_1,\lambda_2$ depend only on the metric measure space $(M,d,\mu)$, and 
    $$
        C(G)=\int_{\partial\Omega}|\nabla G|\,d\sigma>0.
    $$
\end{remark}

\begin{lemma}\label{lem:singular harmonic uniform bound}
    Let $r_0>0$ and $q\in\mathcal{R}$. Let $\Omega$ be a bounded open set containing $q$ and the $2r_0$-neighborhood of $\mathcal{S}$. Suppose that $\Omega=\Omega_0\subset\Omega_1\subset\cdots$ is an exhaustion of $M$ by bounded open sets. Let $u^{(i)}_a$ be the singular harmonic function on $\Omega_i$ with pole $a\in\mathcal{S}$, and $u^{(i)}_a(q)=1$, as in Lemma \ref{lem:singular harmonic}. Then there exist constants $C_1,C_2\in(0,\infty)$ that depend only on the point $q$, the number $r_0$, and the metric measure space $(M,d,\mu)$, such that for any $i\in\mathbb{Z}_{>0}$ and $a\in\mathcal{S}$, we have
    $$
        u^{(i)}_a(x)\ge C_1d(x,a)^{2-n}, \quad |\nabla u^{(i)}_a|(x)\le C_2 u^{(i)}_a(x)^{\frac{n-1}{n-2}},
    $$
    for all $x\in B_{r_0}(a)\setminus\{a\}$.
\end{lemma}
\begin{proof}
    By Lemma \ref{lem:singular harmonic}, we have
    $$
        u^{(i)}_a(x)\ge C_1(a,i)d(x,a)^{2-n}, \quad |\nabla u^{(i)}_a|(x)\le C_2(a,i) u^{(i)}_a(x)^{\frac{n-1}{n-2}},
    $$
    for all $x\in B_{r_0}(a)\setminus\{a\}$. We must show that $C_1(a,i)$ and $C_2(a,i)$ can be chosen independently of $a$ and $i$. By the proof of Lemma \ref{lem:singular harmonic}, we have $C_2=CC_1^{1/(2-n)}$, where $C$ depends only on the metric measure space $(M,d,\mu)$. Thus, it suffices to show that $C_1(a,i)$ can be chosen independently of $a$ and $i$. 

    \textbf{Step 1}: We show that $C_1(a,0)$ can be chosen to be independent of $a$. From the proof of Lemma \ref{lem:singular harmonic}, it suffices to show that 
    $$
        C(u^{(0)}_a)=\int_{\partial\Omega_0}|\nabla u^{(0)}_a|\,d\sigma
    $$
    can be bounded from below by a positive constant that is independent of $a$. If this is not the case, then there exists a sequence $\{a_j\}\subset\mathcal{S}$ and $x_j\in B_{r_0}(a_j)\setminus\{a_j\}$, such that 
    $$
        C(u^{(0)}_{a_j})<\frac{1}{j}.
    $$
    By passing to a subsequence, we may assume that $a_j\to a\in\mathcal{S}$ and that $u_{a_j}^{(0)}$ converges to a singular harmonic function $u^{(0)}_{a}$. This limit coincides with the one defined in Lemma \ref{lem:singular harmonic} by uniqueness. Moreover, the convergence is smooth away from $\mathcal{S}$, and thus in a neighborhood of $\partial\Omega_0$. Hence, 
    $$
        C(u^{(0)}_a)=\lim_{j\to\infty}C(u^{(0)}_{a_j})=0,
    $$
    which is a contradiction.

    \textbf{Step 2}: We show that for a sequence of singular harmonic functions $G^{(i)}_a$ on $\Omega_i$ with pole $a\in\mathcal{S}$, with the monotone property: 
    $$
        G^{(i)}_a\le G^{(i+1)}_a \text{ on }\Omega_i\setminus\{a\},
    $$
    the constant $C_1$ can be chosen independently of $i$. Moreover, $G^{(0)}_a$ can be normalized by setting $G^{(0)}_a(q)=1$. 
    
    Such $G^{(i)}_a$ exists by \cite[Proposition 2.23]{HSY25}, and the difference $G^{(i)}_a-G^{(0)}_a$ is a bounded function. By monotonicity, we have
    $$
        G^{(i)}_a(x)\ge G^{(0)}_a(x)=u^{(0)}_a(x)\ge C_1(a,0)d(x,a)^{2-n}.
    $$
    Hence, the constant can be chosen to be independent of $i$. 

    \textbf{Step 3}: We show that, for any $x\in\Omega_0\setminus\{a\}$, $G^{(0)}_a(x)$ is bounded from above by a positive constant that is independent of $a$. 

    This follows from the fact that $G^{(0)}_a(q)=1$ and Harnack's inequality. Indeed, we can connect $x$ and $q$ by a curve and cover it by small balls $\{B_k:k=1,\ldots,k_0\}$ on which Harnack's inequality holds:
    $$
        \sup_{B_k}G^{(0)}_a\le C\inf_{B_k}G^{(0)}_a.
    $$
    The constant $C$ depends only on the metric measure space $(M,d,\mu)$. Since $\mathcal{S}$ is compact, the number $k_0$ has a uniform upper bound that is independent of $a$. Thus, $G^{(0)}_a(x)\le C^{k_0}$ as required. 
    
    \textbf{Step 4}: We show that $G^{(i)}_a(q)$ can be bounded from above by a positive constant that is independent of $a$ and $i$. 

    We claim that
    $$
        \sup\{G^{(i)}_a(x):x\in\partial\Omega_0, a\in\mathcal{S}, i\in\mathbb{Z}_{>0}\}<\infty.
    $$
    If this is not the case, then we can find $a_i\in\mathcal{S}$ such that $a_i\to a$ and 
    $$
        s_i:=\sup\{G^{(i)}_{a_i}(x):x\in\partial\Omega_0\}\to \infty.
    $$
    Set $\varphi_i=s_i^{-1}G^{(i)}_{a_i}$. For any $x\in\Omega_0$, we have 
    $$
        G^{(i)}_{a_i}(x)\le G^{(0)}_{a_i}(x)+s_i,
    $$
    by applying the maximum principle to the difference $G^{(i)}_{a_i}-G^{(0)}_{a_i}$. Hence, 
    $$
        0\le \varphi_i(x)\le s_i^{-1}G^{(0)}_{a_i}(x)+1\to 1 \text{ as }i\to\infty,
    $$
    for any $x\in\Omega_0\cap\mathcal{R}$. Here, we have used the fact that $G^{(0)}_{a_i}(x)$ is bounded, proved in Step 3. Thus, by passing to a subsequence, we have $\varphi_i\to\varphi$ where $\varphi$ is bounded on $\Omega_0$. On the other hand, $\varphi$ attains its maximum at the removable singularity $a$, and thus $\varphi$ is constant by the strong maximum principle (cf.\ \cite[Proposition 2.23]{HSY25}). Since $\sup_{x\in\partial\Omega_0}\varphi_i(x)=1$, we have $\varphi\equiv 1$. 

    By comparison with a harmonic function $f$ on the AH end $E_0$ that approaches $0$ as $x\to\partial X$, we have $\varphi_i\le f$ on $\Omega_i\cap E_0$. Hence, $\varphi\to 0$ as $x\to\partial X$, which is a contradiction. This proves the claim. 

    \textbf{Step 5}: Complete the proof. 

    Note that the uniqueness in Lemma \ref{lem:singular harmonic} implies that 
    $$
        u^{(i)}_a(x)=\frac{G^{(i)}_a(x)}{G^{(i)}_a(q)}.
    $$
    By Steps 1 and 2, 
    $$
        G^{(i)}_a(x)\ge C_1d(x,a)^{2-n}
    $$
    for some $\tilde C_1$ independent of $a$ and $i$. By Step 4, 
    $$
        G^{(i)}_a(q)\le A
    $$
    for some $A$ independent of $a$ and $i$. Thus, 
    $$ 
        u^{(i)}_a(x)\ge C_1d(x,a)^{2-n},
    $$
    where $C_1=A^{-1}\tilde C_1$ depends on the metric measure space $(M,d,\mu)$, the number $r_0$, the point $q$, but is independent of $a$ and $i$. 
\end{proof}

\begin{proposition}\label{prop:estimate Green's fun1}
    Let $q\in E_0$ and $r_0>0$ be fixed. For any $a\in\mathcal{S}$, there exists a positive harmonic function $G_a$ on $M\setminus\{a\}$ such that 
    $$
        \lim_{x\to\partial X}G_a(x)=0,\quad \lim_{x\to a}G_a(x)=\infty, \quad G_a(q)=1.
    $$
    Moreover, there exist $C_1,C_2\in (0,\infty)$ that depend only on the point $q$, the number $r_0$ and the metric measure space $(M,d,\mu)$, and there exists a function $C:\Omega\to(0,\infty)$, such that
    \begin{enumerate}
        \item\label{item:Green lb} $G_a(x)\ge C_1d(x,a)^{2-n}$ for any $x\in B_{r_0}(a)$, 
        \item\label{item:Green grad} $|\nabla G_a|(x)\le C_2G_a(x)^{\frac{n-1}{n-2}}$ for any $x\in B_{r_0}(a)$, and 
        \item\label{item:Green ub} $G_a(x)\le C(x)$ for any $x\in\mathcal{R}$. 
    \end{enumerate}
    Here, $C_1,C_2$, and $C$ are independent of $a\in\mathcal{S}$. 
\end{proposition}

\begin{proof}
    Let $\Omega=\Omega_0\subset\Omega_1\subset\cdots$ be a compact exhaustion of $M$. We set $G_a$ to be a subsequential limit of the normalized singular harmonic functions $u_i$ on $\Omega_i$ with pole $p$, defined in Lemma \ref{lem:singular harmonic}. Then $G_a$ satisfies \eqref{item:Green lb}, \eqref{item:Green grad} by Lemma \ref{lem:singular harmonic uniform bound}, and we have 
    $$
        \lim_{x\to a}G_a(x)=\infty, \quad G_a(q)=1.
    $$
    The fact that $\lim_{x\to\partial X}G_a(x)=0$ can be proved the same way as in \cite[Proposition 2.23]{HSY25}, since there exists a positive harmonic function on $E_0$ that approaches $0$ as $x\to\partial X$. It remains to show that \eqref{item:Green ub} holds. 

    Connect $x$ and $q$ by a curve and cover it by small balls $\{B_k:k=1,\ldots,k_0\}$ on which Harnack's inequality holds: 
    $$
        \sup_{B_k}G_a\le C\inf_{B_k}G_a.
    $$
    The constant $C$ depends only on the constant in the Poincar\'e inequality, and the number $k_0$ depends only on the points $x,q$ (and the geometry of $M$ away from $\mathcal{S}$). Thus, 
    $$
        G_a(x)\le C^{k_0}=:C(x)
    $$
    as required. 
\end{proof}

\subsection{Bounding the limit under the RCD condition}\label{sec:bound under RCD cond}

In this subsection, we will prove Proposition \ref{prop:bound for u_infty} under the assumption that $(M,d,\mu)$ is an almost-manifold as in the introduction and satisfies the $RCD(K, N)$ condition. 

\begin{proposition}\label{prop:upper bound RCD}
    Let $\Omega\supset\mathcal{S}$ be a bounded connected domain and $k=\dim_H(\mathcal{S})<\frac{n-2}{2}$. Then there exists a constant $K\in(0,\infty)$ depending only on $n$, such that for any solution $u$ to \eqref{eq:Yamabe} on $M\setminus\mathcal{S}$, we have
    $$
        u\le\max\{K,\sup_{\partial\Omega}u\} \text{ on }\Omega\setminus\mathcal{S}.
    $$
\end{proposition}
\begin{proof}
    First, note that if 
    $$
        u\ge K:=(2n(n-1))^{\frac{n-2}{4}},
    $$
    then we have 
    $$
        R_gu-fu^\alpha\ge -n(n-1)u+u^\alpha\ge \frac{1}{2}u^\alpha.
    $$
    From $L_gu=0$ we have $c_n\Delta_gu=R_gu-fu^\alpha$, hence $\Delta_gu\ge\frac{1}{2c_n}u^\alpha$.
    
    We prove this by contradiction. If the conclusion does not hold, then we can find $x_0\in\Omega\setminus\mathcal{S}$ such that $u(x_0)>\max\{K,\sup_{\partial\Omega}u\}$. By \eqref{item:Green ub} of Proposition \ref{prop:estimate Green's fun1}, we have 
    $$
        G_a(x_0)\le C(x_0)=C_1\delta^{2-n}\le C_1 R^{2-n}
    $$
    for all $a\in\mathcal{S}$ and $R<\delta:=(C(x_0)/C_1)^{1/(2-n)}$. 
    
    Fix any $\epsilon>0$. Since $\dim_H(\mathcal{S})=k<\frac{n-2}{2}$, and since $\mathcal{S}$ is compact, we can find $a_1,\ldots,a_N\in\mathcal{S}$ and $R_1,\ldots R_N\in (0,\delta/2)$ such that
    $$
        \mathcal{S}\subset \bigcup_{i=1}^N B_{R_i}(a_i), \quad \sum_{i=1}^N R_i^{\frac{n-2}{2}}<\epsilon.
    $$
    For each $i$, define 
    $$
        v_i(x)=\Lambda \left(\frac{R_i}{(G_{a_i}(x)/C_1)^{\frac{2}{2-n}}-R_i^2}\right)^{\frac{n-2}{2}},
    $$
    where $\Lambda$ is a constant to be determined. The domain of $v_i$ is 
    $$
        D_i:=\{x\in M: G_{a_i}(x)<C_1R_i^{2-n}\}.
    $$

    We claim that $\{B_i:i=1,\ldots,N\}$ covers $\mathcal{S}$, where 
    $$
        B_i:=\{x\in M: G_{a_i}(x)>C_1R_i^{2-n}\}.
    $$ 
    If not, then there exists $x\in\mathcal{S}$ such that $G_{a_i}(x)\le C_1R_i^{2-n}$ for all $i$. By \eqref{item:Green lb} of Proposition \ref{prop:estimate Green's fun1}, we have $C_1d(x,a_i)^{2-n}\le G_{a_i}(x)\le C_1R_i^{2-n}$, and thus $d(x,a_i)\ge R_i$ for all $i$. This contradicts the fact that $\{B_{R_i}(a_i):i=1,\ldots,N\}$ covers $\mathcal{S}$. It follows that $\bigcap_{i=1}^N D_i$ is disjoint from $\mathcal{S}$. 
    
    If we write $r=(C_1^{-1}G_{a_i}(x))^{\frac{1}{2-n}}$, $R=R_i$, and $v(r)=\Lambda(R/(r^2-R^2))^{\frac{n-2}{2}}$, then a direct calculation shows that
    $$
        v''+\frac{n-1}{r}v'=n(n-2)\Lambda \left(\frac{R}{r^2-R^2}\right)^{\frac{n+2}{2}}.
    $$
    It follows that 
    \begin{align*}
        \Delta_g v_i&=\frac{C_1^{2/(n-2)}}{(n-2)^2}G_{a_i}^{-\frac{2(n-1)}{n-2}}|\nabla G_{a_i}|^2(v''+\frac{n-1}{r}v')\\
        &=\frac{C_1^{2/(n-2)}\Lambda n}{n-2}G_{a_i}^{-\frac{2(n-1)}{n-2}}|\nabla G_{a_i}|^2(v_i/\Lambda)^\alpha\\
        &\le \frac{C_1^{2/(n-2)}C_2^2\Lambda^{-\frac{4}{n-2}}n}{n-2}v_i^\alpha,
    \end{align*}
    where the last inequality follows from \eqref{item:Green grad} of Proposition \ref{prop:estimate Green's fun1}. We choose 
    $$
        \Lambda=\left(\frac{2c_n n}{n-2}C_1^{\frac{2}{n-2}}C_2^2\right)^{\frac{n-2}{4}},
    $$
    and thus we have $\Delta_g v_i\le \frac{1}{2c_n}v_i^\alpha$. 

    Lastly, for $\zeta>0$, we set 
    $$
        V=\sum_{i=1}^N v_i, \quad u_\zeta=u-\max\{K-\zeta,\sup_{\partial\Omega}u\},
    $$
    where $V$ is defined on $\bigcap_{i=1}^N D_i$. We compare these two functions on the set 
    $$
        \Omega_{\zeta}=\left\{x\in\Omega\cap \bigcap_{i=1}^N D_i:u_\zeta(x)>V\right\}.
    $$

    By Sard's theorem, there is a sequence $\zeta_i\to 0$ such that $\Omega_{\zeta_i}$ has a smooth boundary. On the boundary $\partial\Omega_{\zeta}$, we have $V\ge u_\zeta$ since on $\partial\Omega$, we have $V\ge 0\ge u_\zeta$, and on $\partial D_i$, we have $V\to\infty$ while $u_\zeta$ is bounded. Moreover, we have 
    $$
        \Delta_g V\le \frac{1}{2c_n}V^\alpha,\quad \Delta_g u_\zeta\ge \frac{1}{2c_n}u_\zeta^\alpha,
    $$
    and thus $V\ge u_\zeta$ on $\Omega_\zeta$ by the maximum principle. However, this contradicts the definition of $\Omega_\zeta$, which forces $\Omega_\zeta=\varnothing$. In particular, we have 
    $$
        u(x_0)\le V(x_0)+\max\{K-\zeta,\sup_{\partial\Omega}u\}.
    $$
    Let $\zeta=\zeta_i\to 0$, and we conclude that 
    $$
        u(x_0)\le V(x_0)+\max\{K,\sup_{\partial\Omega}u\}.
    $$
    Note that 
    $$
        V(x_0)=\Lambda \sum_{i=1}^N \frac{R_i^{(n-2)/2}}{((G_{a_i}(x_0)/C_1)^{\frac{2}{2-n}}-R_i^2)^{\frac{n-2}{2}}}\le \frac{\Lambda}{(3\delta^2/4)^{\frac{n-2}{2}}} \sum_{i=1}^N R_i^{(n-2)/2}<\frac{\epsilon\Lambda}{(3\delta^2/4)^{\frac{n-2}{2}}}.
    $$
    Let $\epsilon\to 0$. Then we obtain that $u(x_0)\le \max\{K,\sup_{\partial\Omega}u\}$, which is a contradiction.
\end{proof}

Using this, Proposition \ref{prop:bound for u_infty} can be proved in the same way as before, and we will not repeat it here. 

\subsection{Proof of the positive mass theorem}\label{sec:PMT proof}

We prove the positive mass theorem in two cases. 

\textbf{Case 1: $R_g>-n(n-1)$ at some $p\in\mathcal{R}$.} In this case, we will take $K$ to be a compact set containing $p$. Then we see that $u_\infty\equiv 1$ cannot happen. By the strong maximum principle, $u_\infty<1$. In particular, on $\partial K$, $u_\infty<1-\delta<1$ for some constant $\delta>0$, and thus for large enough $j$, we have $u_j<1-\frac{\delta}{2}$ on $\partial K$. 

\begin{lemma}\label{lem:normal derivative}
    For sufficiently large $j$, $u_j\le 1$ on $E_i\cap K'$ for each $i\in I$. In particular, $\nabla^{g}_{\nu}u_j\ge 0$ for the normal vector field $\nu$ on $\Sigma_i$ pointing away from $K'$. 
\end{lemma}
\begin{proof}
    On the boundary of $E_i\cap K'$, we know that $u_j\le 1$. Suppose the contrary that $u_j(x_0)=\max_{E_i\cap K'}u_j>1$. Then at $x_0$, 
    $$
        0\ge c_n\Delta_gu_j=R_gu_j-fu_j^\alpha\ge (u_j-u_j^\alpha)f>0.
    $$
    This contradiction shows that $u_j\le 1$ on $E_i\cap K'$. Since $u_j=1$ on $\Sigma_i$, the normal derivative must be non-negative with respect to $\nu$. 
\end{proof}

\begin{lemma}\label{lem:mean curvature}
    Let $\nu$ be the normal vector field as in Lemma \ref{lem:normal derivative}. With respect to the metric $\tilde{g}_j=u_j^{\frac{4}{n-2}}g$, the mean curvature $\tilde{H}_j(\Sigma_i)$ is no smaller than $H(\Sigma_i)$ with respect to the original metric $g$. 
\end{lemma}
Here, we use the convention that the mean curvature of the standard ball is positive with respect to the outward normal vector field.
\begin{proof}
    Since $u_j=1$ on $\Sigma_i$, we know that 
    $$
        \tilde{H}_j(\Sigma_i)=H(\Sigma_i)+(n-1)\nabla^{g}_{\nu}\left(\frac{2}{n-2}\ln u_j\right)=H(\Sigma_i)+\frac{2(n-1)}{n-2}\frac{\nabla^g_\nu u_j}{u_j}\ge H(\Sigma_i),
    $$
    by Lemma \ref{lem:normal derivative}. 
\end{proof}

Now we can prove the positive mass theorem. 

\begin{proposition}\label{prop:PMT strict}
    Assume that $R_g>-n(n-1)$ at some $p\in\mathcal{R}$. Then 
    $$
        \EADM(g)>|\PADM(g)|.
    $$
\end{proposition}
\begin{proof}
    Since $u_j<1-\frac{\delta}{2}$ on $\partial K$ by our choice of $j$, we claim that $u_j<1$ on the AH end $E_0$. Indeed, if $u_j>1$ somewhere on $E_0$, then there exists $x_0\in E_0$ such that $u_j(x_0)=\max_{E_0}u_j$. Thus, at $x_0$, 
    $$
        0\ge c_n\Delta_gu_j=R_gu_j-fu_j^\alpha\ge (u_j-u_j^\alpha)f>0.
    $$
    The contradiction shows that $u_j\le 1$ on $E_0$ and the strong maximum principle shows that $u_j<1$ on $E_0$. Proposition \ref{prop:gap} then implies that there exists $\delta_1>0$ such that 
    $$
        u_j<1-\delta_1\rho^n
    $$
    near $\partial X$. 

    Consider the $C^0$-metric on $M\setminus\bar U_j$ defined by $\tilde{g}_j$ on $M'\setminus\bar{U}_j$ and $g$ on $E_i\setminus K'$ for each $i\in I$. For simplicity, we also denote this metric by $\tilde{g}_j$. By Lemma \ref{lem:mean curvature}, we can apply \cite{Miao02} to smooth out the metric near $\Sigma_i$: There exists $\epsilon_0>0$ such that for any $\epsilon\in(0,\epsilon_0)$, there is a $C^2$-metric $g_\epsilon$ on $M\setminus\bar U_j$ such that 
    \begin{enumerate}
        \item $\|g_\epsilon-\tilde{g}_j\|_{C^0}=O(1)\epsilon^2$,
        \item $g_\epsilon=\tilde{g}_j$ when $|d|>\frac{\epsilon}{2}$, 
        \item $R_{g_\epsilon}=O(1)$ when $|d|\in (\frac{\epsilon^2}{100},\frac{\epsilon}{2})$, 
        \item $R_{g_\epsilon}=O(1)+2(H_--H_+)(\frac{100}{\epsilon^2}\phi(\frac{100d}{\epsilon^2}))$ when $|d|<\frac{\epsilon^2}{100}$,
    \end{enumerate}
    where $O(1)$ is a uniformly bounded term, $d$ is the signed $\tilde{g}_j$-distance from $\Sigma_i$, $\phi\in C_c^\infty(-1,1)$, $0\le \phi \le 1$, and $H_--H_+=\tilde{H}_j-H\ge 0$ in our case. We may assume $g_\epsilon$ is smooth. Define 
    $$
        f_\epsilon=c_n^{-1}[R_{g_\epsilon}+n(n-1)]^-.
    $$
Since $R_{\tilde g_j}\ge-n(n-1)$ away from $\Sigma_i$, $f_\epsilon=0$ when $|d|>\epsilon/2$. The inequality $H_--H_+\ge0$ gives $f_\epsilon=O(1)$, and its support has volume $O(\epsilon)$. Therefore
    $$
\|f_\epsilon\|_{L^p}^p\le C_p\epsilon,\qquad \|f_\epsilon\|_{L^p}=O(\epsilon^{1/p}),\quad 1\le p<\infty.
    $$
    By the argument in \cite{BQ08}, we obtain a smooth function $v_\epsilon$, cf.\ Lemma \ref{lem:BQ conformal change}. Then the new metric $\tilde{g}_\epsilon=(1+v_\epsilon)^{\frac{4}{n-2}}g_\epsilon$ on $M\setminus\bar{U}_j$ is complete with AH end $E_0$. We choose $\epsilon$ small enough so that in the expansion 
    $$v_\epsilon=A\rho^n+O(\rho^{n+1}),$$
    there holds $|A|<\frac{\delta_1}{2}$. Since 
    $$
        \tilde{g}_\epsilon=((1+v_\epsilon)u_j)^{\frac{4}{n-2}}g
    $$
    near $\partial X$, and 
    $$
        (1+v_\epsilon)u_j<1-\frac{\delta_1}{2}\rho^n
    $$
    near $\partial X$, we know that 
    $$
        \EADM(\tilde{g}_\epsilon)-|\PADM(\tilde{g}_\epsilon)|<\EADM(g)-|\PADM(g)|
    $$
    by Lemma \ref{cor:mass change}. Since the initial data set $(M\setminus\bar U_j, \tilde g_\epsilon,\tilde g_\epsilon)$ satisfies the dominant energy condition, we obtain 
    $$
        \EADM(\tilde{g}_\epsilon)-|\PADM(\tilde{g}_\epsilon)|\ge 0,
    $$
    via recent works of PMT on AH manifolds; see \cite{hirsch2026,BW2026} for AH manifolds and \cite{tsang2026} for AH manifolds with arbitrary ends.
\end{proof}

\begin{lemma}\label{lem:BQ conformal change}
    For $f_\epsilon$ as in the proof of Proposition \ref{prop:PMT strict}, there exists a smooth function $v=v_\epsilon$ on $M\setminus\bar U_j$ satisfying 
    \begin{enumerate}
        \item $-\Delta_{g_\epsilon}v+nv-f_\epsilon v=f_\epsilon$ on $M\setminus\bar{U}_j$, 
        \item $v(x,\rho)=A(x)\rho^n+B(x,\rho)$ near $\partial X$ for some $A\in C^\infty(\partial X)$ and $B\in C^\infty_{n+1}(E_0)$, and 
        \item $v\ge 0$ on $M\setminus\bar{U}_j$. 
    \end{enumerate}
    Moreover, $|A|\le C\epsilon^{\frac{1}{n+1}}$ for some constant $C$ independent of $\epsilon$. 
\end{lemma}
\begin{proof}
    Let $\{\Omega_l\}$ be a compact exhaustion of $M\setminus\bar U_j$, where each $\Omega_l$ is connected. We first solve the Dirichlet problem 
    \begin{equation}\label{eq:BQ}
        \begin{cases}
            -\Delta_{g_\epsilon}v+nv-f_\epsilon v=f_\epsilon &\text{on }\Omega_l, \\
            v=0 &\text{on }\partial\Omega_l.
        \end{cases}
    \end{equation}
    Without loss of generality, we assume that $D:=\supp f_\epsilon$ is a subset of $\Omega_l$. 

    Let $U$ be an open neighborhood of $D$ in $\Omega_l$, with smooth boundary $\partial U$. We claim that 
    $$
        \|v\|_{H^1(U)}<C(\epsilon,U)\text{ on }U
    $$
    where the constant $C(\epsilon,U)$ is independent of $\Omega_l$. To see this, multiply the equation \eqref{eq:BQ} by $v$ and integrate over $\Omega_l$; we obtain 
    \begin{equation}\label{eq:BQ-IBP}
        \int_{\Omega_l} |\nabla v|^2+nv^2=\int_{\Omega_l} v(1+v)f_\epsilon=\int_U v(1+v)f_\epsilon.
    \end{equation}
    Hence, if we let $C(U)$ be the Sobolev constant of $H^1(U)\to L^{\frac{2n}{n-2}}(U)$, then 
    \begin{align*}
        \|v\|_{H^1(U)}^2&\le\int_U vf_\epsilon+v^2f_\epsilon\\
        &\le\left(\int_U |v|^{\frac{2n}{n-2}}\right)^{\frac{n-2}{2n}}\left(\int_U f_\epsilon^{\frac{2n}{n+2}}\right)^{\frac{n+2}{2n}}+\left(\int_U |v|^{\frac{2n}{n-2}}\right)^{\frac{n-2}{n}}\left(\int_U f_\epsilon^{\frac{n}{2}}\right)^{\frac{2}{n}}\\
        &\le C(U)\|f_\epsilon\|_{L^{\frac{2n}{n+2}}(U)}\|v\|_{H^1(U)}+C(U)^2\|f_\epsilon\|_{L^{\frac{n}{2}}(U)}\|v\|_{H^1(U)}^2
    \end{align*}
    We take $\epsilon$ to be so small that 
    $$
        C(U)^2\|f_\epsilon\|_{L^{\frac{n}{2}}(U)}<\frac{1}{2},
    $$
    and then 
    $$
        \|v\|_{H^1(U)}<2C(U)\|f_\epsilon\|_{L^{\frac{2n}{n+2}}(U)}=:C(\epsilon,U),
    $$
    proving the claim. Furthermore, by \eqref{eq:BQ-IBP}, we obtain 
    $$
        \|v\|_{H^1(\Omega_l)}^2=\int_U v(1+v)f_\epsilon\le 4C(U)^2\|f_\epsilon\|_{L^{\frac{2n}{n+2}}(U)}^2=:C_2(\epsilon,U)^2.
    $$
    Hence, we obtain a universal bound $C_2(\epsilon,U)$ of $\|v\|_{H^1(\Omega_l)}$ independent of $l$. 

    Let $v_l$ be a solution to \eqref{eq:BQ}. From the above estimate, we conclude that there is a subsequence, still denoted as $\{v_l\}$, such that $v_l\to v_\infty$ smoothly on compact subsets of $M\setminus\bar U_j$. This limit, which will also be referred to as $v_\epsilon$, satisfies 
    $$
        -\Delta_{g_\epsilon}v+nv-f_\epsilon v=f_\epsilon\text{ on }M\setminus\bar{U}_j.
    $$
    On the AH end $E_0$, we have the barrier functions as in the proof of Proposition \ref{prop:gap}, so 
    $$
        -a\rho^n+b\rho^{n+1}<v_\epsilon<a\rho^n-b\rho^{n+1}
    $$
    for some constants $a,b>0$. The same argument as in \S\ref{sec:smoothness} implies that 
    $$
        v_\epsilon=A(x)\rho^n+B(x,\rho)
    $$
    for $A\in C^\infty(\partial X)$ and $B\in C^\infty_{n+1}(E_0)$; see \cite[Section 5]{BQ08} for another proof, where the estimate $|A|<C\epsilon^{\frac{1}{n+1}}$ is proved. 

    Lastly, we show that $v_\epsilon\ge 0$. Fix any $\delta\in(0,1)$. Since $v(x)\to 0$ as $x\to\partial X$, we know that for a small $\rho_1>0$, 
    $$
        v\ge -\delta\text{ on }E_0\cap\{\rho<\rho_1\}.
    $$
    Then there exists $l_0$ such that for all $l\ge l_0$, 
    $$
        v_l>-\frac{\delta}{2}\text{ on }E_0\cap\partial\Omega_l.
    $$
    Let $w=v_l+\frac{\delta}{2}$. Then $w>0$ on $\partial\Omega_l$ and $w$ satisfies 
    $$
        -\Delta_{g_\epsilon} w+nw-fw>0
    $$
    on $\Omega_l$. Note that the first eigenvalue of 
    $$
        \begin{cases}
            (-\Delta_{g_\epsilon}+n-f)\phi=0&\text{on }\Omega_l\\ 
            \phi=0&\text{on }\partial\Omega_l
        \end{cases}
    $$
    is positive. Indeed, similar to before, we have 
    \begin{align*}
        \int_{\Omega_l}|\nabla \phi|^2+n\phi^2-f\phi^2&\ge\left(\int_{\Omega_l}|\nabla \phi|^2+n\phi^2\right)-C(U)^2\|f_\epsilon\|_{L^{\frac{n}{2}}(U)}\|\phi\|_{H^1(U)}^2\\
        &\ge \left(\int_{\Omega_l}|\nabla \phi|^2+\phi^2\right)-C(U)^2\|f_\epsilon\|_{L^{\frac{n}{2}}(U)}\|\phi\|_{H^1(\Omega_l)}^2\\
        &\ge (1-C(U)^2\|f_\epsilon\|_{L^{\frac{n}{2}}(U)})\|\phi\|_{H^1(\Omega_l)}^2\ge\frac{1}{2}\|\phi\|_{H^1(\Omega_l)}^2\ge\frac{1}{2}\int_{\Omega_l}\phi^2.
    \end{align*}
    Hence, the first eigenvalue is at least $\frac{1}{2}$. Let $\phi>0$ be an eigenfunction corresponding to the first eigenvalue. Then $\frac{w}{\phi}>0$ on $\Omega_l$ by the generalized maximum principle. Hence, $v_l>-\frac{\delta}{2}$ and taking the limit, we have $v_\infty\ge-\frac{\delta}{2}$. Since this is true for any $\delta\in(0,1)$, we have $v_\epsilon=v_\infty\ge 0$. 
\end{proof}

\textbf{Case 2: $R_g\equiv -n(n-1)$.} In this case, we modify the construction of $u_j$ as follows. Instead of cutting off the arbitrary ends, we solve the Yamabe equation \eqref{eq:Yamabe} on $M\setminus\bar U_j$. 

\begin{lemma}\label{lem:u_j arbitrary end}
    If $R_g\le R_i$ for some constant $R_i<0$ on $E_i$, then there exists a positive function $u_j\in C^\infty(M\setminus\bar U_j)$ such that 
    \begin{enumerate}
        \item $L_gu_j=0$ on $M\setminus\bar U_j$, 
        \item $u_j(x)r_j(x)^{\frac{n-2}{2}}\to 1$, as $x\to\partial U_j$,
        \item $u_j\ge a_i$ for some constant $a_i>0$ on $E_i$, and 
        \item $u_j=1+A_j\rho^n+O(\rho^{n+1})$ near $\partial X$. 
    \end{enumerate}
    Here, $a_i$ is independent of $j$, and $|A_j|\le A$ for all $j$.
\end{lemma}
\begin{proof}
    Let $\{K'_l:l\in\mathbb{Z}_{>0}\}$ be a compact exhaustion of $M$ with $K'_1=K'$ and $K'\cap E_i$  connected for all $i\in I$. As before, we set 
    $$
        M'_k=K'_k\cup E_0.
    $$
    Let $u_{j,l}$ be the unique positive solution to the boundary value problem 
    $$
        \begin{cases}
            L_gu=-c_n\Delta_gu+R_gu-fu^\alpha=0&\text{on }M'_l\setminus\bar U_j, \\
            u(x)r_j(x)^{\frac{n-2}{2}}\to 1&\text{as }x\to\partial U_j, \\
            u(x)\to 1&\text{as }x\to\partial X, \\ 
           u(x)=1 &\text{on }\partial K'\cap E_i\text{ for each }i\in I.
        \end{cases}
    $$
    Note that on $\partial E_i$, there exists a constant $b_i>0$ such that for all $j,l$, 
    $$
        \inf_{\partial E_i}u_{j,l}\ge b_i.
    $$
    For if not, then there is a convergent subsequence of $u_{j,l}$ such that the limit vanishes at some point on $\partial E_i$, which contradicts the  strong  maximum principle. 

    We claim that 
    $$
        u_{j,l}\ge\min\left\{\frac{b_i}{2},\left(\frac{-R_i}{n(n-1)}\right)^{\frac{1}{\alpha-1}}\right\}=:a_i \text{ on } E_i\setminus K_l.
    $$ 
    If not, then we can find an interior point $x_0$ of $E_i\setminus K_l$ such that 
    $$
        u_{j,l}(x_0)=\min_{E_i\setminus K_l}u_{j,l}<a_i.
    $$
    Then at $x_0$, we have 
    $$
        0\le c_n\Delta_gu_{j,l}=R_gu_{j,l}-fu_{j,l}^{\alpha}\le R_iu_{j,l}+n(n-1)u_{j,l}^{\alpha}<0.
    $$
    Hence, the claim holds, and we can let $l\to\infty$ to see that a subsequence converges to $u_j$ with the desired properties. (Cf.\ Corollary \ref{cor:convergence} and Proposition \ref{prop:tangential smoothness}.)
\end{proof}

We complete the proof of the positive mass theorem for Case 2 in the following proposition: 

\begin{proposition}\label{prop:PMT R<0}
    If $R_g\le R_i$ for some constant $R_i<0$ on $E_i$ for each $i\in I$, then 
    $$\EADM(g)\ge |\PADM(g)|.$$
\end{proposition}
\begin{proof}
    By Lemma \ref{lem:u_j arbitrary end}, let $\tilde{g}_j=u_j^{\frac{4}{n-2}}g$, and then $(M\setminus\bar U_j,\tilde{g}_j)$ is a complete AH manifold. Let $\tilde{h}_j$ be the mass aspect of $\tilde{g}_j$. Then by the recent works \cite{BW2026, hirsch2026,tsang2026}, the positive mass theorem holds, i.e., 
    \begin{equation}\label{eq:PMT for M_j}
        \int_{S^{n-1}} \operatorname{tr}_{\gamma_{\mathrm{std}}}(\tilde{h}_j)\, d \mu_{\gamma_{\mathrm{std}}} \ge \left|\int_{S^{n-1}} \operatorname{tr}_{\gamma_{\mathrm{std}}}(\tilde{h}_j) x \, d \mu_{\gamma_{\mathrm{std}}}\right|.
    \end{equation}
    Let $j\to\infty$; by the dominated convergence theorem, we have 
    $$
        \int_{S^{n-1}} \operatorname{tr}_{\gamma_{\mathrm{std}}}(\tilde{h}_\infty)\, d \mu_{\gamma_{\mathrm{std}}} \ge \left|\int_{S^{n-1}} \operatorname{tr}_{\gamma_{\mathrm{std}}}(\tilde{h}_\infty) x\, d \mu_{\gamma_{\mathrm{std}}}\right|,
    $$
    where $\tilde{h}_\infty$ is the mass aspect of the metric $\tilde{g}_\infty=u_\infty^{\frac{4}{n-2}}g$. On the other hand, by Proposition \ref{prop:bound for u_infty}, we have 
    $$
        u_\infty=1+A_\infty\rho^n+O(\rho^{n+1}) \text{ as }\rho\to 0,
    $$
    where $A_\infty\le 0$. Hence, Corollary \ref{cor:mass change} implies that 
    $$
        \EADM(g)-|\PADM(g)|\ge \EADM(\tilde{g}_\infty)-|\PADM(\tilde{g}_\infty)|\ge 0,
    $$
    as desired. 
\end{proof}

Combining Propositions \ref{prop:PMT strict} and \ref{prop:PMT R<0}, we obtain: 

\begin{theorem}\label{thm:PMT scalar rigidity}
    Let $(M^n, d, \mu)$ be an almost manifold with its regular part $\mathcal R$ containing an AH end $E_0$ and possibly some arbitrary ends. Suppose its scalar curvature $R_g\geq -n(n-1)$ on $\mathcal R$ and $\dim_{H}(\mathcal{S})< \frac{n-2}{2}$. Moreover, we assume that either 
    \begin{itemize}
        \item there exists a neighborhood $U$ of $\mathcal S$ such that $\Ric_g(x)\ge-Cd(x,\mathcal S)^{-2}$ for all $x\in U\cap\mathcal R$, or
        \item $(M^n, d, \mu)$ is an $RCD(K,N)$ space.
    \end{itemize}
    Here, $C, K, N$ are universal constants that depend only on $(M^n,d,\mu)$. Then we have 
    
    \begin{equation}
        \int_{S^{n-1}} \operatorname{tr}_{\gamma_{\mathrm{std}}}(h) \,d \mu_{\gamma_{\mathrm{std}}} \geq\left|\int_{S^{n-1}} \operatorname{tr}_{\gamma_{\mathrm{std}}}(h) x \,d \mu_{\gamma_{\mathrm{std}}}\right|.
    \end{equation}
    
    The equality holds only if $R_g=-n(n-1)$ on $\mathcal{R}$.
\end{theorem}

\subsection{Proof of Ricci rigidity in the smooth case}\label{sec:Ricci rigid smooth}

We now turn to the Ricci rigidity part of the proof. In the remainder of
this subsection, we restrict to the smooth setting ($\mathcal{S}=\varnothing$)
and assume additionally that the traceless Ricci tensor
$\kappa:=\Ric_g+(n-1)g$ satisfies
\[
\int_M \|\kappa\|_g^2\,d\mu_g<\infty.
\]
Under these hypotheses, we will show that the equality case of the positive mass theorem forces the metric to be Einstein; the argument in Section \ref{sec:hyperbolic} then shows that any smooth Einstein AH manifold with arbitrary ends is necessarily hyperbolic.

\begin{proposition}\label{prop:ricci-rigidity}
Let $(M^n,g)$ be a smooth complete manifold with an AH end $E_0$ and possibly some arbitrary ends. Assume $R_g\ge-n(n-1)$. Set $\kappa=\Ric_g+(n-1)g$ and assume that $\int_M\|\kappa\|_g^2\,d\mu_g<\infty$.
If equality holds in the positive mass theorem, i.e., $\EADM(g)=|\PADM(g)|$, then
\[
\Ric_g\equiv-(n-1)g\quad\text{on }M.
\]
\end{proposition}
\begin{proof}
    Throughout the proof, we assume that $M$ possesses at least one arbitrary end.  This is no loss of generality: if $M$ had no arbitrary end, then $M$ would be a (smooth) complete asymptotically hyperbolic manifold with a single AH end, for which the rigidity in the equality case of the positive mass theorem is already known (cf.\ \cite{Huang2020,Hir2025}).

    By Theorem~\ref{thm:PMT scalar rigidity}, equality in the positive mass theorem forces $R_g\equiv -n(n-1)$ on $M$. If $\kappa\equiv0$ then $\Ric_g=-(n-1)g$ and we are done. Otherwise, pick a point $p\in M$ with $\kappa(p)\neq0$.

    We list some notations that will be defined and used repeatedly in this proof. 

    \begin{center}
        \begin{tabular}{c|c}
            $g$ & the metric in question \\ 
            \hline
            $g_t=g+t\eta\kappa$ & the deformed metric \\ 
            \hline
            $\tilde{g}_t=u_t^{4/(n-2)}g_t$ & the conformally deformed metric \\ 
            \hline
            $u_t^{(k)}$ & the solution to the Yamabe equation on $\Omega_k$ \\
            \hline
            $v_t^{(k)}=t^{-1}(u_t^{(k)}-1)$ & the linear part of $u_t^{(k)}$ \\ 
            \hline
            $v^{(k)}=\lim_{t\to 0^-}v_t^{(k)}$ & the solution to the linearized equation on $\Omega_k$ \\ 
            \hline
            $v^{(k)}_1, v^{(k)}_2$ & the positive part and the small part of $v^{(k)}$ \\ 
            \hline
            $v_1=\lim_{k\to\infty}v^{(k)}_1$, $v_2=\lim_{k\to\infty}v^{(k)}_2$ & the global limits \\ 
            \hline 
            $v=v_1+v_2$ & the solution to the linearized equation on $M$ \\ 
            \hline
            $u_t=\lim_{k\to\infty} u_t^{(k)}$ & the solution to the Yamabe equation on $M$ \\
            \hline
            $v_t=\lim_{k\to\infty} v_t^{(k)}$ & the linear part of $u_t$
        \end{tabular}
    \end{center}
    For simplicity, we write $\hat{R}=R+n(n-1)$, and in particular, $\hat{R}(g)=0$.

    \textbf{Step 1: Construction of a cutoff function.} Let $U_*$ be a neighborhood of the AH end $E_0$ and the point $p$ such that $U_*\setminus E_0$ is relatively compact, and that $\partial U_*$ is smooth. The distance function to $U_*$ is $1$-Lipschitz. By the smooth approximation theorem~\cite[Theorem~1]{AFLMR07}, choose a smooth approximation $d$ such that
    \begin{itemize}
        \item $d=0$ on $U_*$, 
        \item $|d(x)-\dist_g(x,U_*)|\le 1$ for all $x\in M$, and
        \item $|\nabla d|_g\le 2$ on $M$.
    \end{itemize}
    Choose a smooth nonincreasing $\psi:\mathbb R\to[0,1]$ equal to $1$
    on $(-\infty,1]$ and to $0$ on $[2,\infty)$, and, for $\ell\ge2$, set
    \[
        \eta=\eta_\ell:=\psi\left(\frac{d+1}{\ell}\right).
    \]
    Then $\eta_\ell=1$ on $U_*$, and $\eta=0$ far away from $U_*$, that is, $\eta=0$ eventually in arbitrary ends $E_i$. We have the gradient estimate 
    \begin{equation}\label{eq:eta-estimates}
        |\nabla\eta_\ell|_g\le \frac{2\|\psi'\|_\infty}{\ell}=:Cl^{-1},
    \end{equation}
    where $C$ is a universal constant. The point $p$ and the functions $d,\psi$ will be fixed throughout the proof, while $\ell$ is the only parameter that will be decided later, cf.\ \eqref{eq:ell condition}. Set 
    \[
        g_t=g_{t,l}:=g+t\kappa\eta_\ell. 
    \]
    We collect some basic properties of $g_t$ in the following lemma.
    \begin{lemma}\label{lem:g_t curvature}
        There exists $t_\ell>0$ such that for any $|t|<t_\ell$, $p\in [1,\infty]$, there exists a constant $C_{\ell,p}>0$
        \[
            \|\hat{R}(g_t)\|_{L^p(M)}\le C_{\ell,p}|t|,
        \]
        \[
            \|\hat{R}(g_t)-tDR(g)(\eta\kappa)\|_p\le C_{\ell,p} t^2,
        \]
        where 
        \[
            DR(g)(\eta\kappa)=-\eta\|\kappa\|_g^2+\divv_g\divv_g(\eta\kappa).
        \]
        As $\rho\to 0$, there holds 
        \[
            |\hat{R}(g_t)|\le C_\ell|t|\rho^{2n}.
        \]
    \end{lemma}
    \begin{proof}
        By fixing a small $t_\ell>0$, we can assume that for any $|t|\le t_\ell$,
        \[
            \frac{1}{2}g\le g_t\le 2g.
        \]
        Then the first variation for the scalar curvature (cf.\ Lemma \ref{lem:1st variation of Sc}) gives
        \[
            |\hat{R}(g_t)|\le C_n|t|(K|\eta\kappa|_g+|\nabla^2(\eta\kappa)|_g+|\nabla(\eta\kappa)|_g^2),
        \]
        for some dimensional constant $C_n$, where 
        \[K=\sup_{\supp\eta}|\Ric(g)|\] 
        is a constant depending on $\ell$. The $L^\infty$-bound follows immediately. 
        
        For the $L^p$-bound, it suffices to analyze the AH end $E_0$ where $\eta\equiv 1$, since $(\supp\eta)\setminus E_0$ is compact. In $E_0$, we have
        \[
            |\hat R(g_t)|\le C_{\ell}|t|(|\kappa|_g+|\nabla^2\kappa|_g+|\nabla\kappa|_g^2),
        \]
        where $C_\ell$ is a constant depending on $(M^n,g)$ and $\ell$. By Lemma \ref{lem:kappa expansion}, 
        \[
            |\nabla^k\kappa|_g=O(\rho^n), \qquad k=0,1,2,
        \]
        we have
        \[
            \|\hat R(g_t)\|_{L^p(E_0)}\le C_{\ell,p}|t|,
        \]
        for some constant $C_{\ell,p}$ depending on $(M^n,g)$, $\ell$, and $p$. This proves the lemma.

        The estimate for the second variation follows in the same way using Lemma \ref{lem:2nd variation of Sc}. This also gives a better decay rate for $\hat{R}(g_t)$, since at the AH end $E_0$, 
        \[
            DR(g)(\eta\kappa)=DR(g)(\kappa)=-\|\kappa\|_g^2=O(\rho^{2n}),
        \]
        and
        \[
            |\hat{R}(g_t)-tDR(g)(\eta\kappa)|\le C_nt^2(|\kappa|^2_g+|\nabla\kappa|^2_g+|\kappa|_g|\nabla^2\kappa|_g^2)=O(\rho^{2n}).
        \]
        Hence, $|\hat{R}(g_t)|\le C_\ell |t| \rho^{2n}$. 
    \end{proof}

    \textbf{Step 2: The Yamabe equation on $\Omega_k$.} Let $\{\Omega_k\}_{k=1}^\infty$ be a monotonically increasing exhaustion of $M$ by relatively compact connected domains with smooth boundary. Consider the Dirichlet problem
    \begin{equation}\label{eq:Yamabe-Omega_k}
        \begin{cases}
            -c_n\Delta_{g_t}u+R(g_t)u+n(n-1)u^\alpha=0 & \text{on }\Omega_k, \\
            u=1 &\text{on }\partial\Omega_k.
        \end{cases}
    \end{equation}
    Then Lemma \ref{lem:variation} shows that, for $|t|<t_\ell$, the Dirichlet problem has a unique solution 
    \[
        u=u_t^{(k)}\in C^\infty(\bar{\Omega}_k).
    \]
    Here, $t_\ell$ is a small constant depending on $\ell$ so that Lemma \ref{lem:g_t curvature} holds. 

    To extract a limit, we introduce 
    \[
        v_t^{(k)}:=t^{-1}(u_t^{(k)}-1), \quad t\ne 0,
    \]
    and establish their $W^{1,2}$-bound. 

    \begin{lemma}\label{lem:W12-vt}
        There exists $C_\ell>0$ such that $v_t^{(k)}$ satisfies
        \[\|v_t^{(k)}\|_{L^\infty(\Omega_k)}\le C_\ell, \quad \|v_t^{(k)}\|_{W^{1,2}(\Omega_k)}\le C_{\ell},\] 
        for any $0<|t|\le t_\ell$ and any $k\ge 1$.
    \end{lemma}
    \begin{proof}
        For simplicity, write $u=u_t^{(k)}$ and $v=v_t^{(k)}$. Since $u$ solves \eqref{eq:Yamabe-Omega_k}, we have
        \[
            -c_n\Delta_{g_t}u+\hat R(g_t) u+n(n-1)(u^\alpha-u)=0,
        \]
        and thus $v$ satisfies
        \begin{equation}\label{eq:v-Omega_k}
            -c_n\Delta_{g_t}v+\hat R(g_t) v+n(n-1)t^{-1}(u^\alpha-u)
            =-t^{-1}\hat R(g_t).
        \end{equation}

        We first note that there exists a constant $C_\ell$ such that
        \[
            |u_t^{(k)}-1|\le C_\ell |t| \quad\text{on }\Omega_k,
        \]
        for any $0<|t|\le t_\ell$ and any $k\ge 1$. Indeed, it is easy to check that the functions
        \[
            u_\pm=\left(1\pm\frac{\|\hat R(g_t)\|_{L^\infty(\Omega_k)}}{n(n-1)}\right)^{1/(\alpha-1)},
        \]
        are sub- and super-solutions of \eqref{eq:Yamabe-Omega_k}. Lemma \ref{lem:g_t curvature} says that 
        \[
            |\hat R(g_t)|\le C_\ell |t|,
        \]
        for some constant $C_\ell$ independent of $t$ and $k$. Hence, we can choose $t_\ell$ small enough so that $u_\pm>0$ and the bound $|u_t^{(k)}-1|\le C_\ell |t|$ holds.

        Now we multiply \eqref{eq:v-Omega_k} by $v$ and integrate over $\Omega_k$. We have 
        \begin{align*}
            c_n\|\nabla v\|_{L^2(\Omega_k)}^2+\int_{\Omega_k}\hat R(g_t) v^2+n(n-1)\int_{\Omega_k}t^{-1}(u^\alpha-u)v
            =-\int_{\Omega_k}t^{-1}\hat R(g_t) v.
        \end{align*}
        The second term is bounded by $C_\ell |t|\|v\|_{L^2(\Omega_k)}^2$, and we show that the third term has a positive lower bound: Using 
        \[
            (1+tv)^\alpha-(1+tv)=(\alpha-1)tv+O(t^2v^2),
        \]
        we have 
        \[
            t^{-1}(u^\alpha-u)v\ge \frac{\alpha-1}{2}v^2,
        \]
        if $|tv|=|u-1|\le C_\ell |t|$ is sufficiently small. Hence, for sufficiently small $t_l$, we have 
        \[
            c_n\|\nabla v\|_{L^2(\Omega_k)}^2+d_n\|v\|_{L^2(\Omega_k)}^2\le \int_{\Omega_k}|t^{-1}\hat R(g_t)| |v|,
        \]
        where $d_n>0$ is a dimensional constant. The right-hand side can be bounded as 
        \begin{align*}
            \int_{\Omega_k}|t^{-1}\hat R(g_t) | |v|&\le \|t^{-1}\hat R(g_t) \|_{L^2(\Omega_k)}\|v\|_{L^2(\Omega_k)}\\
            &\le \frac{d_n}{2}\|v\|_{L^2(\Omega_k)}^2+\frac{1}{2d_n}\|t^{-1}\hat R(g_t) \|_{L^2(\Omega_k)}^2\\
            &\le \frac{d_n}{2}\|v\|_{L^2(\Omega_k)}^2+\frac{1}{2d_n}C_{\ell,2}^2,
        \end{align*}
        where we used Lemma \ref{lem:g_t curvature} in the last step. Combining the above estimates, we have
        \[
            \|v\|_{W^{1,2}(\Omega_k)}^2\le C_\ell,
        \]
        for some constant $C_\ell>0$ depending on $(M^n,g)$ and $l$. 
    \end{proof}

    \textbf{Step 3: The linearized equation on $\Omega_k$.} The first consequence of Lemma \ref{lem:W12-vt} is that we can take a convergent subsequence
    \[
        v_{t_{i,k}}^{(k)}\to v^{(k)} \quad \text{in } C^\infty(\bar\Omega_k),
    \]
    where $t_{i,k}\nearrow 0$, and $v^{(k)}$ solves the \emph{linearized equation}
    \begin{equation}\label{eq:linearized-v-k}
        \begin{cases}
            -\Delta_g v^{(k)}+n v^{(k)}
            =c_n^{-1}\Bigl(\eta\|\kappa\|_g^2
            -\divv_g\divv_g(\eta\kappa)\Bigr)
            &\text{on }\Omega_k,\\
            v^{(k)}=0 & \text{on }\partial\Omega_k.
        \end{cases}
    \end{equation}
    We write
    \[
        v^{(k)}=v^{(k)}_1+v^{(k)}_2, 
    \]
    where 
    \[\begin{cases}
        -\Delta_g v_1^{(k)}+n v_1^{(k)}=
        c_n^{-1}\eta\|\kappa\|_g^2&\text{on }\Omega_k,\\
        v_1^{(k)}=0&\text{on }\partial\Omega_k,
    \end{cases}\]
    \[\begin{cases}
        -\Delta_g v_2^{(k)}+n v_2^{(k)}=
        -c_n^{-1}\divv_g\divv_g(\eta\kappa)&\text{on }\Omega_k,\\
        v_2^{(k)}=0&\text{on }\partial\Omega_k.
    \end{cases}\]
    The decomposition is unique for a fixed $\Omega_k$, since the Dirichlet problem for a positive operator on a smooth compact domain is unique. Now we establish $W^{1,2}$-estimates for the decomposition. 
    \begin{lemma}\label{lem:W12-v1v2}
        For each $k$,
        \begin{align}
        \|v_1^{(k)}\|_{W^{1,2}(\Omega_k)} &\le C_1,\label{eq:W12-v1}\\
        \|v_2^{(k)}\|_{W^{1,2}(\Omega_k)} &\le C_2\,\ell^{-1},\label{eq:W12-v2}
        \end{align}
        where $C_1$ depends on $(M^n,g)$ and $\ell$,
        while $C_2$ depends only on $(M^n,g)$; 
        neither constant depends on $k$.
    \end{lemma}
    Here, for simplicity we say $C_1$ and $C_2$ depend on $(M^n,g)$. In fact, they may also depend on the choice of $p,d,\psi$. However, these choices are fixed throughout the proof. 
    \begin{proof}
        For simplicity, we suppress the index $k$. Multiplying the equation by $v_1$ and integrating over $\Omega_k$, we have 
        \begin{align*}
            \|\nabla v_1\|_{L^2(\Omega_k)}^2+n\|v_1\|_{L^2(\Omega_k)}^2&=c_n^{-1}\int_{\Omega_k}\eta\|\kappa\|_g^2v_1,\\
            &\le c_n^{-1}\left(\int_{\supp\eta}\|\kappa\|_g^4\right)^{1/2}\|v_1\|_{L^2(\Omega_k)}\\
            &\le \frac{n}{2}\|v_1\|^2_{L^2(\Omega_k)}+A_n\int_{\supp\eta}\|\kappa\|^4_g,
        \end{align*}
        for some dimensional constant $A_n$. This shows that 
        \[
            \|v_1\|_{W^{1,2}(\Omega_k)}\le B_n\left(\int_{\supp\eta}\|\kappa\|_g^4\right)^{1/4}=:C_1
        \]
        for some dimensional constant $B_n$. From the explicit bound above, we see that $C_1$ may depend on $\ell$ since $\eta=\eta_{\ell}$. 

        Multiplying the equation by $v_2$ and integrating over $\Omega_k$, we have 
        \begin{align*}
            \|\nabla v_2\|_{L^2(\Omega_k)}^2+n\|v_2\|_{L^2(\Omega_k)}^2&=-c_n^{-1}\int_{\Omega_k}\divv_g\divv_g(\eta\kappa)v_2=c_n^{-1}\int_{\Omega_k}\langle\divv_g(\eta\kappa),\nabla v_2\rangle\\
            &=c_n^{-1}\int_{\Omega_k}\kappa(\nabla\eta,\nabla v_2)\le C_n\int_{\Omega_k}|\kappa|_gl^{-1}|\nabla v_2|\\
            &\le C_n\ell^{-1}\int_{\Omega_k}\left(\frac{C_n}{2\ell}|\kappa|_g^2+\frac{\ell}{2C_n}|\nabla v_2|^2\right)
        \end{align*}
        for some dimensional constant $C_n$. Hence, 
        \[
            \|v_2\|_{W^{1,2}(\Omega_k)}\le D_n\ell^{-1}\|\kappa\|_{L^2(M)}=:C_2\ell^{-1},
        \]
        for some dimensional constant $D_n$. From the explicit bound above, we see that $C_2$ depends only on the dimension $n$ and the $L^2$-norm of $\kappa$. 
    \end{proof}

    \textbf{Step 4: Global solutions of linearized equations.} By Lemma \ref{lem:W12-v1v2}, we can find a sequence $t_k\nearrow 0$ such that the following limits exist in $C^\infty_{\mathrm{loc}}(M)$: 
    \[
        u_{t_k}^{(k)}\to u,\qquad
        v_{t_k}^{(k)}\to v,\qquad
        v_1^{(k)}\to v_1,\qquad
        v_2^{(k)}\to v_2.
    \]
    The following lemma follows immediately from the definition. 
    \begin{lemma}\label{lem:global-limits}
        The limits $u,v,v_1,v_2\in C^\infty(M)$ satisfy:
        \begin{enumerate}
            \item[(i)] $u=1$ on $M$;
            \item[(ii)] $v$ satisfies
            \begin{equation}\label{eq:linearized-v}
                -\Delta_g v+n v
                =c_n^{-1}\Bigl(\eta\|\kappa\|_g^2
                -\divv_g\divv_g(\eta\kappa)\Bigr)\quad\text{on }M;
            \end{equation}
            \item[(iii)] $v_1,v_2\in W^{1,2}(M)$ satisfy
            \begin{align}
            -\Delta_g v_1+n v_1&=c_n^{-1}\eta\|\kappa\|_g^2
            \quad\text{on }M,\label{eq:global-v1}\\
            -\Delta_g v_2+n v_2&=-c_n^{-1}\divv_g\divv_g(\eta\kappa)
            \quad\text{on }M,\label{eq:global-v2}
            \end{align}
            \item[(iv)] $v=v_1+v_2$ on $M$.
        \end{enumerate}
    \end{lemma}
    The key issue is the behavior of the solutions near the conformal infinity $\partial X$. 
    \begin{lemma}\label{lem:v3-exterior-expansion}
        Let $u\in W^{1,2}(E_0)\cap L^\infty(\partial E_0)$ satisfy $Lu=F$, where $L=-\Delta_g+n$, $F$ extends smoothly to $\partial X$ and $F\in C^\infty_{2n}(E_0)$. Then
        \[
        u(\theta,\rho)=b(\theta)\rho^n+O(\rho^{n+1}),\qquad
        b\in C^\infty(\partial X).
        \]
        The remainder belongs to $C^\infty_{n+1}(E_0)$.
    \end{lemma}
    \begin{proof}
        We first show that $u$ must decay at $\partial X$. A computation similar to that in the proof of Proposition \ref{prop:gap} shows that 
        \[
            Lw=(n+2)\rho^{n+1}+O(\rho^{n+2}), \quad w=\rho^n-\rho^{n+1}.
        \]
        Since $F\in C^\infty_{2n}(E_0)$, there exists $\rho_0>0$ such that 
        \[
            w>0, \quad Lw\ge\max\{F,0\}, 
        \]
        and thus, for any $a>1$, $b<-1$, 
        \[
            L(bw)\le F\le L(aw).
        \]
        Choose $a,b$ so that
        \[
            bw<u<aw\quad \text{on }\partial E_0.
        \]
        Then the maximum principle shows that 
        \[
            bw<u<aw\quad \text{on }E_0.
        \]
        Hence, $u=O(\rho^{n})$. Elliptic regularity then implies that $u\in C^\infty_{n}(E_0)$. The argument in \cite[\S4]{ACF92} gives the tangential regularity that we need. 
    \end{proof}

    The functions $v_1,v_2$ and $v$ satisfy the assumption of Lemma \ref{lem:v3-exterior-expansion}. Moreover, $v_1>0$ due to $v_1^{(k)}>0$ and the strong maximum principle. Next, we establish a uniform lower bound of $v_1$ independent of the parameter $\ell$.

    \begin{lemma}\label{lem:v1-lower-bound}
        There exists a constant $a_0>0$ such that for some $\rho_0>0$, for all sufficiently large $\ell$, 
        the solution $v_1=v_{1,\ell}$ of~\eqref{eq:global-v1} satisfies
        \[
            v_{1,l}(\theta,\rho)=a_\ell(\theta)\rho^n+O(\rho^{n+1}), \quad \rho\in(0,\rho_0), \quad a_\ell>a_0.
        \]
        Here, $a_0$ and the higher-order terms are independent of $\ell$.
    \end{lemma}
    \begin{proof}
        Fix $\rho_0>0$ so that $E_0\cap\{\rho<2\rho_0\}$ can be identified with $\partial X\times (0,2\rho_0)$. We first establish a uniform lower bound of $v_1^{(k)}$ on $\Sigma:=E_0\cap\{\rho=\rho_0\}$, for all large $k$. 

        Let $G_k(x,y)$ be the Dirichlet Green function of $-\Delta_g+n$ on $\Omega_k$. Assume that $K$ is a compact neighborhood of $\Sigma\cup\{p\}$, and $K\subset\Omega_k$. We may assume that $\eta=1$ on $K$ by shrinking $K$ if needed. Then the integral representation of $v_1^{(k)}$ gives 
        \begin{align*}
            v_1^{(k)}(x)&=\int_{\Omega_k}G_k(x,y)c_n^{-1}\eta(y)\|\kappa(y)\|_g^2\,d\mu_g(y)\\
            &\ge\int_KG_k(x,y)c_n^{-1}\|\kappa(y)\|_g^2\,d\mu_g(y).
        \end{align*}
        Let $k\to\infty$, and passing to a subsequence, we have $G_k\to G>0$. Hence, 
        \[
            v_1(x)\ge c_n^{-1}\int_K G(x,y)\|\kappa(y)\|_g^2\,d\mu_g(y)\ge (\inf_{K\times K}G)c_n^{-1}\int_{K}\|\kappa\|_g^2\,d\mu_g=:m>0.
        \]
        Note that this $m$ is independent of $\ell$. 
        
        As a result, the barrier functions in the proof of Lemma \ref{lem:v3-exterior-expansion} can be made to be independent of $\ell$, namely, we take 
        \[
            \underline{v}_1=\rho^n+\rho^{n+1}, \quad (-\Delta+n)\underline{v}_1\le 0, \quad \text{on }E_0\cap\{\rho<2\rho_0\},
        \]
        and some $a_0>0$ such that 
        \[
            v_1>a_0\underline{v}_1 \quad \text{on }\Sigma.
        \]
        Hence, the maximum principle shows that $v_1>a_0\underline{v}_1$ on $E_0\cap\{\rho<\rho_0\}$ as required. 
    \end{proof}

    Next, we prove an upper bound for the leading coefficient of $v_2$. 
    \begin{lemma}\label{lem:b-estimate}
        There exist constants $b_0,\rho_0>0$, such that for sufficiently large $\ell$, the solution $v_2=v_{2,\ell}$ of~\eqref{eq:global-v2} satisfies
        \[
            v_{2,\ell}(\theta,\rho)=b_\ell\rho^n+O(\rho^{n+1}), \qquad \rho\in(0,\rho_0)
        \]
        and
        \[
            b_\ell\in C^\infty(\partial X), \qquad |b_\ell|\le b_0\ell^{-1}.
        \]
    \end{lemma}
    In fact, we will see in the proof that the remainder $O(\rho^{n+1})$ is also bounded by $C\ell^{-1}\rho^{n+1}$ for some constant $C$ independent of $\ell$. 
    \begin{proof}
        Fix $\rho_0>0$ so that $E_0\cap\{\rho<2\rho_0\}$ can be identified with $\partial X\times (0,2\rho_0)$. We first establish an upper bound of $v_2$ on $\Sigma:=E_0\cap\{\rho=\rho_0\}$ independent of $\ell$. 

        Cover $\Sigma$ by finitely many balls $\{B_r(x_i)\}$ with $B_{2r}(x_i)\subset E_0\cap\{\rho<2\rho_0\}$. Interior estimates for $-\Delta_g+n$ give 
        \[
            \sup_{B_r(x_i)}|v_{2,\ell}| \le C_i\|v_{2,\ell}\|_{L^2(B_{2r}(x_a))},\qquad
            \sup_\Sigma|v_{2,\ell}|\le C\|v_{2,\ell}\|_{L^2(M)}\le C\ell^{-1}.
        \]
        The constants $C_i,C$ depend only on the geometry near $\Sigma$, independent of $\ell$. 

        As before, set 
        \[
            w=\rho^n-\rho^{n+1}, \quad (-\Delta_g+n)w>0 \quad \text{on }E_0\cap\{\rho<2\rho_0\}.
        \]
        Then for sufficiently small $b>0$ (independent of $\ell$), we have 
        \[
            -b\ell^{-1}w\le v_{2,\ell}\le b\ell^{-1}w \quad \text{on }\Sigma.
        \]
        Hence, the maximum principle shows the desired estimate on $E\cap\{\rho<\rho_0\}$. 
    \end{proof}

    \textbf{Step 5: Global solutions of the Yamabe equations.} The second consequence of Lemma \ref{lem:W12-vt} is that we can take a convergent subsequence 
    \[
        u_t^{(k_{i,t})}\to u_t, \quad v_t^{k_{i,t}}\to v_t \qquad \text{in } C^\infty_{\mathrm{loc}}(M)
    \]
    for a fixed $|t|\le t_\ell$, and some sequence $k_{i,t}\to\infty$. Then
    \[
        v_t=t^{-1}(u_t-1), \qquad t\ne 0,
    \]
    and
    \begin{equation}\label{eq:v_t estimates}
        \|v_t\|_{L^\infty(M)}\le C_\ell, \qquad \|v_t\|_{W^{1,2}(M)}\le C_\ell.
    \end{equation}
    We now compare $v_t$ with the solution to the linearized equation $v=v_1+v_2$ that we investigated in Step 4. 
    \begin{lemma}\label{lem:vt-strong-linearization}
        There exist constants $t_{\ell},C_{\ell}>0$ such that for all $0<|t|\le t_{\ell}$,
        \begin{equation}\label{eq:v3-strong-linearization}
        \|v_t-v\|_{W^{1,2}(M)}\le C_\ell|t|.
        \end{equation}
        Moreover, for every compact domain $D\subset M$ and $0<\beta<1$,
        \[
            \|v_t-v\|_{C^{2,\beta}(D)}\le C_{\ell,D,\beta}|t|.
        \]
    \end{lemma}
    The norms are taken with respect to $g$ unless otherwise indicated. 
    \begin{proof}
        By definition, $v_t$ satisfies the equation
        \begin{equation}\label{eq:v3-vt-equation}
                c_n(-\Delta_{g_t}+n)v_t=-\hat R(g_t) v_t-t^{-1}\hat R(g_t)
                -n(n-1)t^{-1}\bigl((1+tv_t)^\alpha-1-\alpha tv_t\bigr),
        \end{equation}
        Then for any test function $\varphi\in C_c^\infty(M)$, \eqref{eq:v3-vt-equation} implies 
        \begin{equation}\label{eq:vt-weak}
            \begin{split}
                c_n\int_M\left(\langle dv_t,d\varphi\rangle_{g_t}+nv_t\varphi\right)\,d\mu_{g_t}=&-\int_M\hat R(g_t) v_t\varphi\,d\mu_{g_t}-\int_M t^{-1}\hat R(g_t) \varphi\,d\mu_{g_t}\\
                &-n(n-1)\int_Mt^{-1}\bigl((1+tv_t)^\alpha-1-\alpha tv_t\bigr)\varphi\,d\mu_{g_t}.
            \end{split}
        \end{equation}
        We rewrite this in terms of the metric $g$. Since $g_t=g+t\eta\kappa$, we have
        \[
            \|g_t-g\|_{L^\infty(M)}\le C_\ell |t|, \quad \|g_t^{-1}-g^{-1}\|_{L^\infty(M)}\le C_\ell|t|, \quad \left\|\frac{d\mu_{g_t}}{d\mu_{g}}-1\right\|_{L^\infty(M)}\le C_\ell|t|.
        \]
        Here, $g^{-1}$ denotes the inverse matrix of $g$ in local coordinates, and $C_\ell$ denotes a constant that depends on the parameter $\ell$ (and various choices in this section), independent of $t$. It follows that 
        \begin{equation}\label{eq:vt-weak-1}
            \begin{split}
                \left|\int_M\left(\langle dv_t,d\varphi\rangle_{g_t}+nv_t\varphi\right)\,d\mu_{g_t}-\int_M\left(\langle dv_t,d\varphi\rangle_{g}+nv_t\varphi\right)\,d\mu_{g}\right|\\
                \le C_\ell|t|\|v_t\|_{W^{1,2}(M)}\|\varphi\|_{W^{1,2}(M)}.
            \end{split}
        \end{equation}
        By Lemma \ref{lem:g_t curvature}, 
        \[
            |\hat R(g_t) -tDR(g)(\eta\kappa)|\le C_\ell t^2, \qquad DR(g)(\eta\kappa)=-\eta\|\kappa\|_g^2+\divv_g\divv_g(\eta\kappa),
        \]
        and it follows that
        \begin{equation}\label{eq:vt-weak-2}
            \left|\int_M\hat R(g_t) v_t\varphi\,d\mu_{g_t}-t\int_M DR(g)(\eta\kappa)v_t\varphi\,d\mu_{g}\right|
            \le C_\ell|t|\|v_t\|_{L^2(M)}\|\varphi\|_{L^2(M)},
        \end{equation}
        and
        \begin{equation}\label{eq:vt-weak-3}
            \left|\int_M t^{-1}\hat R(g_t) \varphi\,d\mu_{g_t}-\int_M DR(g)(\eta\kappa)\varphi\,d\mu_{g}\right|\le C_\ell|t|\|\varphi\|_{L^2(M)}.
        \end{equation}
        The Taylor expansion gives 
        \[
            |t^{-1}\bigl((1+tv_t)^\alpha-1-\alpha tv_t\bigr)|\le C_\ell t^2v_t,
        \]
        where we have used the $L^\infty$ bound of $v_t$, cf.\ \eqref{eq:v_t estimates}, and it follows that 
        \begin{equation}\label{eq:vt-weak-4}
            \left|\int_Mt^{-1}\bigl((1+tv_t)^\alpha-1-\alpha tv_t\bigr)\varphi\,d\mu_{g_t}\right|\le C_\ell|t|\|\varphi\|_{L^2(M)}.
        \end{equation}

        On the other hand, 
        \begin{equation}\label{eq:v-equation}
            c_n(-\Delta_g+n)v=DR(g)(\eta\kappa).
        \end{equation}
        Then for any test function $\varphi\in C^\infty_c(M)$, \eqref{eq:v-equation} implies
        \begin{equation}\label{eq:v-weak}
            c_n\int_M(\langle dv,d\varphi\rangle+nv\varphi)\,d\mu_g=\int_M DR(g)(\eta\kappa)\varphi\,d\mu_g.
        \end{equation}
        Subtract it from \eqref{eq:vt-weak}; using \eqref{eq:vt-weak-1}--\eqref{eq:vt-weak-4} and the $W^{1,2}$-bound of $v$ and $v_t$, we obtain
        \[
            \left|c_n\int_M(\langle d(v_t-v),d\varphi\rangle+n(v_t-v)\varphi)\,d\mu_g\right|\le C_\ell |t|\|\varphi\|_{W^{1,2}(M)}.
        \]

        The above estimate also holds if $\varphi\in W^{1,2}(M)$ by a density argument. Hence, we can take $\varphi=v_t-v$ and \eqref{eq:v3-strong-linearization} follows. Lastly, the $C^{2,\beta}(D)$-bound of $|v_t-v|$ follows from the $W^{1,2}$-bound and elliptic estimates. 
    \end{proof}

    Next, we establish the decay of $v_t$ near the conformal infinity. 
    \begin{lemma}\label{lem:vt-weighted-decay}
        There exist constants $t_\ell,\rho_0>0$ such that, for any $|t|<t_\ell$, 
        \[
            |v_t|\le C_\ell\rho^n\quad\text{on }E_{\rho_0}:=E_0\cap\{\rho<\rho_0\},
        \]
        and that for any $m\ge0$ and $0<\beta<1$,
        \begin{equation}\label{eq:vt-weighted-decay}
            \|v_t\|_{C^{m,\beta}_n(E_{\rho_0})}\le C_{\ell,m,\beta}.
        \end{equation}
    \end{lemma}
    The key point is that the estimates are uniform in $t$. 
    \begin{proof}
        Similar to the proof of Lemma \ref{lem:YamabeAH}, we construct sub- and super-solutions to $u_t^{(k)}$ as follows. For $\eta_0<n-1$ and a small $\rho_0>0$, we have (cf.\ Lemma \ref{lem:kappa expansion})
        \[
            |\nabla^k\rho^{\eta_0}|_g\le C\rho^{\eta_0}, \quad |\nabla^k\kappa|_g\le C_\ell\rho^n, \quad \rho\in(0,\rho_0), 
        \]
        and thus, using Lemma \ref{lem:1st variation of Laplacian}, 
        \[
            |(\Delta_{g_t}-\Delta_g)\rho^{\eta_0}|\le C_\ell|t|\rho^{n+\eta_0}, \quad \rho\in(0,\rho_0), \quad |t|<t_\ell.
        \]
        Hence, for some $c_\ell,C_\ell>0$, 
        \[
            c_\ell \rho^{\eta_0}\le -\Delta_{g_t}\rho^{\eta_0}\le C_\ell \rho^{\eta_0}, \quad \rho\in(0,\rho_0), \quad |t|<t_\ell.
        \]
        For any $u\in C^2(M)$, write 
        \begin{align*}
            L_tu&=-c_n\Delta_{g_t}u+R(g_t)u+n(n-1)u^\alpha\\
            &=-c_n\Delta_{g_t}u+(R(g_t)+n(n-1))u+n(n-1)(u^\alpha-u).
        \end{align*}
        Then for any $a\ge 1$ and $\rho\in(0,\rho_0)$, 
        \[
            L_t(1-a\rho^{\eta_0})\le -ac_\ell\rho^{\eta_0}+n(n-1)((1-a\rho^{\eta_0})^\alpha-(1-a\rho^{\eta_0}))\le 0, \quad \text{if }0\le a\rho^{\eta_0}\le 1, 
        \]
        and
        \[
            L_t(1+a\rho^{\eta_0})\ge ac_\ell\rho^{\eta_0}+n(n-1)((1+a\rho^{\eta_0})^\alpha-(1+a\rho^{\eta_0})) \ge 0.
        \]
        As before, for any $a>0$, we have a sub-solution
        \[
            \underline{u}=\max\{1-a\rho^{\eta_0},0\}, 
        \]
        and a super-solution
        \[
            \bar{u}=\min\{1+a\rho^{\eta_0}, C_{\ell,a}\}.
        \]
        Hence, 
        \[
            \underline{u}\le u^{(k)}_t\le \bar{u};
        \]
        passing to the limit, we have 
        \[
            |u_t-1|\le a\rho^{\eta_0}.
        \]

        Lemma \ref{lem:g_t curvature} gives
        \[
            |\hat R(g_t)|\le C_\ell\rho^{2n}.
        \]
        The decay rate of $\hat R(g_t)$ allows us to improve the decay rate of $tv:=u-1$. Indeed, let 
        \[
            \bar{L}_t=-\Delta_{g_t}-c_n^{-1}n(n-1)\left(1-\alpha\int_0^1(1+sw)^{\alpha-1}\,ds\right)=-\Delta_{g_t}+n+O(\rho^{\eta_0}). 
        \]
        We have 
        \[
            \bar{L}_t(tv)=-c_n^{-1}\hat R(g_t) (1+tv)
        \]
        and thus 
        \[
            |\bar{L}_t(tv)|\le C_\ell \rho^{2n}. 
        \]
        Let $w=\rho^n-\rho^{n+1}$. Then 
        \[
            \bar{L}_t\ge c_\ell\rho^{n+1}, \quad \rho\in(0,\rho_0).
        \]
        Thus, choose $a>0$ (depending on $l$, but not on $t$) so that 
        \[
            -a|t|w\le tv\le a|t|w \quad \text{on }\partial E_{\rho_0}.
        \]
        Then the maximum principle implies that 
        \[
            -a|t|w\le tv\le a|t|w \quad \text{on }E_{\rho_0}.
        \]
        Thus, we conclude that 
        \[|tv|\le C_\ell|t|\rho^n.\]

        The second part \eqref{eq:vt-weighted-decay} then follows from standard elliptic estimates. 
    \end{proof}

    \begin{lemma}\label{lem:vt-leading-coefficient}
        There exists $t_\ell>0$ such that for any $|t|<t_\ell$, 
        \[
            u_t(\theta,\rho)=1+A_t(\theta)\rho^n+B_t(\theta,\rho), \quad A_t\in C^\infty(\partial X), \quad B_t\in C^\infty_{n+1}(E_0), 
        \]
        and for $t\ne 0$, 
        \[
            v_t(\theta,\rho)=\frac{A_t(\theta)}t\rho^n+\frac{B_t(\theta,\rho)}{t}.
        \]
        Moreover,
        \begin{equation}\label{eq:v3-coefficient-remainder}
            \|A_t-t(a_\ell+b_\ell)\|_{C^0(\partial X)}\le C_\ell t^2,
        \end{equation}
        with $C_\ell$ independent of $t$. 
        
        Here, $a_\ell$ and $b_\ell$ are defined in Lemmas \ref{lem:v1-lower-bound} and \ref{lem:b-estimate}. 
    \end{lemma}
    \begin{proof}
        The expansion of $u_t$ follows from Lemma \ref{lem:vt-weighted-decay} and the argument in \cite[\S4]{ACF92}; see also Proposition \ref{prop:tangential smoothness}. 

        Since $v_t=t^{-1}(u_t-1)$, we compare \eqref{eq:v3-vt-equation} with \eqref{eq:linearized-v}: 
        \begin{align*}
            (-\Delta_g+n)(v_t-v)=&(\Delta_{g_t}-\Delta_g)v_t-c_n^{-1}(t^{-1}\hat{R}(g_t)-DR(g)(\eta\kappa))-c_n^{-1}\hat{R}(g_t)v_t\\
            &-c_n^{-1}n(n-1)t^{-1}((1+tv_t)^\alpha-1-\alpha tv_t).
        \end{align*}
        On $E_{\rho_0}:=E_0\cap\{\rho<\rho_0\}$, we bound each term on the right-hand side (at some point $(\theta,\rho)$): Lemmas \ref{lem:vt-weighted-decay} and \ref{lem:1st variation of Laplacian} imply that
        \[
            |(\Delta_{g_t}-\Delta_g)v_t|\le C_\ell |t| \|v_t\|_{C^2(E_\rho)}\rho^n\le C_\ell |t|\rho^{2n}.
        \]
        Lemma \ref{lem:g_t curvature} says
        \[
            |t^{-1}\hat{R}(g_t)-DR(g)(\eta\kappa)|\le C_\ell |t| \rho^{2n}. 
        \]
        The Taylor expansion, the bound \eqref{eq:v_t estimates} and Lemma \ref{lem:vt-weighted-decay} imply 
        \[
            |t^{-1}((1+tv_t)^\alpha-1-\alpha tv_t)|\le C_\ell |t|v_t^2\le C_\ell |t|\rho^{2n}.
        \]
        Hence, 
        \[
            |(-\Delta_g+n)(v_t-v)|\le C_\ell |t|\rho^{2n}. 
        \]

        On $\partial E_{\rho_0}$, \eqref{eq:v3-strong-linearization} gives 
        \[
            |v_t-v|\le C_\ell |t|.
        \]
        Repeat the comparison argument in the proof of Lemma \ref{lem:b-estimate}, and we obtain 
        \[
            |v_t-v|\le C_\ell |t|\rho^n \quad\text{on }E_{\rho_0}. 
        \]
        Hence,
        \[
            |u_t-1-tv|\le C_\ell t^2\rho^n.
        \]
        Dividing by $\rho^n$ and letting $\rho\to 0$, we obtain the desired inequality. 
    \end{proof}
    
    \textbf{Step 6: Change of mass under the conformal deformation.} We now analyze the mass of $g_t=g+t\eta\kappa$ and $\tilde{g}_t=u^{4/(n-2)}g_t$. 

    \begin{lemma}\label{lem:g_t mass aspect}
        Up to a change of coordinates, the metric $g_t$ is asymptotically hyperbolic, with mass aspect
        \begin{equation}\label{eq:v3-metric-mass}
            \begin{aligned}
                h(g_t)&=h+t\bigl((\operatorname{tr}_{\gamma_{\mathrm{std}}}h)
                                \gamma_{\mathrm{std}}-\tfrac n2h\bigr),\\
                \operatorname{tr}_{\gamma_{\mathrm{std}}}h(g_t)
                &=\bigl(1+\tfrac{n-2}{2}t\bigr)
                \operatorname{tr}_{\gamma_{\mathrm{std}}}h.
            \end{aligned}
        \end{equation}
    \end{lemma}
    \begin{proof}
        By Lemma \ref{lem:kappa expansion}, we can write 
        \[
            g_t=\sinh^{-2}\rho\left(\gamma_{\mathrm{std}}+\frac{\rho^n}{n}h_1(t)+A_1(t)\rho^n\,d\rho^2+O(\rho^{n+1})\right),
        \]
        where 
        \[
            h_1(t)=h+\frac{nt}{2}(\gamma_{\mathrm{std}}\operatorname{tr}_{\gamma_{\mathrm{std}}}h-h), 
        \]
        \[
            A_1(t)=-\frac{(n-2)t}{2}\operatorname{tr}_{\gamma_{\mathrm{std}}}h.
        \]

        By Lemma \ref{lem:AH coordinates}, a change of coordinates shows that $g_t$ is AH, with 
        \[
            h(g_t)=h_1(t)+A_1(t)\gamma_{\mathrm{std}}=h+t\bigl((\operatorname{tr}_{\gamma_{\mathrm{std}}}h)\gamma_{\mathrm{std}}-\tfrac n2h\bigr),
        \]
        Taking the trace gives the second part of \eqref{eq:v3-metric-mass}. 
    \end{proof}

    Since $\EADM(g)=|\PADM(g)|$, we fix a unit vector $e$ such that 
    \[
        \EADM(g)-e\cdot\PADM(g)=0.
    \]
    For $\widetilde g_t=u_t^{4/(n-2)}g_t$, the expansion of $u_t$ (Lemma \ref{lem:vt-leading-coefficient}) with Proposition \ref{prop:mass aspect change} now gives
    \[
    h(\widetilde g_t)=h(g_t)+\frac{4(n+1)}{n-2}A_t\gamma_{\mathrm{std}}.
    \]
    (The leading term $A_t$ is unchanged by the preceding
    change of coordinates, detailed in the proof of Lemma \ref{lem:AH coordinates}.) Hence, 
    \begin{align}
        \EADM(\widetilde g_t)
        &=\bigl(1+\tfrac{n-2}{2}t\bigr)\EADM(g)
            +c_n(n+1)\int_{\partial X}A_t\,d\mu_{\gamma_{\mathrm{std}}},
            \label{eq:mass-change}\\
        \PADM(\widetilde g_t)
        &=\bigl(1+\tfrac{n-2}{2}t\bigr)\PADM(g)
            +c_n(n+1)\int_{\partial X}A_t\theta\,d\mu_{\gamma_{\mathrm{std}}}.
            \label{eq:mom-change}
    \end{align}
    Combining these identities with~\eqref{eq:v3-coefficient-remainder}
    gives 
    \[
    \lim_{t\to0^-}\frac{\EADM(\widetilde g_t)-e\cdot\PADM(\widetilde g_t)}{t}=c_n(n+1)\int_{\partial X}(a_\ell+b_\ell)(1-e\cdot\theta)\,d\mu_{\gamma_{\mathrm{std}}}.
    \]
    
    Lemmas~\ref{lem:v1-lower-bound} and~\ref{lem:b-estimate} give
    \[
        \int_{\partial X}(a_\ell+b_\ell)(1-e\cdot\theta)
        \,d\mu_{\gamma_{\mathrm{std}}}\ge\int_{\partial X}a_0(1-e\cdot\theta)
        \,d\mu_{\gamma_{\mathrm{std}}}-\frac{C}\ell,
    \]
    for some constant $C$ independent of $\ell$. 

    Now we choose $\ell$ so large that 
    \begin{equation}\label{eq:ell condition}
        \int_{\partial X}(a_\ell+b_\ell)(1-e\cdot\theta)
        \,d\mu_{\gamma_{\mathrm{std}}}\ge\frac{1}{2}\int_{\partial X}a_0(1-e\cdot\theta)
        \,d\mu_{\gamma_{\mathrm{std}}},
    \end{equation}
    This $\ell$ is now fixed. For $-t_\ell<t<0$, where $t_\ell$ is small enough, we have
    \begin{equation}\label{eq:F-upper}
        \begin{split}
        \EADM(\widetilde g_t)-|\PADM(\widetilde g_t)|&\le\EADM(\widetilde g_t)-e\cdot\PADM(\widetilde g_t)\\
        &=t c_n(n+1)\int_{\partial X}
            (a_\ell+b_\ell)(1-e\cdot\theta)\,d\mu_{\gamma_{\mathrm{std}}}
            +O(t^2)\\
        &\le-\frac{c_n(n+1)}4|t|\int_{\partial X}
            a_0(1-e\cdot\theta)\,d\mu_{\gamma_{\mathrm{std}}}<0,
        \end{split}
    \end{equation}
    which contradicts Theorem~\ref{thm:PMT scalar rigidity}. Therefore, our assumption $\kappa(p)\ne 0$ was impossible, and this proves the Ricci rigidity.
\end{proof}

\subsection{Proof of Ricci rigidity in the spin case}\label{sec:Ricci rigid spin}

Let $M\setminus\mathcal{S}$ be spin. Assume the same conditions as in Theorem \ref{thm:pmt AH}, and that the inequality in \eqref{eq:mass nonnegative} is an equality. We want to show that $(M\setminus\mathcal{S}, g)$ is Einstein, namely, $\Ric_g\equiv-(n-1)g$.

\textbf{Step 1: Solve the Dirac equation on the blown-up manifold.} Let $U_j=\{x\in M: G^{\frac{1}{2+k-n}}(x)<a_j\}$ be as before, and $\tilde{g}_j=u_j^{\frac{4}{n-2}}g$ be the conformal metric on $M\setminus\bar{U}_j$. Here, $u_j$ is defined as in Lemma \ref{lem:u_j arbitrary end}, i.e., 
$$
    \begin{cases}
        -c_n\Delta_gu_j-n(n-1)u_j+n(n-1)u_j^\alpha=0 &\text{on }M\setminus\bar{U}_j, \\
        u_j(x)r_j(x)^{\frac{n-2}{2}}\to 1 &\text{as }x\to\partial U_j, \\
        u_j\ge a_i>0 &\text{on }E_i, \\
        u_j=1+A_j\rho^n+O(\rho^{n+1}) &\text{near }\partial X,
    \end{cases}
$$
for $A_j\in C^\infty(\partial X)$ with $|A_j|\le A$ for all $j$, and some constants $a_i,A>0$ independent of $j$.

We know that $u_j\to u_\infty$ smoothly on compact subsets of $M\setminus\mathcal{S}$ by Proposition \ref{prop:bound for u_infty}, and $u_\infty\le 1$. Then Proposition \ref{prop:gap} implies that if $u_\infty\not\equiv 1$, then there exists $\delta>0$ such that $u\le 1-\delta\rho^n$ near $\partial X$. The same argument as in the proof of Proposition \ref{prop:PMT strict} shows that for some $j$, we have $u_j\le 1-\frac{\delta}{2}\rho^n$ near $\partial X$. By Corollary \ref{cor:mass change}, we see that 
$$
    \EADM(\tilde{g}_j)-|\PADM(\tilde{g}_j)|<\EADM(g)-|\PADM(g)|=0,
$$
which contradicts the positive mass theorem for the AH manifold $(M\setminus\bar{U}_j, \tilde{g}_j)$. Hence, we must have $u_\infty\equiv 1$. This will be assumed throughout this subsection. 

\begin{lemma}
    Let $r_j(x)=\dist_g(x,\partial U_j)$. Near $\partial U_j$, we have 
    $$
        \Ric_{\tilde{g}_j}+(n-1)\tilde{g}_j=O(r_j^{-1}),\quad \|\Ric_{\tilde{g}_j}+(n-1)\tilde{g}_j\|_{\tilde{g}_j}=O(r_j),
    $$
    and the mean curvature of the level set $\{r_j=a\}$ with respect to $\tilde{g}_j$ satisfies
    $$
        H_{\tilde{g}_j}(\{r_j=a\})=(n-1)+O(a).
    $$
\end{lemma}
\begin{proof}
    Since $u_j=r_j^{-\frac{n-2}{2}}(1+O(r_j))$ near $\partial U_j$ (Lemma \ref{lem:u_j arbitrary end}), we have
    $$
        \tilde{g}_j=r_j^{-2}(1+O(r_j))g.
    $$
    Hence, $(M\setminus\bar{U}_j, \tilde{g}_j)$ is conformally compact and the sectional curvature of $\tilde{g}_j$ approaches $-1$ as $r_j\to 0$. If we write $\tilde{g}_j=r_j^{-2}g_j$, then the conformal change of Ricci curvature gives
    $$
        \Ric_{\tilde{g}_j}=\Ric_{g_j}+r_j^{-1}((n-2)\nabla^2_{g_j}r_j+(\Delta_{g_j}r_j)g_j)-(n-1)r_j^{-2}|dr_j|^2_{g_j}g_j.
    $$
    The terms $\Ric_{g_j}$, $\nabla^2_{g_j}$, $\Delta_{g_j}r_j$ and $|dr_j|^2_{g_j}$ are bounded as $r_j\to 0$, so
    $$
        \Ric_{\tilde{g}_j}+(n-1)\tilde{g}_j=O(r_j^{-1}).
    $$
    Since the inverse matrix of $\tilde{g}_j$ is $O(r_j^2)$, we have
    $$
        \|\Ric_{\tilde{g}_j}+(n-1)\tilde{g}_j\|_{\tilde{g}_j}=O(r_j).
    $$

    The mean curvature of the level set $\{r_j=a\}$ with respect to $\tilde{g}_j$ is given by
    $$
        H_{\tilde{g}_j}(\{r_j=a\})=r_jH_{g_j}(\{r_j=a\})+(n-1)dr_j(\nu)=(n-1)+O(a),
    $$
    where $\nu$ is the unit normal vector field with respect to $g_j$, pointing into $\tilde{M}_j$. 
\end{proof}

Thus, we can find a slightly larger $V_j\supset U_j$ such that the mean curvature of $\partial V_j$ with respect to $\tilde{g}_j$ is positive. We denote $\tilde{M}_j=(M\setminus\bar{V}_j, \tilde{g}_j)$. Then $\tilde{M}_j$ is a smooth, complete, asymptotically hyperbolic spin manifold with mean convex boundary. 

Let $\tilde{\slashed{S}}_j$ be the spinor bundle on $\tilde{M}_j$, and define a modified connection on $\tilde{\slashed{S}}_j$ by
$$
    \hat{\nabla}_X=\nabla_X+\frac{\sqrt{-1}}{2}c(X)
$$
for any vector field $X$ on $\tilde{M}_j$, where $\nabla=\nabla^{\tilde{g}_j}$ is the spinor connection on $\tilde{\slashed{S}}_j$ and $c(X)=c_{\tilde{g}_j}(X)$ is the Clifford multiplication. The modified Dirac operator is 
$$
    \hat{D}=\sum_{i=1}^n c(e_i)\hat{\nabla}_{e_i}=D-\frac{\sqrt{-1}}{2}n,
$$
where $\{e_i\}_{i=1}^n$ is a local orthonormal frame on $\tilde{M}_j$. The Lichnerowicz formula reads 
$$
    \hat{D}^*\hat{D}=\hat{\nabla}^*\hat{\nabla}+\frac{1}{4}(R_{\tilde{g}_j}+n(n-1))=\hat{\nabla}^*\hat{\nabla},
$$
since $u_j$ solves the Yamabe equation and $R_{\tilde{g}_j}\equiv-n(n-1)$. Integrating the Lichnerowicz formula on a compact domain $\Omega$, we have 
\begin{equation}\label{eq:Lichnerowicz}
    \int_{\Omega}(|\hat{\nabla}\psi|^2_{\tilde{g}_j}-|\hat{D}\psi|^2_{\tilde{g}_j})\,d\mu_{\tilde{g}_j}=\int_{\partial\Omega}B_{\tilde\nu}(\psi)\,d\sigma_{\tilde{g}_j},
\end{equation}
where $\tilde\nu$ is the outward unit normal with respect to $\tilde{g}_j$, $d\sigma_{\tilde{g}_j}$ is the induced volume form on $\partial\Omega$, and the boundary term $B$ is defined by 
$$
    B_Y(\varphi)=\langle\hat{\nabla}_Y\varphi+c(Y)\hat{D}\varphi,\varphi\rangle_{\tilde{g}_j}.
$$
As $\Omega$ increases to $\tilde{M}_j$, we need to examine the boundary terms that appear in the AH end, the arbitrary ends, and the boundary $\partial\tilde{M}_j$. Thus, to obtain a meaningful result, it is necessary to restrict to a special class of spinors.

Recall the following result from \cite{bartnik2005}:
\begin{lemma}
    There exists $w_j\in L^1_{\mathrm{loc}}(\tilde{M}_j)$ that is strictly positive and satisfies
    \begin{equation}
        \int_{\tilde{M}_j}|\psi|^2_{\tilde{g}_j}w_j\,d\mu_{\tilde{g}_j}\le\int_{\tilde{M}_j}|\hat{\nabla}\psi|^2_{\tilde{g}_j}\,d\mu_{\tilde{g}_j}
    \end{equation}
    for all $\psi\in C^\infty_c(\tilde{M}_j,\tilde{\slashed{S}}_j)$. Moreover, $w_j$ can be chosen to be constant on the AH end of $\tilde{M}_j$ (independent of $j$).
\end{lemma}
\begin{proof}
    See, for example, Proposition 8.3(4) and Theorem 9.10 in \cite{bartnik2005}. Here, we record one definition of $w_j$ for later use. Let $\tilde{g}_j=r^{-2}(dr^2+\gamma'_r)$ for some geodesic defining function $r$, and a family of metrics $\{\gamma'_r\}$ on $\partial X$. 

    For any $\psi\in C^\infty_c(\tilde{M}_j,\slashed{S}_j)$, we take $\chi\in C^\infty_c(\tilde{M}_j)$ with $0\le\chi\le 1$ and $\chi=1$ on the support of $\psi$. For $\epsilon>0$, set 
    $$
        f_\epsilon=\chi\sqrt{\epsilon+|\psi|^2}.
    $$
    We first compute with respect to the volume form 
    $$
        d\mathring{\mu}=r^{-n}\,drd\mu_{h_0},
    $$
    on a small neighborhood $U\cong (0,r_0)\times\partial X$ of the conformal infinity. Without loss of generality, we can assume $r_0<1$. A standard calculation (cf.\ \cite[Theorem 9.10]{bartnik2005}) leads to 
    \begin{align*}
        \int_{U}f_\epsilon^2\,d\mathring{\mu} &\le \frac{4}{(n-1)^2}\int_{U}\left(r\frac{\partial f_\epsilon}{\partial r}\right)^2\,d\mathring{\mu}\\
        &\le \frac{4}{(n-1)^2}\int_{U}r^2(\epsilon|d\chi|^2+\chi^2|\nabla_r \psi|^2)\,d\mathring{\mu}.
    \end{align*}
    Let $\epsilon\to 0$, and we obtain 
    \begin{align*}
        \int_U |\psi|^2\,d\mathring{\mu}&\le \frac{4}{(n-1)^2}\int_Ur^2|\nabla_r\psi|^2\,d\mathring{\mu}\\
        &\le \frac{4}{(n-1)^2}\int_U \left((1+\delta^{-1})r^2|\hat{\nabla}\psi|^2+(1+\delta)\frac{r^2}{4}|c(\partial_r)\psi|^2\right)\,d\mathring{\mu}\\
        &\le \frac{4(1+\delta^{-1})}{(n-1)^2}\int_U|\hat{\nabla} \psi|^2\,d\mathring{\mu}+\frac{(1+\delta)}{(n-1)^2}\int_U |\psi|^2\,d\mathring{\mu},
    \end{align*}
    where $\delta=\delta(n)>0$ is fixed so that
    $$
        \frac{(1+\delta)}{(n-1)^2}<1.
    $$
    Hence, we have 
    $$
        \int_U|\psi|^2\,d\mathring{\mu}\le C(n)\int_U|\hat{\nabla}\psi|^2\,d\mathring{\mu}.
    $$
     Since $\tilde{g}_j$ is AH, we have 
    $$
        d\mu_{\tilde{g}_j}=(1+o(r))\,d\mathring{\mu}.
    $$
    Hence, for $r_0$ sufficiently small, we have 
    $$
        \int_U|\psi|^2\,d\mu_{\tilde{g}_j}\le C(n)\int_U|\hat\nabla \psi|^2\,d\mu_{\tilde{g}_j}.
    $$

    To extend this to the entire manifold, take a locally finite open cover $\{\Omega_l\}$ of $\tilde{M}_j$, where each $\Omega_l$ is bounded. Lemma 9.1 in \cite{bartnik2005} gives 
    $$
        \epsilon_l\int_{\Omega_l}|\psi|^2\,d\mu_{\tilde{g}_j}\le\int_{U}|\psi|^2\,d\mu_{\tilde{g}_j}+\int_{\tilde{M}_j} |\hat\nabla \psi|^2\,d\mu_{\tilde{g}_j},
    $$
    for some $\epsilon_l\in(0,1)$ depending on $\tilde{M}_j$, $\Omega_l$ and $U$. Lastly, we define 
    $$
        w_j(x)=\frac{1}{1+C(n)}\sum_{\{l:x\in\Omega_l\}}2^{-l}\epsilon_l,
    $$
    and the required inequality holds. 
\end{proof}

The above construction allows us to obtain a uniform lower bound for $w_j$ on compact subsets. 

\begin{proposition}\label{prop:uniform wPI}
    We can choose $w_j$ with the following property: If $\Omega\Subset\mathcal{R}$ is a bounded open subset, then there is a uniform constant $w(\Omega)>0$ such that 
    $$
        w_j\ge w(\Omega)
    $$
    for sufficiently large $j$. 
\end{proposition}
\begin{proof}
    For $j$ large enough, we have $\Omega\Subset\tilde{M}_{j_0}$ for some $j_0$. When choosing the open cover $\{\Omega_{j,l}:l\in\mathbb{Z}_{>0}\}$ of $\tilde{M}_j$, we can always use the same, finitely many, open sets to cover $\Omega$, for all $j\ge j_0$. Thus, it suffices to show that, for any such open set $\Omega'$, there is a uniform positive constant $\epsilon$ such that 
    $$
        \epsilon\int_{\Omega'}|\psi|^2\,d\mu_{\tilde{g}_{j}}\le\int_{U}|\psi|^2\,d\mu_{\tilde{g}_{j}}+\int_{\tilde{M}_{j}} |\hat\nabla \psi|^2\,d\mu_{\tilde{g}_{j}},
    $$
    where the norm $|\cdot|=|\cdot|_{\tilde{g}_j}$. 

    From the proof of Lemma 9.1 in \cite{bartnik2005}, we can obtain $\epsilon_j>0$ so that 
    \begin{equation}\label{eq:PI-1}
        \epsilon_j\int_{\Omega'}|\psi|^2\,d\mu_{\tilde{g}_{j}}\le\int_{U}|\psi|^2\,d\mu_{\tilde{g}_{j}}+\int_{N} |\hat\nabla \psi|^2\,d\mu_{\tilde{g}_{j}},
    \end{equation}
    for some neighborhood $N\subset\tilde{M}_j$ of paths from a point in $U$ and points in $\Omega'$. In particular, we can make $N\subset\tilde{M}_{j_0}$ satisfy the following 
    \begin{itemize}
        \item $N\supset \Omega'\cup U$, and 
        \item $N\setminus E_0$ is precompact.
    \end{itemize} 
    We may assume that $\epsilon_j$ is large enough so that $2\epsilon_j$ violates \eqref{eq:PI-1}, i.e., there exists $\psi_j\in C^\infty_c(\tilde{M}_j,\slashed{S}_j)$ such that 
    \begin{equation}\label{eq:PI-2}
        2\epsilon_j\int_{\Omega'}|\psi_j|^2\,d\mu_{\tilde{g}_{j}}>\int_{U}|\psi_j|^2\,d\mu_{\tilde{g}_{j}}+\int_{N} |\hat\nabla \psi_j|^2\,d\mu_{\tilde{g}_{j}},
    \end{equation}
    and we normalize $\psi_j$ so that 
    \begin{equation}\label{eq:PI-normalization}
        \int_N |\psi_j|^2\,d\mu_{\tilde{g}_j}=1,
    \end{equation}
    and it follows that 
    $$
        \int_N |\hat\nabla\psi_j|^2\,d\mu_{\tilde{g}_j}<2\epsilon_j.
    $$

    Since $u_j\to u_\infty\equiv 1$, we have $\tilde{g}_j\to g$ uniformly on $N$. We claim that $\epsilon_j$ is bounded from below by a positive constant. If not, assume the contrary that $\epsilon_j\to 0$, by passing to a subsequence if necessary. We can transplant $\psi_j$ to a spinor $\psi_j'$ on $\mathcal{R}$ via an isometry; see Step 2 below for details. The $H^1$-bound of $\psi_j$ then implies that $\psi_j'\to\psi$ after passing to a further subsequence. Then \eqref{eq:PI-2} gives $\psi\equiv 0$ on $N$, since $\epsilon_j\to 0$. This is impossible since we have assumed \eqref{eq:PI-normalization}. 
\end{proof}

\begin{corollary}
    There exists a strictly positive function $w\in L^\infty_{\mathrm{loc}}(\mathcal{R})$ such that 
    \begin{equation}\label{eq:wPI}
        \int_{\tilde{M}_j}|\psi|^2_{\tilde{g}_j}w\,d\mu_{\tilde{g}_j}\le\int_{\tilde{M}_j}|\hat{\nabla}\psi|^2_{\tilde{g}_j}\,d\mu_{\tilde{g}_j}
    \end{equation}
    for all $\psi\in C^\infty_c(\tilde{M}_j,\tilde{\slashed{S}}_j)$. Moreover, $w$ can be chosen to be constant on the AH end of $\tilde{M}_j$.
\end{corollary}
\begin{proof}
    We take $w_j$ as in Proposition \ref{prop:uniform wPI}, and then $w(x)=\inf w_j(x)$ is the desired function.
\end{proof}

\begin{definition}
    The weighted Sobolev space $H^1_{w}=H^1_{w}(\tilde{M}_j,\tilde{\slashed{S}}_j)$ is the completion of $C^\infty_c(\tilde{M}_j,\tilde{\slashed{S}}_j)$ with respect to the norm
    $$
        \|\psi\|_{H^1_{w}}^2=\int_{\tilde{M}_j}\left(|\hat{\nabla}\psi|^2_{\tilde{g}_j}+|\psi|^2_{\tilde{g}_j}w\right)\,d\mu_{\tilde{g}_j}.
    $$
\end{definition}
The norm is equivalent to the homogeneous Sobolev norm
$$
    \|\psi\|_{\dot{H}^1}^2:=\int_{\tilde{M}_j}|\hat{\nabla}\psi|^2_{\tilde{g}_j}\,d\mu_{\tilde{g}_j}.
$$
Since $\tilde{M}_j$ is complete, the space $H^1_{w}$ can be identified with the subspace of measurable spinors $\psi$ on $\tilde{M}_j$ such that
$$
    \int_{\tilde{M}_j}(|\hat{\nabla}\psi|^2_{\tilde{g}_j}+|\psi|^2_{\tilde{g}_j}w)\,d\mu_{\tilde{g}_j}<\infty,
$$
where $\hat{\nabla}\psi$ is understood in the distributional sense, cf. \cite[Proposition 4.5]{CH03}.

Let $S$ be a spinor representation for $Cl_n$. For any constant unit vector $\mathbf{u}\in S$, one can associate a Killing spinor $\varphi_{\mathbf{u}}$ relative to the hyperbolic metric on $\mathbb{H}^n$. By truncation, we get a spinor $\phi_{\mathbf{u}}\in C^\infty(\tilde{M}_j,\tilde{\slashed{S}}_j)$, which is supported in the AH end $E_0$. 

\begin{lemma}\label{lem:Lichnerowicz}
    For any $\psi\in H^1_{w}$, the spinor $\Phi=\phi_{\mathbf{u}}+\psi$ satisfies
    \begin{equation}\label{eq:Lichnerowicz-2}
        \int_{\tilde{M}_j}(|\hat{\nabla}\Phi|^2_{\tilde{g}_j}-|\hat{D}\Phi|^2_{\tilde{g}_j})\,d\mu_{\tilde{g}_j}\le \lim_{\rho'\to 0}\int_{\rho=\rho'}B_{\tilde\nu}(\phi_{\mathbf{u}})\,d\sigma_{\tilde{g}_j},
    \end{equation}
    if $\psi$ satisfies the boundary condition $\psi=\sqrt{-1}c(\tilde\nu)\psi$ on $\partial\tilde{M}_j$. The limit on the right-hand side exists and equals 
    $$
        \frac{1}{4}\int_{S^{n-1}} \operatorname{tr}_{\gamma_{\mathrm{std}}}(\tilde{h}_j)(1+\sqrt{-1}\langle c(x)\mathbf{u},\mathbf{u}\rangle_{\tilde{g}_j})\, d \mu_{\gamma_{\mathrm{std}}},
    $$
    where $\tilde{h}_j$ is the mass aspect of $\tilde{g}_j$.
\end{lemma}
\begin{proof}
    First, we note that if $\psi\in C^\infty_c(\tilde{M}_j,\tilde{\slashed{S}}_j)$, then \eqref{eq:Lichnerowicz-2} follows from \eqref{eq:Lichnerowicz} and the fact that 
    $$
        \int_{\partial\tilde{M}_j}B_{\tilde\nu}(\psi)\,d\sigma_{\tilde{g}_j}\le 0,
    $$
    cf. \cite[Equation (4.28)]{CH03}, since the boundary is mean convex. 

    Next, assume that $\psi_i\to\psi$ in the $H^1_{w}$-norm, where $\psi_i$ is compactly supported satisfying the boundary condition above. Then $\psi_i\to\psi$ in the $\dot{H}^1$-norm as well. Denote $\Phi_i=\phi_{\mathbf{u}}+\psi_i$. Note that
    \begin{align*}
        (|\hat{\nabla}\Phi_i|^2_{\tilde{g}_j}-|\hat{D}\Phi_i|^2_{\tilde{g}_j})-(|\hat{\nabla}{\Phi}|^2_{\tilde{g}_j}-|\hat{D}{\Phi}|^2_{\tilde{g}_j})&=\begin{aligned}[t]
            &\langle \hat{\nabla}(\psi_i-\psi),\hat{\nabla}\Phi_i\rangle_{\tilde{g}_j}+\langle \hat{\nabla}\Phi,\hat{\nabla}(\psi_i-\psi)\rangle_{\tilde{g}_j}\\
            &-\langle \hat{D}(\psi_i-\psi),\hat{D}\Phi_i\rangle_{\tilde{g}_j}-\langle \hat{D}\Phi,\hat{D}(\psi_i-\psi)\rangle_{\tilde{g}_j}
        \end{aligned}\\
        &\le C|\hat{\nabla}(\psi_i-\psi)|_{\tilde{g}_j}(|\hat{\nabla}\Phi_i|_{\tilde{g}_j}+|\hat{\nabla}\Phi|_{\tilde{g}_j})
    \end{align*}
    where $C$ is a constant depending only on $n$, since $|\hat{D}\psi|_{\tilde{g}_j}^2\le n|\hat{\nabla}\psi|_{\tilde{g}_j}^2$. Thus, 
    $$
        \lim_{i\to\infty}\int_{{\tilde{M}_j}}(|\hat{\nabla}\Phi_i|^2_{\tilde{g}_j}-|\hat{D}\Phi_i|^2_{\tilde{g}_j})\,d\mu_{\tilde{g}_j}=\int_{{ \tilde{M}_j}}(|\hat{\nabla}\Phi|^2_{\tilde{g}_j}-|\hat{D}\Phi|^2_{\tilde{g}_j})\,d\mu_{\tilde{g}_j}.
    $$
    Hence, \eqref{eq:Lichnerowicz-2} holds for $\psi\in H^1_{w}$ as well. 

    Finally, the limit on the right-hand side of \eqref{eq:Lichnerowicz-2} exists and is computed, for example, in \cite[Theorem 2.4]{Wang2001}.
\end{proof}

By solving an elliptic boundary value problem, we can find $\psi\in H^1_{w}$ (cf. \cite{CH03}) such that $\Phi_{\mathbf{u}}=\phi_{\mathbf{u}}+\psi$ satisfies 
\begin{equation}\label{eq:Dirac}
    \begin{cases}
        \hat{D}\Phi_{\mathbf{u}}=0, & \text{ on }\tilde{M}_j,\\
        \Phi_{\mathbf{u}}=\sqrt{-1}c(\tilde\nu)\Phi_{\mathbf{u}}, &\text{ on }\partial\tilde{M}_j.
    \end{cases}
\end{equation}
Later we will write
$$
    \Phi_{\mathbf{u}}^j=\phi_{\mathbf{u}}^j+\psi^j
$$
to indicate the dependence on $j$. By Lemma \ref{lem:Lichnerowicz}, we have
\begin{equation}\label{eq:Phi^j estimate}
    \|\Phi_{\mathbf{u}}^j\|_{\dot{H}^1}^2=\int_{\tilde M_j}|\hat{\nabla}\Phi_{\mathbf{u}}^j|^2_{\tilde{g}_j}\,d\mu_{\tilde{g}_j}\le\frac{1}{4}\int_{S^{n-1}} \operatorname{tr}_{\gamma_{\mathrm{std}}}(\tilde{h}_j)(1+\sqrt{-1}\langle c(x)\mathbf{u},\mathbf{u}\rangle_{\tilde{g}_j})\,d \mu_{\gamma_{\mathrm{std}}}.
\end{equation}

\textbf{Step 2: Transplant $\Phi_{\mathbf{u}}^j$'s to spinors on the same bundle.} Recall that $\tilde M_j$ is conformal to $M\setminus\bar{V}_j$ with the conformal factor $u_j^{\frac{4}{n-2}}$, and that $u_j\to u_\infty=1$. We can identify the spinor bundle $\tilde{\slashed{S}}_j$ with the spinor bundle $\tilde{\slashed{S}}$ on $(M\setminus\bar{V}_j,g)$ via an isometry: 
$$
    \iota_j:\tilde{\slashed{S}}\to \tilde{\slashed{S}}_j,
$$
and the spinor connections are related by 
\begin{equation}
    \nabla^{\tilde{g}_j}_X(\iota_j\varphi)=\iota_j(\nabla_X^{g}\varphi)+\frac{1}{2(n-2)}\left[c\left(\frac{\nabla^{g} u_j}{u_j}\right),c(X)\right](\iota_j\varphi),
\end{equation}
for any tangent vector $X$ and any $\varphi\in C^\infty(M\setminus\bar{V}_j,\tilde{\slashed{S}})$, cf. \cite[Lemma 5.27]{LM89}. Together with 
$$
    \iota_j(c(X)\varphi)=u_j^{-\frac{2}{n-2}}c(X)\iota_j(\varphi),
$$
we have 
\begin{equation}
    \hat\nabla^{\tilde{g}_j}_X(\iota_j\varphi)=\iota_j(\hat\nabla_X^{g}\varphi)+\tfrac{1}{2(n-2)}\left[c\left(\tfrac{\nabla^{g} u_j}{u_j}\right),c(X)\right](\iota_j\varphi)+\tfrac{\sqrt{-1}}{2}(1-u_j^{-\frac{2}{n-2}})c(X)\iota_j(\varphi).
\end{equation}

\begin{lemma}\label{lem:spinor convergence}
    A subsequence of $\{\iota_j^{-1}(\Phi_{\mathbf{u}}^j)\}$ converges to a spinor $\Phi_{\mathbf{u}}^\infty$ smoothly on every compact subset of $M\setminus\mathcal{S}$. Moreover, 
    \begin{equation}\label{eq:Phi^infty estimate}
        \int_{M\setminus\mathcal{S}}|\hat{\nabla}^g\Phi_{\mathbf{u}}^\infty|^2_{g}\,d\mu_{g}\le\frac{1}{4}\int_{S^{n-1}} \operatorname{tr}_{\gamma_{\mathrm{std}}}(h)(1+\sqrt{-1}\langle c(x)\mathbf{u},\mathbf{u}\rangle_{g})\,d\mu_{\gamma_{\mathrm{std}}}.
    \end{equation}
\end{lemma}

\begin{proof}
    Let $\phi_{\mathbf{u}}\in C^\infty(M,\slashed{S})$ be a spinor which coincides with the pullback of the Killing spinor $\varphi_{\mathbf{u}}$ on $\mathbb{H}^n$. We claim that
    $$
        \psi_j:=\iota_j^{-1}(\Phi_{\mathbf{u}}^j)-\phi_{\mathbf{u}}
    $$
    is bounded in $H^1_w(\Omega)$, where $\Omega\subset M\setminus\mathcal{S}$ is any bounded domain: 
    $$
        \int_{\Omega}(|\hat\nabla^g\psi_j|_g^2+|\psi_j|_g^2w)\,d\mu_{\tilde{g}}\le C(\Omega).
    $$
    If this is proved, then passing to a subsequence, we have $\psi_j\to\psi_\infty$ in $H^1_w$, and thus 
    $$
        \iota^{-1}_j(\Phi_{\mathbf{u}}^j)\to\Phi_{\mathbf{u}}^\infty=\phi_{\mathbf{u}}+\psi_\infty,
    $$
    smoothly on every compact subset of $M\setminus\mathcal{S}$, and \eqref{eq:Phi^infty estimate} follows from \eqref{eq:Phi^j estimate} immediately. 

    Recall that we have 
    $$
        \Phi_{\mathbf{u}}^j=\phi_{\mathbf{u}}^j+\psi^j.
    $$
    On a fixed $\Omega$, since $g_j\to g$ uniformly, we can choose $\phi_{\mathbf{u}}^j$ so that 
    $$
        \|\iota^{-1}(\phi_{\mathbf{u}}^j)-\phi_{\mathbf{u}}\|_{H^1_w(\Omega)}\le C(\Omega),
    $$
    for some constant $C(\Omega)>0$. Thus, it suffices to show that $\iota_j^{-1}(\psi^j)$ is bounded in $H^1_w(\Omega)$. 

    Let $\{e_i\}$ be a local orthonormal frame on $M\setminus\mathcal{S}$ with respect to $g$. Then we have
    \begin{align*}
        |\hat{\nabla}^{g}(\iota_j^{-1}(\psi^j))|_{g}^2&=\sum_{i=1}^n |\hat{\nabla}^{g}_{e_i}(\iota_j^{-1}(\psi^j))|_{g}^2\\
        &=\begin{aligned}[t]
            \sum_{i=1}^n\left|\hat{\nabla}^{\tilde{g}_j}_{e_i}(\psi^j)-\frac{1}{2(n-2)}\left[c\left(\frac{\nabla^{g} u_j}{u_j}\right),c(e_i)\right](\psi^j)\right.\\
            \left.-\frac{\sqrt{-1}}{2}(1-u_j^{-\frac{2}{n-2}})c(e_i)\psi^j\right|_{\tilde{g}_j}^2
        \end{aligned}\\
        &=\begin{aligned}[t]
            \sum_{i=1}^n\left|u_j^{\frac{2}{n-2}}\hat{\nabla}^{\tilde{g}_j}_{{u_j}^{-\frac{2}{n-2}}e_i}(\psi^j)-\frac{1}{2(n-2)}\left[c\left(\frac{\nabla^{g} u_j}{u_j}\right),c(e_i)\right](\psi^j)\right.\\
            \left.-\frac{\sqrt{-1}}{2}(1-u_j^{-\frac{2}{n-2}})c(e_i)\psi^j\right|_{\tilde{g}_j}^2.
        \end{aligned}
    \end{align*}
    Fix any $\epsilon>0$. Since $u_j\to 1$ smoothly on $\Omega$, there exists $j_0$ such that for any $j\ge j_0$, 
    $$
        \left||\hat{\nabla}^{g}(\iota_j^{-1}(\psi^j))|_{g}^2-|\hat{\nabla}^{\tilde{g}_j}(\psi^j)|_{\tilde{g}_j}^2\right|\le \epsilon(|\Phi_{\mathbf{u}}^j|_{\tilde{g}_j}^2+|\hat{\nabla}^{\tilde{g}_j}(\Phi_{\mathbf{u}}^j)|_{\tilde{g}_j}^2).
    $$
    On the other hand, note that 
    $$
        d\mu_{g}=u_j^{-\frac{2n}{n-2}}\,d\mu_{\tilde{g}_j},
    $$
    Thus, by making $j_0$ large enough, for any $j\ge j_0$, we have
    $$
        \int_{\Omega}|\hat{\nabla}^{g}(\iota_j^{-1}(\psi^j))|^2_{g}\,d\mu_{g}\le (1+\epsilon)^2\int_{\Omega}|\hat{\nabla}^{\tilde{g}_j}(\psi^j)|^2_{\tilde{g}_j}\,d\mu_{\tilde{g}_j}+\epsilon(1+\epsilon)\int_{\Omega}|\psi^j|_{\tilde{g}_j}^2\,d\mu_{\tilde{g}_j}.
    $$
    Denote the modified Dirac operator on $\tilde{M}_j$ by $\hat{D}_j$. Then 
    $$
        \hat{D}_j:H^1(\Omega)\to L^2(\Omega)
    $$
    is bounded (for a fixed compact $\Omega$), and $\hat{D}\phi^j_{\mathbf{u}}$ is also bounded in $L^2(\Omega)$. Hence, 
    $$
        \|\psi^j\|_{H^1(\Omega)}\le C(\Omega)
    $$
    for some constant $C(\Omega)>0$. 
    
    We conclude that, by letting $\epsilon\to 0$, there exists a constant $C(\Omega)>0$ such that 
    $$
        \int_{\Omega}|\hat{\nabla}^{g}(\iota_j^{-1}(\psi^j))|^2_{g}\,d\mu_{g}\le C(\Omega).
    $$
    Since $\iota_j$ is an isometry, a similar argument applies to the $C^0$-part of the $H^1_w(\Omega)$-norm, and thus we have 
    $$
        \|\iota_j^{-1}(\psi^j))\|_{H^1_w(\Omega)}\le C(\Omega),
    $$
    as required. 
\end{proof}

\textbf{Step 3: Conclude the existence of a Killing spinor.} Since we have assumed that the inequality in \eqref{eq:mass nonnegative} is an equality, i.e., 
$$
    \int_{S^{n-1}} \operatorname{tr}_{\gamma_{\mathrm{std}}}(h)\,d \mu_{\gamma_{\mathrm{std}}} = \left|\int_{S^{n-1}} \operatorname{tr}_{\gamma_{\mathrm{std}}}(h) x\, d \mu_{\gamma_{\mathrm{std}}}\right|,
$$
there exists a unit vector $\mathbf{u}\in S$ such that 
$$
    \int_{S^{n-1}} \operatorname{tr}_{\gamma_{\mathrm{std}}}(h)(1+\sqrt{-1}\langle c(x)\mathbf{u},\mathbf{u}\rangle_{g}) \, d\mu_{\gamma_{\mathrm{std}}}=0.
$$
Thus, we conclude that 
\begin{proposition}\label{prop:Ricci rigidity spin}
    If the inequality in \eqref{eq:mass nonnegative} is an equality and if $\mathcal{R}$ is spin, then there exists a unit vector $\mathbf{u}\in S$ such that $\hat{\nabla}\Phi_{\mathbf{u}}^\infty=0$. In particular, $\Ric_g\equiv-(n-1)g$.
\end{proposition}
\begin{proof}
    In this case, from \eqref{eq:Phi^infty estimate}, we obtain that 
    $$
        \hat{\nabla}^{g}(\Phi_{\mathbf{u}}^\infty)\equiv 0,
    $$
    i.e., $\Phi_{\mathbf{u}}$ is a Killing spinor on $\mathcal{R}$. It remains to show that $\Phi_{\mathbf{u}}\ne 0$, which immediately implies that $g$ is Einstein, and thus $\Ric_g\equiv -(n-1)g$. 

    To see this, recall that, as in the proof of Lemma \ref{lem:spinor convergence}, we have
    $$
        \psi_j:=\iota_j^{-1}(\Phi_{\mathbf{u}}^j)-\phi_{\mathbf{u}}\to\psi_\infty
    $$
    by passing to a subsequence. If the limit $\Phi^\infty_{\mathbf{u}}=0$, then we have $\phi_{\mathbf{u}}=-\psi_\infty$ and in particular, it is in $L^2$. On the other hand, one can use the explicit formula
    $$
        \phi_{\mathbf{u}}(x)=\left(\frac{2}{1-|x|^2}\right)^{1/2}(\mathbf{u}-\sqrt{-1}x\cdot\mathbf{u}),
    $$
    where $x\in B(0,1)$. Thus, 
    $$
        |\phi_{\mathbf{u}}(x)|=\left(\frac{2(1+|x|^2)}{1-|x|^2}\right)^{1/2}\notin L^2.
    $$
    The contradiction implies that $\Phi^\infty_{\mathbf{u}}\ne 0$ as required.
\end{proof}

\subsection{From Einstein to hyperbolic space}\label{sec:hyperbolic}

By Anderson's work \cite{anderson03}, we can show that if the singular set $\mathcal{S}=\varnothing$, then an Einstein manifold with an AH end must be hyperbolic. 

\begin{proposition}\label{prop:diameter bound}
    Let $(M,d,\mu)$ be a metric measure space as in the introduction, with an AH end $E_0$. Let $r$ be a geodesic defining function of $\partial X$, i.e., in a collar neighborhood $U$ of the conformal infinity $\partial X$, 
    $$
        g=r^{-2}\bar g, \quad \bar g=dr^2+\gamma_r',
    $$
    where $\gamma_r'$ is a family of smooth metrics on $\partial X$. Assume that $\mathcal{S}=\varnothing$ and $g$ is Einstein. Then any minimizing $\bar g$-geodesic $c$ starting from $\partial X$ has length 
    $$
        L_{\bar{g}}(c)\le D=2.
    $$
\end{proposition}
By a $\bar g$-geodesic, we mean a smooth curve $c$ that is geodesic with respect to $\bar g$. 
\begin{proof}
    This is \cite[Proposition 5.1]{anderson03}. We only sketch their argument here for reference, and the constant $D$ should be computed in a general dimension $n\ge 3$. 
    
    Fix a small $r_0>0$ so that $U$ contains all points $p$ with $r(p)<2r_0$. Let $S_0=\{p\in M:\dist_{\bar g}(x,\partial X)=r_0\}$. Since $r>0$, we can write the defining function as 
    $$
        r=2e^{-s},
    $$
    where $s$ is the signed distance to $S_0$ with respect to $g$, near $\partial X$, shifted by a constant: 
    $$
        s(p)=\operatorname{sgn.dist}_g(p,S_0)+\ln\frac{2}{r_0}, \quad\forall p\in U\setminus\partial X.
    $$
    On $U\setminus\partial X$, we can write 
    $$
        g=ds^2+\gamma_s'', \quad \gamma_s''=r^{-2}\gamma_r'.
    $$

    A key observation is that, if $c$ is a minimal $\bar g$-geodesic starting from $\partial X$, then there is a neighborhood $V$ of the image of $c$ (except the endpoint) that is disjoint from the cut locus of $\partial X$. Hence, on $U\cup V$, we can write 
    $$
        \bar g=dr^2+\gamma_r'
    $$
    and as a consequence, on $(U\cap V)\setminus\partial X$,  
    $$
        g=ds^2+\gamma_s''
    $$
    where $\gamma_r'$ and $\gamma_s''$ are defined on some open subset of $\partial X$. 

    Thus, the calculation in \cite[Proposition 1.4]{anderson03} holds locally, i.e., 
    $$
        \bar R'=2(n-1)r^{-1}|\operatorname{Hess}_{\bar g}r|_{\bar g}^2,
    $$
    where $\bar R$ is the scalar curvature of $\bar g$, and the derivative is taken along the geodesic $c$. By Cauchy-Schwarz, 
    $$
        |\operatorname{Hess}_{\bar g}r|_{\bar g}^2\ge \frac{|\Delta_{\bar g}r|^2}{n-1}.
    $$
    The conformal change formula for the scalar curvature implies that 
    $$
        \bar R=-2(n-1)r^{-1}\Delta_{\bar g}r.
    $$
    Hence, 
    $$
        \bar R'\ge \frac{\bar R^2}{2(n-1)^2}r.
    $$
    When $r=0$, we have $\bar R=(n-1)^2$, which can be computed by Gauss-Codazzi. Hence, integrating the differential inequality gives $L_{\bar g}(c)\le 2$. 
\end{proof}

\begin{corollary}
    If $\mathcal{S}=\varnothing$, $\kappa\in L^2$, and $\EADM(g)=|\PADM(g)|$, then $(M^n,g)$ is isometric to $\mathbb{H}^n$. 
\end{corollary}
\begin{proof}
    We have seen that $M$ is Einstein in Section \ref{sec:Ricci rigid smooth}. If $E_0$ is not the only end of $M$, then we can find a $\bar g$-geodesic starting from $\partial X$ to some other end $E_i$, which has infinite length. This contradicts Proposition \ref{prop:diameter bound}. Hence, $E_0$ is the only end. It is known that if $(M^n,g)$ has a single AH end, then $\EADM(g)=|\PADM(g)|$ implies that $(M^n,g)$ is isometric to $\mathbb{H}^n$. See, for example, \cite{Huang2020}.
\end{proof}
\section{Appendix}

\subsection{Ricci tensor of an asymptotically hyperbolic metric}

Let $g$ be an asymptotically hyperbolic metric with defining function $\rho$, as in Definition \ref{def:AH mfd}, and let 
\[
    \kappa=\Ric(g)+(n-1)g.
\]
Choose local coordinates $x^0,x^1,\ldots,x^{n-1}$, where $x^0=\rho$ and $x^1,\ldots,x^{n-1}$ are local coordinates on the conformal infinity $\partial X$. A standard calculation gives the following:

\begin{lemma}\label{lem:kappa expansion}
    Let $\gamma_{\mathrm{std}}=\gamma_{ij}\,dx^idx^j$ and $h=h_{ij}\,dx^idx^j$. The Christoffel symbols of $g$ are given by 
    \[
        \Gamma_{ij}^k=O(\rho^n)+\begin{cases}
            \Gamma(\gamma_{\mathrm{std}})_{ij}^k, & i,j,k>0; \\ 
            (\coth\rho)\gamma_{ij}-\frac{n-2}{2n}\rho^{n-1}h_{ij}, &i,j>0,\ k=0; \\ 
            -(\coth\rho)\delta_j^k+\frac{1}{2}\rho^{n-1}h_j^k, & i=0,\ j,k>0; \\ 
            -(\coth\rho)\delta_i^k+\frac{1}{2}\rho^{n-1}h_i^k, & j=0,\ i,k>0; \\ 
            -\coth\rho, & i=j=k=0; \\
            0, & \text{otherwise}.
        \end{cases}
    \]
    The remainder belongs to $C^\infty_n(E_0)$, and $h_i^k=\gamma^{jk}h_{ij}$. Thus, 
    \[
        \kappa_{ij}=O(\rho^{n-1})+\begin{cases}
            \frac{1}{2}\rho^{n-2}(\gamma_{ij}\operatorname{tr}_{\gamma_{\mathrm{std}}}h-h_{ij}), & i,j>0; \\ 
            -\frac{n-2}{2}\rho^{n-2}\operatorname{tr}_{\gamma_{\mathrm{std}}}h, & i=j=0; \\
            0, & \text{otherwise}.
        \end{cases}
    \]
    The remainder belongs to $C^\infty_{n-1}(E_0)$. 
\end{lemma}

As a consequence, we can compute the norm of $\kappa$. 

\begin{corollary}
    For any integer $k\ge 0$, $|\nabla^k\kappa|_g=O(\rho^n)$.
\end{corollary}

\subsection{Variation formulas for the scalar curvature}

Let $(M^n,g)$ be a Riemannian manifold, and let $h$ be a symmetric $(0,2)$-tensor. Consider the one-parameter family of metrics $g(t)=g_t:=g+th$. Suppose that for any $|t|<t_0$, we have 
\[
    \frac{1}{2}g\le g_t\le 2g.
\]

\begin{lemma}\label{lem:1st variation of Sc}
    If $|\Ric(g)|_g\le K$, then 
    \[
        |R(g_t)-R(g)|\le C_n|t|(K|h|_g+|\nabla h|_g^2+|\nabla^2 h|_g),
    \]
    for some dimensional constant $C_n>0$. 
\end{lemma}

\begin{proof}
    Recall that
    \[
        |g|_g=|g^{-1}|_g=\sqrt{n}.
    \]
    Since $g_t=g+th$, and since 
    \[
        \frac{1}{2}g^{-1}\le g_t^{-1}\le 2g^{-1},
    \] 
    we have 
    \[
        g_t^{-1}-g^{-1}=-tg_t^{-1}hg^{-1} \implies |g_t^{-1}-g^{-1}|\le C_n|t||h|_g.
    \]
    The difference of the Christoffel symbols is given by
    \[
        A^k_{ij}=A(t)^k_{ij}:=\Gamma(g_t)^k_{ij}-\Gamma(g)^k_{ij}=\frac{t}{2}g(t)^{kl}(\nabla_i h_{jl}+\nabla_j h_{il}-\nabla_l h_{ij}),
    \]
    where $\nabla$ is taken with respect to $g$, and thus
    \[
        |A|_g\le C_n|t||\nabla h|_g.
    \]

    From 
    \[
        \nabla g_t^{-1}=-g_t^{-1}(\nabla g_t)g_t^{-1}=-tg_t^{-1}(\nabla h)g_t^{-1},
    \]
    we obtain
    \[
        |\nabla g_t^{-1}|_g\le C_n|t||\nabla h|_g.
    \]
    It follows that
    \[
        |\nabla A|_g\le C_n|t|(|\nabla^2 h|_g+|\nabla h|_g^2).
    \]

    By the standard formula
    \[
        \Ric(g_t)_{ij}-\Ric(g)_{ij}=\nabla_k A^k_{ij}-\nabla_j A^k_{ki}+A^k_{kl}A^l_{ij}-A^k_{jl}A^l_{ik},
    \]
    we have
    \[
        |\Ric(g_t)-\Ric(g)|_g\le C_n|t|(|\nabla^2 h|_g+|\nabla h|_g^2).
    \]

    Lastly, for the scalar curvature, we have
    \[
        R(g_t)-R(g)=(g(t)^{ij}-g^{ij})\Ric(g)_{ij}+g(t)^{ij}(\Ric(g_t)_{ij}-\Ric(g)_{ij}).
    \]
    The first term on the right-hand side is bounded by 
    \[
        C_n|t|K|h|_g,
    \]
    and the second term is bounded by 
    \[
        C_n|t|(|\nabla^2 h|_g+|\nabla h|_g^2).
    \]
    This proves the lemma.
\end{proof}

The first variational formula for the scalar curvature says
\[
    \frac{d}{dt}\Big|_{t=0}R(g_t)=-\Delta_g(\operatorname{tr}_g h)+\operatorname{div}_g(\operatorname{div}_g h)-\langle h,\operatorname{Ric}(g)\rangle_g.
\]
We next estimate the error of the linearization.
\begin{lemma}\label{lem:2nd variation of Sc}
    If $|\Ric(g)|_g\le K$, then
    \[
        \left|R(g_t)-R(g)-t\frac{d}{dt}\Big|_{t=0}R(g_t)\right|\le C_nt^2(K|h|_g^2+|\nabla h|_g^2+|h|_g|\nabla^2 h|_g^2),
    \]
    for some dimensional constant $C_n>0$.
\end{lemma}

\begin{proof}
    The proof is similar to that of the previous lemma, but we need to expand the difference of the Christoffel symbols to second order in $t$.

    Let $h^{\sharp}=g^{-1}hg^{-1}$ be the $(2,0)$-tensor corresponding to $h$ via the metric $g$, or equivalently, 
    \[
        \frac{d}{dt}\Big|_{t=0}g_t^{-1}=-h^{\sharp}.
    \]
    A direct calculation shows that
    \[
        g_t^{-1}=g^{-1}-th^{\sharp}+t^2g_t^{-1}hg^{-1}hg^{-1},
    \]
    and thus,
    \[
        |g_t^{-1}-g^{-1}+th^{\sharp}|_g\le C_nt^2|h|_g^2.
    \]

    Let $A=A(t):=\Gamma(g_t)-\Gamma(g)$ be as in the previous lemma. Then we have 
    \[
        A(t)=tA_1+O(t^2), \quad A_1=\frac{d}{dt}\Big|_{t=0}A(t), \quad (A_1)^k_{ij}=\frac{1}{2}g^{kl}(\nabla_i h_{jl}+\nabla_j h_{il}-\nabla_l h_{ij}),
    \]
    and a direct calculation shows that
    \[
    A(t)-tA_1=t^2h^{\sharp}\ast\nabla h,
    \]
    where $\ast$ denotes a contraction with respect to $g$. Hence, we have
    \[
        |A(t)-tA_1|_g\le C_nt^2|h|_g|\nabla h|_g.
    \]
    
    For $\nabla A$, we have
    \[
        \nabla A(t)-t\nabla A_1=t^2(\nabla h^{\sharp}\ast\nabla h+h^{\sharp}\ast\nabla^2 h),
    \]
    and thus
    \[
        |\nabla A(t)-t\nabla A_1|_g\le C_nt^2(|h|_g|\nabla^2 h|_g+|\nabla h|_g^2).
    \]

    Let $D\Ric(g)$ be the linearization of the Ricci tensor at $g$. Then we continue the calculation as in the previous lemma to obtain
    \[
        \Ric(g_t)-\Ric(g)-tD\Ric(g)(h)=t^2(\nabla h^{\sharp}\ast\nabla h+h^{\sharp}\ast\nabla^2 h+\nabla h\ast \nabla h),
    \]
    and thus
    \[
        |\Ric(g_t)-\Ric(g)-tD\Ric(g)(h)|_g\le C_nt^2(|h|_g|\nabla^2 h|_g+|\nabla h|_g^2).
    \]

    Lastly, for the scalar curvature, let $DR(g)$ be the linearization of the scalar curvature at $g$. Then we have
    \begin{align*}
        R(g_t)-R(g)-tDR(g)(h)=(g_t^{-1}-g^{-1}+th^{\sharp})\ast \Ric(g)+(g_t^{-1}-g^{-1})\ast(\Ric(g_t)-\Ric(g))\\
        +g^{-1}\ast(\Ric(g_t)-\Ric(g)-tD\Ric(g)(h)),
    \end{align*}
    and thus
    \[
        |R(g_t)-R(g)-tDR(g)(h)|\le C_nt^2(K|h|_g^2+|\nabla h|_g^2+|h|_g|\nabla^2 h|_g^2),
    \]
    as required.
\end{proof}

Under the same assumptions as above, we derive a similar estimate for the Laplacian. 

\begin{lemma}\label{lem:1st variation of Laplacian}
    For any $f\in C^2(M)$, 
    \[
        |(\Delta_{g_t}-\Delta_g)f|\le C_n|t|\|f\|_{C^2(g)}(|h|_g+|\nabla h|_g),
    \]
    for some dimensional constant $C_n>0$. 
\end{lemma}
\begin{proof}
    From a standard calculation, we have 
    \[
        (\Delta_{g_t}-\Delta_g)f=(g(t)^{ij}-g^{ij})\partial_i\partial_jf-(g(t)^{ij}\Gamma(g_t)^k_{ij}-g^{ij}\Gamma(g)^k_{ij})\partial_kf.
    \]
    As calculated in Lemma \ref{lem:1st variation of Sc}, 
    \[
        |g(t)^{ij}-g^{ij}|\le C_n|t||h|_g, \quad |\Gamma(g_t)^k_{ij}-\Gamma(g)^k_{ij}|\le C_n|t||\nabla h|_g.
    \]
    The desired inequality follows easily. 
\end{proof}


\bibliographystyle{alpha}

\bibliography{Positive}

@article{Jiang2014,
  title = {Cheeger-Harmonic Functions in Metric Measure Spaces Revisited},
  author = {Jiang, Renjin},
  year = 2014,
  month = feb,
  journal = {Journal of Functional Analysis},
  volume = {266},
  number = {3},
  pages = {1373--1394},
  issn = {00221236},
  doi = {10.1016/j.jfa.2013.11.022},
}

@misc{khuri2026,
  title = {Riemannian {{Positive Mass Theorem}} in {{All Dimensions}} in the {{Presence}} of {{Low-Codimension Singularities}}},
  author = {Khuri, Marcus and Wang, Jian and Wang, Jinmin},
  year = 2026,
  month = jul,
  number = {arXiv:2606.23529},
  eprint = {2606.23529},
  primaryclass = {math.DG},
  publisher = {arXiv},
  doi = {10.48550/arXiv.2606.23529},
}

@misc{BZ2026,
  title = {Positive {{Scalar Curvature Obstructions}} via {{Singular Dimension Descent}}},
  author = {Bi, Yuchen and Zhu, Jintian},
  year = 2026,
  month = jul,
  publisher = {arXiv},
  doi = {10.48550/ARXIV.2606.20528},
}

@article{bartnik2005,
  title = {Boundary Value Problems for {{Dirac-type}} Equations},
  author = {Bartnik, Robert A. and Chru{\'s}ciel, Piotr T.},
  year = 2005,
  month = mar,
  journal = {Journal f\"ur die reine und angewandte Mathematik (Crelles Journal)},
  volume = {2005},
  number = {579},
  pages = {13--73},
  issn = {0075-4102, 1435-5345},
  doi = {10.1515/crll.2005.2005.579.13},
}

@article{Miao02,
  title = {Positive Mass Theorem on Manifolds Admitting Corners along a Hypersurface},
  author = {Miao, Pengzi},
  year = 2002,
  journal = {Advances in Theoretical and Mathematical Physics},
  volume = {6},
  number = {6},
  pages = {1163--1182},
  issn = {10950761, 10950753},
  doi = {10.4310/ATMP.2002.v6.n6.a4},
}

@article {Huang2020,
    AUTHOR = {Huang, Lan-Hsuan and Jang, Hyun Chul and Martin, Daniel},
     TITLE = {Mass rigidity for hyperbolic manifolds},
   JOURNAL = {Comm. Math. Phys.},
  FJOURNAL = {Communications in Mathematical Physics},
    VOLUME = {376},
      YEAR = {2020},
    NUMBER = {3},
     PAGES = {2329--2349},
      ISSN = {0010-3616,1432-0916},
   MRCLASS = {53C20 (53C24)},
  MRNUMBER = {4104551},
MRREVIEWER = {Theodoros\ Vlachos},
       DOI = {10.1007/s00220-019-03623-0},
       URL = {https://doi.org/10.1007/s00220-019-03623-0},
}

@article {Hir2025,
    AUTHOR = {Hirsch, Sven and Jang, Hyun Chul and Zhang, Yiyue},
     TITLE = {Rigidity of asymptotically hyperboloidal initial data sets
              with vanishing mass},
   JOURNAL = {Comm. Math. Phys.},
  FJOURNAL = {Communications in Mathematical Physics},
    VOLUME = {406},
      YEAR = {2025},
    NUMBER = {12},
     PAGES = {Paper No. 307, 26},
      ISSN = {0010-3616,1432-0916},
   MRCLASS = {53C80 (53C20 53C24 83C05 83C30)},
  MRNUMBER = {4983798},
       DOI = {10.1007/s00220-025-05485-1},
       URL = {https://doi.org/10.1007/s00220-025-05485-1},
}

@misc{hirsch2026,
      title={The Hyperboloidal and Spacetime Positive Mass Theorem in All Dimensions}, 
      author={Sven Hirsch and Marcus Khuri and Martin Lesourd and Yiyue Zhang},
      year={2026},
      eprint={2604.24746},
      archivePrefix={arXiv},
      primaryClass={math.DG},
      url={https://arxiv.org/abs/2604.24746}, 
}

@misc{tsang2026,
      title={Positive mass theorem for initial data sets with arbitrary ends}, 
      author={Tin-Yau Tsang},
      year={2026},
      eprint={2604.26978},
      archivePrefix={arXiv},
      primaryClass={math.DG},
      url={https://arxiv.org/abs/2604.26978}, 
}

@misc{BHHSZ2026,
      title={A proof for the Riemannian positive mass theorem up to dimension 19}, 
      author={Yuchen Bi and Tianze Hao and Shihang He and Yuguang Shi and Jintian Zhu},
      year={2026},
      eprint={2603.02769},
      archivePrefix={arXiv},
      primaryClass={math.DG},
      url={https://arxiv.org/abs/2603.02769}, 
}

@article {Wang2001,
    AUTHOR = {Wang, Xiaodong},
     TITLE = {The mass of asymptotically hyperbolic manifolds},
   JOURNAL = {J. Differential Geom.},
  FJOURNAL = {Journal of Differential Geometry},
    VOLUME = {57},
      YEAR = {2001},
    NUMBER = {2},
     PAGES = {273--299},
      ISSN = {0022-040X,1945-743X},
   MRCLASS = {53C20 (53C21)},
  MRNUMBER = {1879228},
MRREVIEWER = {Lars\ \AA ke\ Andersson},
       URL = {http://projecteuclid.org/euclid.jdg/1090348112},
}

@article{LM2019,
	author = {Li, Chao and Mantoulidis, Christos},
	doi = {10.1007/s00208-018-1753-1},
	fjournal = {Mathematische Annalen},
	issn = {0025-5831,1432-1807},
	journal = {Math. Ann.},
	mrclass = {53C21 (53C20)},
	mrnumber = {3961306},
	mrreviewer = {David\ J.\ Wraith},
	number = {1-2},
	pages = {99--131},
	title = {Positive scalar curvature with skeleton singularities},
	url = {https://doi.org/10.1007/s00208-018-1753-1},
	volume = {374},
	year = {2019}}

@misc{WX2024,
	archiveprefix = {arXiv},
	author = {Jinmin Wang and Zhizhang Xie},
	eprint = {2407.21312},
	primaryclass = {math.DG},
	title = {Scalar curvature rigidity of spheres with subsets removed and $L^\infty$ metrics},
	url = {https://arxiv.org/abs/2407.21312},
	year = {2024}}

@article{DaSW2024,
	author = {Xianzhe Dai and Yukai Sun and Changliang Wang},
	journal = {Science China Mathematics},
	title = {The positive mass theorem for asymptotically flat manifolds with isolated conical singularities},
	year = {2024}}

@misc{DWWW2024,
	archiveprefix = {arXiv},
	author = {Xianzhe Dai and Changliang Wang and Lihe Wang and Guofang Wei},
	eprint = {2412.09185},
	primaryclass = {math.DG},
	title = {Singular metrics with nonnegative scalar curvature and RCD},
	url = {https://arxiv.org/abs/2412.09185},
	year = {2024}}

@book {LM89,
    AUTHOR = { H. Blaine Lawson and Marie-Louise Michelsohn},
     TITLE = {Spin geometry},
    SERIES = {Princeton Mathematical Series},
    VOLUME = {38},
 PUBLISHER = {Princeton University Press, Princeton},
      YEAR = {1989},
}

@article{Kaz24,
    author = {Demetre Kazaras},
    title = {Desingularizing positive scalar curvature 4-manifolds},
    journal = {Math. Ann.},
    year = {2024},
    volume = {390},
    pages = {4951--4972}
}

@article{JSZ22,
    author = {Wenshuai Jiang and Weimin Sheng and Huaiyu Zhang},
    title = {Removable singularity of positive mass theorem with continuous metrics},
    journal = {Math. Z},
    year = {2022},
    volume = {302},
    pages = {839--874}
}

@misc {CLZ22,
      title={Singular positive mass theorem with arbitrary ends}, 
      author={Chu, Jianchun and Lee, Man Chun and Zhu, Jintian},
      year={2022},
      eprint={2210.08261},
      archivePrefix={arXiv},
      primaryClass={math.DG},
      url={https://arxiv.org/abs/2210.08261}, 
}

@book{Gro23,
    author = {Misha Gromov},
    title = {Four lectures on scalar curvature},
    publisher = {World Sci. Publ., Hackensack, NJ},
    year = {2023}
}

@misc{HSY26,
      title={Singularity removal rigidity theorems for minimal hypersurfaces in manifolds with nonnegative scalar curvature}, 
      author={Shihang He and Yuguang Shi and Haobin Yu},
      year={2026},
      eprint={2602.23705},
      archivePrefix={arXiv},
      primaryClass={math.DG},
      url={https://arxiv.org/abs/2602.23705}, 
}

@article{BQ08,
  title = {A Positive Mass Theorem on Asymptotically Hyperbolic Manifolds with Corners along a Hypersurface},
  author = {Bonini, Vincent and Qing, Jie},
  year = 2008,
  month = apr,
  journal = {Annales Henri Poincar\'e},
  volume = {9},
  number = {2},
  pages = {347--372},
  issn = {1424-0637, 1424-0661},
  doi = {10.1007/s00023-008-0358-8},
}

@article{Aviles82,
  title = {A Study of the Singularities of Solutions of a Class of Nonlinear Elliptic Partial Differential Equation},
  author = {Aviles, Patricio},
  year = 1982,
  month = jan,
  journal = {Communications in Partial Differential Equations},
  volume = {7},
  number = {5},
  pages = {609--643},
  issn = {0360-5302, 1532-4133},
  doi = {10.1080/03605308208820234},
}

@article{ACF92,
  title = {On the Regularity of Solutions to the {{Yamabe}} Equation and the Existence of Smooth Hyperboloidal Initial Data for {{Einstein}}'s Field Equations},
  author = {Andersson, Lars and Chru{\'s}ciel, Piotr T. and Friedrich, Helmut},
  year = 1992,
  month = oct,
  journal = {Communications in Mathematical Physics},
  volume = {149},
  number = {3},
  pages = {587--612},
  issn = {0010-3616, 1432-0916},
  doi = {10.1007/BF02096944},
}

@article{AMcpt88,
  title = {Complete Conformal Metrics with Negative Scalar Curvature in Compact {{Riemannian}} Manifolds},
  author = {Aviles, Patricio and McOwen, Robert C.},
  year = 1988,
  month = apr,
  journal = {Duke Mathematical Journal},
  volume = {56},
  number = {2},
  issn = {0012-7094},
  doi = {10.1215/S0012-7094-88-05616-5},
}

@article{AMnoncpt88,
  title = {Conformal Deformation to Constant Negative Scalar Curvature on Noncompact {{Riemannian}} Manifolds},
  author = {Aviles, Patricio and McOwen, Robert C.},
  year = 1988,
  month = jan,
  journal = {Journal of Differential Geometry},
  volume = {27},
  number = {2},
  issn = {0022-040X},
  doi = {10.4310/jdg/1214441781},
}

@article{Yau75,
  title = {Harmonic Functions on Complete Riemannian Manifolds},
  author = {Yau, Shing-Tung},
  year = 1975,
  month = mar,
  journal = {Communications on Pure and Applied Mathematics},
  volume = {28},
  number = {2},
  pages = {201--228},
  issn = {0010-3640, 1097-0312},
  doi = {10.1002/cpa.3160280203},
}

@article{CY75,
  title = {Differential Equations on Riemannian Manifolds and Their Geometric Applications},
  author = {Cheng, S. Y. and Yau, S. T.},
  year = 1975,
  month = may,
  journal = {Communications on Pure and Applied Mathematics},
  volume = {28},
  number = {3},
  pages = {333--354},
  issn = {0010-3640, 1097-0312},
  doi = {10.1002/cpa.3160280303},
}

@misc{HSY25,
  title = {Positive Mass Theorems on Singular Spaces and Some Applications},
  author = {He, Shihang and Shi, Yuguang and Yu, Haobin},
  year = 2025,
  publisher = {arXiv},
  doi = {10.48550/ARXIV.2502.18000},
}

@article{CH03,
  title = {The Mass of Asymptotically Hyperbolic {{Riemannian}} Manifolds},
  author = {Chru{\'s}ciel, Piotr T. and Herzlich, Marc},
  year = 2003,
  month = dec,
  journal = {Pacific Journal of Mathematics},
  volume = {212},
  number = {2},
  pages = {231--264},
  issn = {0030-8730},
  doi = {10.2140/pjm.2003.212.231},
}

@book{Evans2010,
  title = {Partial Differential Equations},
  author = {Evans, Lawrence C.},
  year = 2010,
  series = {Graduate Studies in Mathematics},
  edition = {2nd ed},
  number = {v. 19},
  publisher = {American Mathematical Society},
  isbn = {978-0-8218-4974-3},
}

@article{anderson03,
  title = {Boundary Regularity, Uniqueness and Non-Uniqueness for {{AH Einstein}} Metrics on 4-Manifolds},
  author = {Anderson, Michael T.},
  year = 2003,
  month = nov,
  journal = {Advances in Mathematics},
  volume = {179},
  number = {2},
  pages = {205--249},
  issn = {00018708},
  doi = {10.1016/S0001-8708(02)00075-0},
}

@misc{BW2026,
  title = {On the Spacetime Positive Energy Theorem in Arbitrary Dimension},
  author = {Brendle, S. and Wang, Y.},
  year = 2026,
  month = apr,
  number = {arXiv:2604.18561},
  eprint = {2604.18561},
  primaryclass = {math.DG},
  publisher = {arXiv},
  doi = {10.48550/arXiv.2604.18561}
}

@article{AFLMR07,
  author = {D. Azagra and J. Ferrera and F. López-Mesas and Y. Rangel},
  title = {Smooth approximation of Lipschitz functions on Riemannian manifolds},
  journal = {Journal of Mathematical Analysis and Applications},
  volume = {326},
  number = {2},
  pages = {1370--1378},
  year = {2007},
  issn = {0022-247X},
  doi = {https://doi.org/10.1016/j.jmaa.2006.03.088},
  url = {https://www.sciencedirect.com/science/article/pii/S0022247X0600343X}
}

\end{document}